\documentclass[11pt]{amsart}

\usepackage{latexsym,amsfonts,amssymb,exscale,comment}
\usepackage{amsmath,amsthm,amscd}
\usepackage{verbatim}

\usepackage{enumerate}

\usepackage{fullpage}

\usepackage{caption}
\usepackage{subcaption}

\usepackage{url}
\usepackage[bookmarks=true,%
    colorlinks=true,%
    linkcolor=blue,%
    citecolor=blue,%
    filecolor=blue,%
    menucolor=blue,%
    urlcolor=blue,%
    breaklinks=true,
		bookmarksdepth=2]{hyperref}

\usepackage{bm}

\input xy
\usepackage[all]{xy}
\SelectTips{cm}{}

\usepackage{tikz}

\tikzset{->-/.style={decoration={
  markings,
  mark=at position #1 with {\arrow{>}}},postaction={decorate}}}

\tikzset{middlearrow/.style={
        decoration={markings,
            mark= at position 0.5 with {\arrow{#1}} ,
        },
        postaction={decorate}
    }
}

\usepackage{graphicx}
\usepackage{color}

\theoremstyle{plain}
\newtheorem{theorem}{Theorem}

\newtheorem{corollary}[theorem]{Corollary}
\newtheorem{proposition}[theorem]{Proposition}
\newtheorem{lemma}[theorem]{Lemma}

\theoremstyle{definition}
\newtheorem{example}[theorem]{Example}
\newtheorem{definition}[theorem]{Definition}
\newtheorem{conjecture}[theorem]{Conjecture}

\theoremstyle{definition}
\newtheorem{remark}[theorem]{Remark}

\numberwithin{equation}{section}
\numberwithin{theorem}{section}

\newcommand{\refequal}[1]{\xy {\ar@{=}^{#1}
(-1,0)*{};(1,0)*{}};
\endxy}

\newcommand{\Hom}{{\rm Hom}}

\newcommand{\maps}{\colon}
\renewcommand{\to}{\rightarrow}
\newcommand{\op}{{\rm op}}
\def\1{\mathbf{1}}%
\def\id{\mathrm{id}}
\def\Id{\mathrm{Id}}

\newcommand{\und}[1]{\underline{#1}}

\def\fmod{{\mathrm{-fmod}}}   
\def\pmod{{\mathrm{-pmod}}}  
\def\gmod{{\mathrm{-gmod}}}

\def\mf{\mathfrak}

\numberwithin{equation}{section}

\def\b{$\blacktriangleright$}

\let\tilde=\widetilde

\let\epsilon=\varepsilon

\usepackage{bbm}
\def\C{{\mathbb{C}}}

\def\cal#1{\mathcal{#1}}%
\def\st{\mathrm{st}}%
\def\nn{\notag}

\def\la{\langle}
\def\ra{\rangle}

\def\cal#1{\mathcal{#1}}

\newcommand\nc{\newcommand}
\nc\rnc{\renewcommand}
\nc\Kar{\operatorname{Kar}}
\nc\End{\operatorname{End}}

\newcommand{\infl}{{\rm infl}}
\newcommand{\pr}{{\rm pr}}

\nc\Sym{\operatorname{Sym}}

\allowdisplaybreaks

\DeclareMathOperator{\can}{can}

\DeclareMathOperator{\Ext}{Ext}
\DeclareMathOperator{\ext}{ext}

\DeclareMathOperator{\gldim}{gl.dim}
\DeclareMathOperator{\Grr}{Gr}

\DeclareMathOperator{\sta}{st}
\DeclareMathOperator{\supp}{supp}

\DeclareMathOperator{\var}{var}

\newcommand{\At}{\tilde{A}}
\newcommand{\Bt}{\tilde{B}}
\newcommand{\Pt}{\tilde{P}}
\newcommand{\Vt}{\tilde{V}}
\newcommand{\xx}{\mathbbm{x}}
\newcommand{\yy}{\mathbbm{y}}
\newcommand{\zz}{\mathbbm{z}}
\renewcommand{\a}{\alpha}
\renewcommand{\b}{\beta}

\newcommand{\A}{\mathcal{A}}
\newcommand{\B}{\mathcal{B}}
\newcommand{\V}{\mathcal{V}}

\newcommand{\F}{\mathbb{F}}

\newcommand{\gl}{\mathfrak{gl}}

\newcommand{\M}{\mathfrak{M}}
\newcommand{\Q}{\mathbb{Q}}
\newcommand{\R}{\mathbb{R}}

\newcommand{\Z}{\mathbb{Z}}
\newcommand{\Zc}{\mathcal{Z}}

\newcommand{\x}{\mathbf{x}}
\newcommand{\y}{\mathbf{y}}
\newcommand{\z}{\mathbf{z}}

\newcommand{\Ib}{\mathbf{I}}

\newcommand{\cV}{\mathcal{V}}
\newcommand{\cH}{\mathcal{H}}
\newcommand{\cF}{\mathcal{F}}
\newcommand{\vp}{\varphi}
\newcommand{\cP}{{\cal P}}
\newcommand{\cB}{{\cal B}}
 
 \newcommand{\ep}{\epsilon}
 \newcommand{\tPi}{\boldsymbol{\Pi}_\a^\bullet}

\newcommand{\gd}{\vee}

\newcommand{\ootimes}{
  \mathbin{
    \mathchoice
      {\buildcircleotimes{\displaystyle}}
      {\buildcircleotimes{\textstyle}}
      {\buildcircleotimes{\scriptstyle}}
      {\buildcircleotimes{\scriptscriptstyle}}
  }
}

\newcommand\buildcircleotimes[1]{%
  \begin{tikzpicture}[baseline=(X.base), inner sep=0, outer sep=0]
    \node[draw,circle] (X)  {$#1\otimes$};
  \end{tikzpicture}%
}

\title{Strands algebras and the affine highest weight property for equivariant hypertoric categories}

\begin{document}

\author{Aaron D. Lauda}
\email{lauda@usc.edu}
\address{Department of Mathematics\\ University of Southern California \\ Los Angeles, CA}
\thanks{Research was sponsored by the Army Research Office and was
accomplished under Grant Number W911NF-20-1-0075. The views and conclusions contained in this
document are those of the authors and should not be interpreted as representing the official policies, either
expressed or implied, of the Army Research Office or the U.S. Government. The U.S. Government is authorized to reproduce and distribute reprints for Government purposes notwithstanding any copyright
notation herein.  }

\author{Anthony M. Licata}
\email{anthony.licata@anu.edu.au}
\address{Mathematical Sciences Institute\\ Australian National University \\ Canberra, Australia}

\author{Andrew Manion}
\email{ajmanion@ncsu.edu}
\address{Department of Mathematics\\ North Carolina State University \\Raleigh, NC}

\begin{abstract}
We show that the equivariant hypertoric convolution algebras introduced by Braden--Licata--Proudfoot--Webster are affine quasi hereditary in the sense of Kleshchev and compute the Ext groups between standard modules.   Together with the main result of \cite{LLM}, this implies a number of new homological results about the bordered Floer algebras of Ozsv{\'a}th--Szab{\'o}, including the existence of standard modules over these algebras.  We prove that the Ext groups between standard modules are isomorphic to the homology of a variant of the Lipshitz--Ozsv{\'a}th--Thurston bordered strands dg algebras.

  \end{abstract}

\maketitle

\tableofcontents

\setcounter{tocdepth}{2}

\section{Introduction}

 \noindent \textbf{Background. }
In \cite{LLM} the authors observed a connection between hypertoric geometry and bordered Heegaard Floer theory.  In particular, it is shown there that the equivariant hypertoric convolution algebras $\Bt(\cal{V})$ introduced by Braden--Licata--Proudfoot--Webster in \cite{Gale}, for certain ``cyclic'' choices of combinatorial data $\cal{V}$, are isomorphic to algebras introduced by Ozsv{\'a}th--Szab{\'o} in their Kauffman states approach to knot Floer homology \cite{OSzNew}.   The proof of this isomorphism makes heavy use of Karp and Williams' cell decomposition~\cite{KarpWilliams} of the $m=1$ amplituhedron and the theory of total positivity.
More generally, \cite{LLM} gives a conjectural algebraic description of the  partially wrapped Fukaya categories associated to complexified hyperplane complements $\cal{X}_{\cal{V}}$ using these equivariant convolution algebras for general $\cal{V}$.
\medskip

 \noindent \textbf{Standard modules and LOT algebras.  }
This article concerns the homological properties of the algebras $\Bt(\cal{V})$ and further develops the connection between hypertoric geometry and bordered Floer theory.  We introduce and study ``standard" modules over the algebras $\Bt(\cal{V})$ for general $\cal{V}$, compute the $\Ext$ groups between them, and (for cyclic $\cal{V}$) relate these $\Ext$ groups with the homology of a variant of Lipshitz--Ozsv{\'a}th--Thurston's \cite{LOTBorderedOrig} bordered strands dg algebras $\A(\Zc^{\sta}_n)$ (see Theorem~\ref{thm:ExtStrands}). Projective resolutions of standard modules are shown to have natural Heegaard Floer theoretic interpretations via certain Heegaard diagrams; we conjecture that holomorphic disk counts in these Heegaard diagrams encode not just the projective resolutions but also an $A_{\infty}$ quasi-isomorphism from $\A(\Zc^{\sta}_n)$ to the endomorphism dg algebra of the projective resolutions. This conjecture would enable the $A_{\infty}$ structure on the sum of $\Ext$ groups between standard modules to be computed in terms of the diagrammatic algebra $\A(\Zc^{\sta}_n)$.

\medskip

 \noindent \textbf{Standard bases and higher tensor products.  }
Knot Floer homology is a categorification of the Alexander polynomial, which admits a description as a quantum invariant via the representation theory of the super algebra $U_q(\mf{gl}(1|1))$. Both the Ozsv{\'a}th--Szab{\'o} algebras and the bordered strands algebras $\A(\Zc^{\sta}_n)$ give rise to categorifications of tensor powers $V^{\otimes n}$ of the defining representation $V$ of $U_q(\mf{gl}(1|1))$ \cite{ManionDecat,ManionTrivalent}; see also \cite{EPV, TianUT}.  Various classes of modules, such as simples or indecomposable projectives, give rise to various bases for $V^{\otimes n}$.  In particular, $V^{\otimes n}$ has a canonical basis which is categorified by indecomposable projectives over the Ozsv{\'a}th--Szab{\'o} algebras \cite{ManionDecat} and a standard basis which is categorified by indecomposable projectives over $\A(\Zc^{\sta}_n)$ \cite{ManionTrivalent}.  We show in Theorem~\ref{thm:bases} that standard modules over Ozsv{\'a}th--Szab{\'o}'s algebras also categorify the standard basis of $V^{\otimes n}$, in line with the conjecture mentioned above.

The standard basis and the algebra $\A(\Zc^{\sta}_n)$ are especially relevant in connection with tensor products of higher representations and cornered Floer homology \cite{DM, DLM, ManionRouquier}; indeed, $\A(\Zc^{\sta}_n)$ arises directly as an $n$-fold higher tensor power of $\A(\Zc^{\sta}_1)$ via the ``$\mf{gl}(1|1)^+$'' tensor operation $\ootimes$ recently introduced by Rouquier and the third author in \cite{ManionRouquier}. To further develop the theory of $\ootimes$, and especially to help understand higher braidings, one approach is to use bordered Floer ideas as in \cite{ManionTrivalent}. The standard modules defined here should help connect more well-developed parts of Floer theory, such as Ozsv{\'a}th--Szab{\'o}'s computational techniques in knot Floer homology, to the study of higher tensor products.
\medskip

 \noindent \textbf{Affine quasi hereditary algebras.  }
The existence of standard modules for the algebras $\Bt(\cal{V})$ is reminiscent of the structure enjoyed by finite-dimensional quasi hereditary algebras.  However, the classical quasi hereditary theory does not directly apply to the infinite-dimensional algebras $\Bt(\cal{V})$ (indeed, this is one of the reasons why the algebras $\Bt(\cal{V})$ are not as well studied as their finite dimensional quasi hereditary cousins in \cite{Gale}).  Fortunately, Kleshchev has recently developed an infinite-dimensional analogue of quasi hereditary algebras, the so-called affine or polynomial quasi hereditary algebras~\cite{Klesh-affine} (generalizing other related work~\cite{KX,Maz-standardly}) which is perfectly suited to our setting.  In Theorem~\ref{thm:affine} we prove that for a general polarized hyperplane arrangement $\cal{V}$, the algebras $\Bt(\cal{V})$ are \emph{affine quasi hereditary} in the sense of Kleshchev, specifically, that they are polynomial quasi hereditary.

It follows from Theorem~\ref{thm:affine} that the representation categories of $\Bt(\cal{V})$ are affine highest weight categories; in particular, the standard modules that filter projective modules have an upper triangular structure with respect to a partial order, and descend to highly structured bases in the Grothendieck group, see Section~\ref{sec:can-aqh}.   This sort of structure is ubiquitous in modern representation theory: other examples of affine highest weight categories include module categories over finite type KLR algebras~\cite{KleshL, KleshLM}, finitely generated modules over current algebras~\cite{CI,CP,Klesh-affine}, and Kato's geometric extension algebras~\cite{Kato1,Kato2,Klesh-affine}.

 A further important consequence of the affine quasi hereditary structure of $\Bt(\cal{V})$ is that these algebras have finite projective and global dimension~\cite[Theorem B]{Klesh-affine}.  Applied to the case of cyclic hyperplane arrangements $\cal{V}$, it follows that the Ozsv{\'a}th--Szab{\'o} algebras are affine quasi hereditary (see Theorem~\ref{thm_OSz-aqh}) and have finite global dimension (Corollary~\ref{cor:OSz-finite-gd}).  Furthermore, we prove in Corollary~\ref{cor:hypSmooth} that, under mild assumptions satisfied by $\Bt(\cal{V})$, affine quasi hereditary algebras are homologically smooth.  This smoothness supports the conjecture from \cite{LLM} that the algebra $\Bt(\cal{V})$ is the homology of the endomorphism algebra of a canonical Lagrangian in a wrapped Fukaya category of the complexified hyperplane complement $\cal{X}_{\cal{V}}$.

The hypertoric convolution algebras $\Bt(\cal{V})$ arise as endomorphism algebras of projective generators of the deformed hypertoric category $\cal{O}$ introduced in \cite{HypertoricCatO}.  It is notable that the computation of $\Ext$ groups between standard modules in Corollary~\ref{cor:BTildeExtComputation} can be carried out  over the integers and the resulting groups have no torsion; by comparison, computing Ext groups between standards, even over a field, in BGG category $\cal{O}$ is a difficult problem of interest in Kazhdan-Lusztig theory.  Some partial results appear
appear in \cite{GJ,Ca}.  A similar situation arises for KLR algebras, where resolutions of standard modules can be computed~\cite{MR4245830}, but the computation of Ext groups appears difficult in general~\cite{MR4104496}. One of the important and interesting aspects of hypertoric representation theory is that it shares many of the basic structural properties of fundamental categories of Lie theory while remaining significantly less complicated from a computational perspective.

\medskip

 \noindent \textbf{Organization.} In Section~\ref{sec:def-affine-qh} we review definitions and prove some results about affine quasi hereditary algebras in general. Section~\ref{sec:HypertoricAlg} focuses on the case of equivariant hypertoric convolution algebras, introducing standard modules over these algebras and proving structural results. Section~\ref{sec:HypertoricAlgAreAQH} shows that these hypertoric algebras are affine quasi hereditary and deduces various consequences. In Section~\ref{sec:ProjRes} we give projective resolutions for standard modules over these algebras and compute $\Ext$ groups; we also define explicit chain maps between projective resolutions representing generators of the $\Ext$ groups. Sections~\ref{sec:Cyclic} and \ref{sec:BorderedFloerApps} focus further on hypertoric algebras in the cyclic case, where we use the above results to derive consequences for Ozsv{\'a}th--Szab{\'o} algebras; in Section~\ref{sec:ExtAndStrandsHomology} we introduce $\A(\Zc^{\sta}_n)$ and show that the $\Ext$ groups from Section~\ref{sec:ProjRes} are isomorphic to the homology of $\A(\Zc^{\sta}_n)$. Section~\ref{sec:HeegaardDiags}, written from more of a Heegaard Floer homology perspective, interprets projective resolutions of standard modules in terms of a Heegaard diagram and discusses further ramifications of this diagram, in particular the conjectured $\A_{\infty}$ quasi-isomorphism mentioned above.

\subsection*{Acknowledgements} The authors are grateful to Sasha Kleshchev, Walter Mazorchuck, Nick Proudfoot, Raphael Rouquier, Joshua Sussan, and Hugh Thomas for helpful conversations. The authors would like to extend special thanks to Hugh Thomas for his help with some of the proofs in Section~\ref{sec:HypertoricStandardStructure}. A.D.L. was partially supported by NSF grant DMS-1664240, DMS-1902092 and Army Research Office W911NF2010075.
A.M.L. was supported by an Australian Research Council Future fellowship.

\section{Affine quasi hereditary algebras and affine highest weight categories} \label{sec:def-affine-qh}

In this section we review the definitions of affine quasi hereditary algebras and affine highest weight categories following Kleshchev \cite{Klesh-affine}, then show that the tensor product of affine quasi hereditary algebras is affine quasi hereditary and prove some related results about Grothendieck groups.

\subsection{Noetherian Laurentian algebras and categories} \label{sec:Noetherian}

A graded vector space $V=\oplus_{n\in \Z}V_n$ over a field $\Bbbk$ is called \emph{Laurentian} if $V_{-n} = 0$ for $n \gg0$ and $\dim_{\Bbbk}V_n < \infty$ for all $n\in \Z$. Given a Laurentian graded vector space $V$, denote its graded dimension by $\dim_q V = \sum_{n \in \Z}(\dim V_n)q^n$, which is a Laurent series. A graded $\Bbbk$-algebra is a Laurentian graded algebra if its underlying graded vector space is Laurentian.  By \cite[Lemma 2.2]{Klesh-31} if $H$ is a Laurentian (graded) algebra then all irreducible $H$-modules (graded unless otherwise specified) are finite-dimensional, and $H$ is graded semiperfect so that there are only a finite number of simple $H$-modules up to isomorphism and grading shift.

In what follows we write $\hom_H(U,V)$ for degree preserving $H$-module homomorphisms and write
\[
\Hom_H(U,V):=\bigoplus_{n \in \Z} \hom_H(q^n U ,V) = \bigoplus_{n \in \Z} \hom_H(U,q^{-n} V)
\]
where $(qV)_n := V_{n-1}$. We define $\ext_H(U,V)$ and $\Ext_H(U,V)$ analogously.  We often drop the subscript of $H$ when no confusion is likely to arise.

\begin{definition}[Section 3 \cite{Klesh-affine}]
A graded $\Bbbk$-linear abelian category $\cal{C}$, with a (possibly infinite) complete and irredundant set of simple objects
$
\{ L(\alpha) \mid \alpha \in \Pi \}
$
up to isomorphism and grading shift, is called a \emph{Noetherian Laurentian category} if the following three conditions hold:
\begin{enumerate}[(i)]
  \item Every object $V \in \cal{C}$ is Noetherian and has a filtration $V \supset V_1 \supset V_2 \supset \dots $ which is separated $(\cap_{n=1}^{\infty} V_n=0)$ such that each quotient $V/V_n$ has finite length;

  \item Each simple object $L(\alpha)$ has a projective cover $P(\alpha) \twoheadrightarrow L(\alpha)$;

  \item For all $\alpha, \beta \in \Pi$, the graded vector space $\Hom_{\cal{C}}(P(\alpha),P(\beta))$ is Laurentian.
\end{enumerate}
\end{definition}

For a graded algebra $H$, let $H\gmod$ denote the category of finitely generated graded left
$H$-modules.
If $H$ is a left (graded) Noetherian Laurentian $\Bbbk$-algebra, then $H\gmod$ is a Noetherian Laurentian category with $\Pi$ finite. Any Noetherian Laurentian category with $\Pi$ finite is graded equivalent to $H\gmod$ for some left
(graded) Noetherian Laurentian algebra $H$~\cite[Theorem 3.9]{Klesh-affine}.

For a Noetherian Laurentian category $\cal{C}$, fix a projective cover $P(\alpha) \twoheadrightarrow L(\alpha)$ for each $\alpha \in \Pi$.  Assume that $\Pi$ is equipped with a partial order $\leq$.  Define the \emph{standard} and \emph{proper standard} objects $\Delta(\alpha)$ and $\bar{\Delta}(\alpha)$ as follows: $\Delta(\alpha)$ is the largest quotient of $P(\alpha)$ such that all composition factors (as defined in~\cite[Section 2.2]{Klesh-affine}) $L(\beta)$ of $\Delta(\alpha)$ satisfy $\beta \leq \alpha$, and $\bar{\Delta}(\alpha)$ is the largest quotient of $P(\alpha)$ which has $L(\alpha)$  as a composition factor with multiplicity 1 and such that all its other composition factors $L(\beta)$ satisfy $\beta < \alpha$.
More explicitly, for  $\alpha \in \Pi$, define the standard object $\Delta(\alpha)$ and the proper standard object $\bar{\Delta}(\alpha)$ as
\begin{align} \label{eq:affine-standard}
\Delta(\alpha) &:= P(\alpha) /\left( \sum_{\beta \nleq \alpha, f \in \Hom_{\cal{C}}(P(\beta),P(\alpha))} {\rm Im} f \right), \\
\bar{\Delta}(\alpha) &:= P(\alpha) /\left( \sum_{\beta \nless \alpha, f \in \Hom_{\cal{C}}(P(\beta),{\rm rad}P(\alpha))} {\rm Im} f \right).
\end{align}

\subsection{Affine highest weight categories}

If $\cal{C}$ is Noetherian Laurentian, we say that an object of $\cal{C}$ has a \emph{$\Delta$-filtration} if it has a separated filtration whose subquotients are isomorphic to grading shifts of $\Delta(\alpha)$ for various $\alpha \in \Pi$.

\begin{definition}[\cite{Klesh-affine}]
 A Noetherian Laurentian category $\cal{C}$ equipped with a partial order $\leq$ on $\Pi$ (written as $(\cal{C},\Pi,\leq)$) is a \emph{polynomial highest weight category}
if, for some (or any) choice of projective covers $P(\alpha) \twoheadrightarrow L(\alpha)$ and for each $\alpha \in \Pi$,
\begin{description}
  \item[(PHW1)]  the kernel of the natural quotient map $P(\alpha) \twoheadrightarrow \Delta(\alpha)$ has a $\Delta$-filtration  whose subquotients are isomorphic to grading shifts of $\Delta(\beta)$ with $\beta > \alpha$,

   \item[(PHW2)] the endomorphism $\Bbbk$-algebra $B_{\alpha}:= \End_{\cal{C}}(\Delta(\alpha))^{\rm op}$ is isomorphic to a graded polynomial ring $\Bbbk[z_1,\dots, z_{n_{\alpha}}]$ for some $n_{\alpha} \in \Z_{\geq 0} $ with $\deg z_i \in \Z_{>0}$.

    \item[(PHW3)] for all $\beta \in \Pi$, the $B_{\beta}$-module $\Hom_{\cal{C}}(P(\alpha), \Delta(\beta))$ is free of finite rank.
\end{description}
The category $\cal{C}$ is called a \emph{affine highest weight category} if $(\mathsf{PHW2})$ is replaced by the weaker condition $(\mathsf{AHW2})$ that $B_{\alpha}$ is an affine (i.e. finitely generated positively graded commutative) $\Bbbk$-algebra, for all $\alpha \in \Pi$.
\end{definition}
Given {\sf (PHW2)} or {\sf (AHW2)}, we need not distinguish between $B_{\alpha}$ and $B_{\alpha}^{\rm op}$.

\subsection{Affine quasi hereditary algebras}
Let $\mf{B}$ denote a class of connected algebras, such as affine algebras or positively graded polynomial algebras. Throughout this section let $H$ be a left Noetherian   Laurentian algebra with simples indexed by $\Pi$.  In particular, for each $\pi \in \Pi$ we have an indecomposable projective $P(\pi)$.

\begin{definition}
 A \emph{$\mf{B}$-hereditary ideal} in $H$ is a two-sided  ideal $J$ satisfying
\begin{description}
  \item[(H1)]  $J^2 = J$.
  \item[(H2)] As a left module $_HJ \cong m(q) P(\pi)$ for some graded multiplicities $m(q)\in\Z[q,q^{-1}]$ and some $\pi \in \Pi$, with $B_{\pi}:= \End(P(\pi))^{{\rm op}} \in \mf{B}$;
     \item[(H3)] $P(\pi)$ is a free finite rank right $B_{\pi}$-module.
\end{description}
The algebra $H$ is said to be \emph{$\mf{B}$-quasi hereditary} if there exists a finite chain of ideals
\begin{equation} \label{eq:hereditary-chain}
H  = J_0 \supsetneq J_1 \supsetneq \dots \supsetneq J_n = 0
\end{equation}
for some $n >0$, such that $J_{i-1}/J_{i}$ is a $\mf{B}$-hereditary ideal of $A/J_{i}$ for all $0<i \leq n$.
Such a chain of ideals is called a $\mf{B}$-hereditary chain.
\end{definition}

The projectivity of $J$ and the condition $J^2=J$ is equivalent to $J=HeH$ for some idempotent $e\in H$, since $H$ is Laurentian \cite[Lemma 2.8]{Klesh-affine}.  We may choose $e$ to be primitive so that $He \cong P(\pi)$ is indecomposable and $B_{\pi}\cong eHe$~\cite[Lemma 6.6]{Klesh-affine}.

\begin{theorem}[Theorem 6.7 of \cite{Klesh-affine}] \label{thm:highestwt-qh}
Let $H$ be a left
Noetherian Laurentian algebra. The category $H\gmod$ of finitely generated graded $H$-modules is affine/polynomial highest weight with respect to some partial order on the index set $\Pi$  of simple $H$-modules if and only if $H$ is affine/polynomial quasi-hereditary.
\end{theorem}

\begin{remark}[cf. Remark 1.5 of \cite{Fuj}] \label{rem:hereditary-po}
A $\mf{B}$-quasi hereditary chain of the form in \eqref{eq:hereditary-chain} determines a total order $\{ \alpha_1, \dots, \alpha_n\}$ on $\Pi$  by $J_{i-1}/J_i \cong m(q)\Delta(\alpha_i)$.  This defines a partial order $\leq$ on $\Pi$ by the condition that for $\alpha, \beta \in \Pi$ we have $\alpha < \beta$ if and only if for any $\mf{B}$-quasi hereditary chain we have $\alpha=\alpha_i$, $\beta=\alpha_j$ for some $i,j$ such that $1 \leq i < j \leq n$. This is the partial order that appears in Theorem~\ref{thm:highestwt-qh}.
\end{remark}

\subsection{Quasi hereditary plus involution}\label{subsec:QHinvolution}

Given an anti-involution $\psi \maps H \to H$ on the affine quasi hereditary algebra $H$,
we can consider any left $H$-module $V$ as a right module $V^{\psi}$ with action $v h:= \psi(h)v$ for all $v\in V$, $h \in H$.
 Given a left
 $H$-module $V$ with finite-dimensional graded components $V_n$, define its graded dual $(V^{\circledast})_n:=( V_{-n})^{\ast}$ for $n\in \Z$ with left $H$-action $h f(v) = f(\psi(h)v)$ for $f \in V^{\circledast}$, $v\in V$
and $h\in H$.  We have $(q^nV )^{\circledast}=q^{-n}(V^{\circledast})$ and $\dim_q V^{\circledast}=\dim_{q^{-1}}V$.

 An algebra anti-involution $\psi\maps H\to H$ is called \emph{balanced} if for all $\alpha \in \Pi$,
\begin{equation}
L(\alpha)^{\circledast} \cong q^{2n}L(\alpha)
\end{equation}
for some $n\in \Z$.

\begin{proposition}[Proposition 9.8 of \cite{Klesh-affine}]
If an affine quasi hereditary algebra $H$ has a balanced anti-involution then it is \emph{affine cellular} in the sense of \cite{KX}.
\end{proposition}

\subsection{Tensor products of affine quasi hereditary algebras}

The following proposition extends \cite[(1.3)]{Wiedemann} showing that the tensor product of ordinary (finite-dimensional) quasi hereditary algebras is quasi hereditary.

\begin{proposition} \label{prop:tensor}
Let $A$ and $B$ be $\mf{B}$-quasi hereditary with simples indexed by $\Pi_A$ and $\Pi_B$, and assume that $A \otimes B$ is left Noetherian. Then $A\otimes B$ is $\mf{B}$-quasi hereditary with simples indexed by $\Pi_A\sqcup \Pi_B$.
\end{proposition}

\begin{proof}
 If $A$ and $B$ are $\mf{B}$-quasi hereditary then we have $\mf{B}$-hereditary chains
\begin{align}
   A = J_0 \supsetneq J_1 \supsetneq \dots \supsetneq J_n = 0,  \qquad
   B = L_0 \supsetneq L_1 \supsetneq \dots \supsetneq L_m = 0 .\nn
\end{align}
Consider the chain of ideals
\begin{align*}
  A \otimes B &\supsetneq \cdots \supsetneq J_{n-2} \otimes L_{m-2} + J_{n-1} \otimes B \supsetneq J_{n-2} \otimes L_{m-1} + J_{n-1} \otimes B \\
	&\supsetneq J_{n-1} \otimes B \supsetneq \cdots \supsetneq J_{n-1} \otimes L_{m-2} \supsetneq J_{n-1} \otimes L_{m-1} \supsetneq 0
\end{align*}
of $A \otimes B$. We claim that this chain is $\mf{B}$-hereditary; by induction on $n$, we may assume the claim holds for $(A',B)$ whenever $A'$ admits a $\mf{B}$-hereditary chain of length $n-1$. In particular, since $A/J_{n-1}$ admits a $\mf{B}$-hereditary chain of length $n-1$ and $(A/J_{n-1}) \otimes B \cong (A \otimes B)/(J_{n-1} \otimes B)$ is left Noetherian, it suffices to show that $(J_{n-1} \otimes L_{j-1})/(J_{n-1} \otimes L_j)$ is a $\mf{B}$-hereditary ideal of $(A \otimes B)/(J_{n-1} \otimes L_j)$ for $1 \leq j \leq m$.

Set
\[
H_{n-1,j} = (A \otimes B)/(J_{n-1} \otimes L_j), \quad  J_{n-1,j} := (J_{n-1} \otimes L_{j-1})/(J_{n-1} \otimes L_j), \quad 1 \leq j \leq m.
\]

Since $J_{n-1}$ and $L_{j-1}/L_j$ are idempotent ideals of $A$ and $B/L_j$ respectively, it follows that $J_{n-1,j}$ is an idempotent ideal of $H_{n-1,j}$. We argue that $J_{n-1,j}$ is also projective as an $H_{n-1,j}$-module.   Observe that
\[
J_{n-1,j} = (J_{n-1} \otimes L_{j-1})/(J_{n-1} \otimes L_j) \cong J_{n-1} \otimes (L_{j-1}/L_j).
\]
By assumption, $L_{j-1}/L_j\cong m(q)Q$ for some indecomposable projective $Q$ over $B/L_j$ with $B_Q:=\End_B(Q)^{{\rm op}} \in \mf{B}$ and $Q$ a finite rank $B_Q$-module. The projectivity of $L_{j-1}/L_j$ implies that it is
 a direct summand of finitely many copies of $B/L_j$ as a $B/L_j$-module, so that $J_{n-1} \otimes (L_{j-1}/L_j)$ is a direct summand of finitely many copies of $J_{n-1} \otimes (B/L_j) = (J_{n-1} \otimes B)/(J_{n-1} \otimes L_j)$ as an $A \otimes (B/L_j) = (A \otimes B)/(A \otimes L_j)$-module and thus as a module over the larger algebra $H_{n-1,j}$. We have
\begin{align*}
  J_{n-1}\otimes (B/L_j) &\cong (J_{n-1} \otimes B)/(J_{n-1} \otimes L_j)
\cong (J_{n-1} \otimes B)/[(J_{n-1} \otimes L_j)(J_{n-1} \otimes B) ] \\
&\cong \left( (A \otimes B)/(J_{n-1} \otimes L_j) \right) \otimes_{A\otimes B} (J_{n-1} \otimes B)
\end{align*}
as $A \otimes (B/L_j)$-modules and thus as $H_{n-1,j}$-modules. But   $J_{n-1} \otimes B$ being $A\otimes B$-projective implies the last expression is a direct summand of finitely many copies of
\[
(A\otimes B)/(J_{n-1} \otimes L_j) \otimes_{A \otimes B} (A\otimes B) \cong (A\otimes B)/(J_{n-1} \otimes L_j) = H_{n-1,j}
\]
as an $H_{n-1,j}$-module. Hence, $J_{n-1,j}$ is projective over $H_{n-1,j}$ for all $1 \leq j \leq m$.

We also have that $J_{n-1} = J_{n-1}/J_n \cong n(q)P$  for some $n(q) \in \Z[q,q^{-1}]$ and $P$ an indecomposable projective over $A/J_n=A$, where $B_P:= \End_A(P) \in \mf{B}$ and $P$ is a finite rank $B_P$-module.   Hence,
\[
J_{n-1,j} = (J_{n-1} \otimes L_{j-1})/(J_{n-1} \otimes L_j) \cong J_{n-1} \otimes (L_{j-1}/L_j) \cong n(q)m(q) P \otimes Q
\]
as $A \otimes (B/L_j)$-modules. Write $P = Ae$ and $Q = (B/L_j)e'$ for $e$ a primitive idempotent in $A$ and $e'$ a primitive idempotent in $B/L_{j}$; note that $P\otimes Q \cong (A \otimes (B/L_j))(e \otimes e')$ .

The natural map
\begin{align*}
B_P \otimes B_Q &= \End_{A}( Ae  ) \otimes \End_{  B/ L_j}( (B/L_j) e') \\
&\xrightarrow{\cong} \End_{A \otimes (B/L_j)}((A \otimes (B/L_j)) (e \otimes e')) \\
&= \End_{A \otimes (B/L_j)}( P \otimes Q) \\
&= \End_{H_{n-1,j}}( P \otimes Q) =: B_{PQ}
\end{align*}
is an isomorphism. It follows that $B_{PQ} \in \cal{B}$ and that $P \otimes Q$ is free of finite rank over $B_{PQ}$.
\end{proof}

\begin{corollary} \label{cor:affineQHisSmooth}
Let $H$ be a $\mf{B}$-quasi hereditary algebra such that $H \otimes H^{\op}$ is left Noetherian and assume that $H \cong H^{{\rm op}}$.  Then $H$ is homologically smooth.
\end{corollary}

\begin{proof}
Proposition~\ref{prop:tensor} implies that $H \otimes H^{{\rm op}} \cong H \otimes H$ is $\mf{B}$-quasi hereditary.  Then by \cite[Corollary 5.25]{Klesh-affine}, any $\mf{B}$-quasi hereditary algebra has finite global dimension. Any finitely generated left module over a left Noetherian algebra of finite global dimension has a bounded resolution by finitely generated projective modules (one can show this using e.g. the generalized Schanuel's lemma). In particular, $H$ thought of as an $H \otimes H^{{\rm op} }$-module has a bounded resolution by finitely generated projective modules.

\end{proof}

\subsection{Proper costandard and tilting modules}\label{sec:ProperCostandardTilting}

When looking at Grothendieck groups below, it will be useful to have some additional families of modules. For an affine quasi hereditary algebra $H$, let $I(\alpha)$ be the injective hull of $L(\alpha)$ in the category of all graded $H$-modules. In general $I(\alpha)$ need not be finitely generated or Laurentian.  Let $D(\alpha)$ be the largest submodule of $I(\alpha)/L(\alpha)$ among all whose composition factors are isomorphic to grading shifts of $L(\beta)$ for various $\beta < \alpha$.
 Define the \emph{proper costandard modules}  $\bar{\nabla}(\alpha) \subset I(\alpha)$ to be the preimage of $D(\alpha)$ under the quotient map $I(\alpha) \to I(\alpha)/L(\alpha)$.
With this class of modules there is a BGG type reciprocity~\cite[Theorem 7.6]{Klesh-affine}
\[
(P(\alpha): \Delta(\beta))_q = [\bar{\nabla}(\beta): L(\alpha)]_{q^{-1}}
\]
where $(-:-)_q$ denotes the $\Delta$-multiplicity defined in \cite[Section 5.4]{Klesh-affine} and $[-:-]_q$ denotes the graded multiplicities of simple modules in composition series.

We say that an object of $\cal{C}=H\gmod$ has a $\bar{\nabla}$-filtration if it has a separated filtration whose subquotients are isomorphic to grading shifts of $\bar{\nabla}( \beta)$ for $\beta \in \Pi$.  Let $\mathsf{F}(\Delta)$ (resp. $\mathsf{F}(\bar{\nabla})$) denote the full subcategory of $H\gmod$ consisting of all $\Delta$-filtered (resp. $\bar{\nabla}$-filtered) objects.  Write $\cal{T}=\mathsf{F}(\Delta)\cap \mathsf{F}(\bar{\nabla})$.

\begin{definition}
A module in $H\gmod$ is called a \emph{tilting module} if it  is both $\Delta$-filtered and $\bar{\nabla}$-filtered, i.e if it is in $\cal{T}$.
\end{definition}

Tilting modules are closely related to the theory of positively graded standardly stratified algebras in the sense of Mazorchuk \cite{Maz-standardly}; affine quasi hereditary algebras are positively graded standardly stratified and their associated categories of tilting modules can be described as above.

\subsection{Grothendieck groups of affine quasi hereditary algebras} \label{sec:Grothendieck-affine}

Let $H$ be an affine or polynomial quasi hereditary algebra with simples indexed by $\Pi$. Let $K_0(H) = K(\cal{D}_c(H))$ denote the Grothendieck group of the compact derived category of $H$ (equivalently, of the homotopy category of perfect complexes over $H$, or of finite complexes of finitely generated projective $H$-modules). Since $H$ is left Noetherian and has finite global dimension by \cite[Corollary 5.25]{Klesh-affine}, all finitely generated graded left $H$-modules are compact as objects of $\cal{D}(H)$ (as in the proof of Corollary~\ref{cor:affineQHisSmooth}).

There is a natural identification of $K_0(H)$ with $K(H\pmod)$, where $H\pmod$ is the category of finitely generated graded projective
left
modules over $H$ and $K$ denotes either the split Grothendieck group or the Grothendieck group of an exact category; both $K_0(H)$ and $K(H\pmod)$ are free $\Z[q,q^{-1}]$-modules with a basis given by classes $[P(\alpha)]$ of indecomposable projectives (in homological degree zero as objects of $\cal{D}_c(H)$) for $\alpha \in \Pi$.

Let $G_0(H)$ denote the Grothendieck group of the abelian category $H\fmod$ of finite-dimensional graded right modules over $H$. The group $G_0(H)$ is a free $\Z[q,q^{-1}]$-module with basis the classes of simple modules $[L(\alpha)]$ for $\alpha \in \Pi$. We also write $K_0^{\C(q)}(H):=K_0(H)\otimes_{\Z[q,q^{-1}]} \C(q)$ and $G_0^{\C(q)}(H):=G_0(H)\otimes_{\Z[q,q^{-1}]} \C(q)$; we define $K_0^{\Z((q))}(H)$ and $G_0^{\Z((q))}(H)$ similarly.

\begin{proposition}\label{prop:BasesFromStandardsProperCostandards}
Given an affine/polynomial quasi hereditary algebra $H$ with simples indexed by $\Pi$, the classes of
standard modules $[\Delta(\alpha)]$ for $\alpha \in \Pi$ give another $\Z[q,q^{-1}]$-basis for $K_0(H)$; the
proper costandard modules $[\bar{\nabla}(\alpha)]$ for $\alpha \in \Pi$ give another $\Z[q,q^{-1}]$-basis for $G_0(H)$.

\end{proposition}

\begin{proof}
The existence of a $\Delta$-filtration for the kernel of $P(\alpha) \twoheadrightarrow \Delta(\alpha)$, with subquotients shifts of $\Delta(\beta)$ for $\beta > \alpha$, implies that in $K_0(H)$ we have $[P(\alpha)] = [\Delta(\alpha)]$ plus a $\Z[q,q^{-1}]$-linear combination of $[\Delta(\beta)]$ for $\beta > \alpha$. It follows that $\{[\Delta(\alpha)]\}$ gives a $\Z[q,q^{-1}]$-basis for $K_0(H)$; the entries for the change-of-basis matrix between this basis and $\{[P(\alpha)]\}$ are the $\Delta$-multiplicities $(P(\alpha) : \Delta(\beta))_q$. As above, we have
$
(P(\alpha): \Delta(\beta))_q = [\bar{\nabla}(\beta): L(\alpha)]_{q^{-1}}
$
since $H$ is affine quasi hereditary, so $\{[\bar{\nabla}(\alpha)]\}$ gives a $\Z[q,q^{-1}]$-basis for $G_0(H)$.
\end{proof}

For an affine quasi hereditary algebra $H$ with balanced involution, we can shift the simple modules so that
$L^{\circledast}(\alpha)\cong L(\alpha)$ for all $\alpha \in \Pi$.  Then, as in \cite[Section 9.1]{Klesh-affine}, there is an isomorphism $\bar{\Delta}(\alpha)^{\circledast} \cong \bar{\nabla}(\alpha)$ for all $\alpha \in \Pi$   and we have
\begin{equation} \label{eq:BGG-prop-standard}
  (P(\alpha) : \Delta(\beta))_q = [\bar{\Delta}(\beta) : L(\alpha)]_q.
\end{equation}
It follows that the classes of proper standard modules $\bar{\Delta}(\alpha)$, for $\alpha \in \Pi$, form a basis of $G_0(H)$ as a free $\Z[q,q^{-1}]$-module.

Now let $H$ be polynomial quasi hereditary; \cite[Proposition 5.7]{Klesh-affine} implies that
\[
[\bar{\Delta}(\alpha)] = \frac{1}{\dim_q B_{\alpha}}[\Delta(\alpha)]
\]
in $K_0(H)$, where $\frac{1}{\dim_q B_{\alpha}} \in \Z[q,q^{-1}]$ by the polynomial quasi hereditary condition. It follows that the classes $[\bar{\Delta}(\alpha)]$ are independent over $\Z[q,q^{-1}]$ in $K_0(H)$. Thus, if $H$ is polynomial quasi hereditary with a balanced anti-involution, we can identify $G_0(H)$ with the $\Z[q,q^{-1}]$-span of the classes $[\bar{\Delta}(\alpha)]$ in $K_0(H)$.

If we pass to $\C(q)$ or $\Z((q))$, then $\frac{1}{\dim_q B_{\alpha}}$ is invertible, so the above paragraph lets us identify $G_0^{\C(q)}(H)$ with $K_0^{\C(q)}(H)$ and $G_0^{\Z((q))}(H)$ with $K_0^{\Z((q))}(H)$. It follows that the classes of standard modules $\{\Delta(\alpha) \mid \alpha \in \Pi\}$ and the classes of indecomposable projectives $\{P(\alpha) \mid \alpha \in \Pi\}$ give bases for $G_0^{\C(q)}(H)$ over $\C(q)$, as well as for $G_0^{\Z((q))}(H)$ over $\Z((q))$.

Finally, if $H$ is polynomial quasi hereditary (and thus positively graded standardly stratified) with a balanced anti-involution, then $G_0^{\C(q)}(H)$ and $G_0^{\Z((q))}$ have bases given by classes of indecomposable tilting modules $T(\alpha)$ for $\alpha \in \Pi$. This follows from \cite{Fuj-tilt}, since $(T(\pi) :\Delta(\pi))_q=1$ and $(T(\pi):\Delta(\sigma))_q=0$ for any $\sigma \nleq \pi$, so that the change of basis from tilting modules to standard modules is lower triangular with ones on the diagonal.

\subsubsection{Bilinear form} \label{subsec:bilinearform}

Assuming that $H$ is affine quasi hereditary with a balanced anti-involution $\psi$ as in Section \ref{subsec:affine-cellular}, there is a $\Z((q))$-bilinear form   defined on $K_0(H)$ given by
\begin{equation} \label{eq:form-HOM}
([X],[Y]) := \sum_{j \in \Z} (-1)^j \dim_q(\Hom^j(X^{\#},Y))
\end{equation}
where, for $X$ finitely generated and projective, $X^{\#}$ is defined to be $\Hom_H(X,H)$ with
left action $(h f)(x) := f(\psi(h)x)$,
and the operation $\#$ is extended to finite complexes in the natural way.

If $X$ is any object of $\cal{D}^c(H)$ (in particular, any finitely generated $H$-module) and $N$ is a finite-dimensional $H$-module, the proof of \cite[Lemma 2.5]{BrundanKlesh} gives
\begin{equation} \label{eq:Ext-pairing-new}
([X],[N])  = \sum_{j \in \Z}(-1)^j\overline{\dim_q(\Ext^j(X,N^{\circledast}))} = \sum_{j \in \Z}(-1)^j \dim_{q^{-1}}(\Ext^j(X,N^{\circledast}))
\end{equation}
 where $\bar{}$ denotes the bar involution of $\Z[q,q^{-1}]$ with $\bar{q}=q^{-1}$ and $N^{\circledast}$ denotes the dual with respect to the antiautomorphism $\psi$ (note that to compute these $\Ext$ groups we may assume $X$ is a finite complex of finitely generated projective modules). We have
\[
([X^{\#}],[N]) = \sum_{j \in \Z} (-1)^j \dim_q(\Hom^j(X,N))
\]
by definition, while
\[
([X],[N^{\circledast}]) =  \sum_{j \in \Z}(-1)^j\overline{\dim_q(\Ext^j(X,N))}
\]
by equation~\eqref{eq:Ext-pairing-new}. Thus,
\begin{equation} \label{eq:PsiAdjointness}
([X^{\#}],[N]) = \overline{([X],[N^{\circledast}])}
\end{equation}
for $[X] \in K_0(H)$ and $[N] \in G_0(H)$.

We see that the form \eqref{eq:form-HOM} restricts to give a $\Z[q,q^{-1}]$ perfect pairing
\begin{align}
K_0(H) \times G_0(H)  &\longrightarrow \Z[q,q^{-1}] \label{eq:perfect-pairing}\\
[P], [M] &\mapsto
\dim_{q^{-1}}
\Hom_H(P, M^{\circledast})
=\dim_q
\Hom_H(P^{\#}, M); \nn
\end{align}
note that $L^{\circledast}(\beta)\cong L(\beta)$ and $\dim_q\Hom_H(P(\alpha), L(\beta)) = \delta_{\alpha \beta}$.
The matrix for this pairing is the identity in the basis of indecomposable projectives and simples, so we can identify $K_0(H)$ and $G_0(H)^{\ast}$ over $\Z[q,q^{-1}]$. Passing to $\C(q)$, the pairing $(,)$ on $K_0^{\C(q)}(H)$ becomes invertible, so we can also use it to identify $K_0^{\C(q)}(H)$ with $(K_0^{\C(q)}(H))^*$ and thereby with $G_0^{\C(q)}(H)$. Similarly, we can identify $K_0^{\Z((q))}(H)$ with $G_0^{\Z((q))}(H)$. These identifications agree with the ones given below Proposition~\ref{prop:BasesFromStandardsProperCostandards}.

\begin{proposition} \label{prop:pairing}
Let $H$ be an affine quasi hereditary algebra with balanced anti-involution. With respect to the form $(-,-)$ defined in \eqref{eq:Ext-pairing-new},
\begin{enumerate}
 \item  indecomposable projectives are dual to the simple modules:
$
([P(\alpha)], [L(\beta)]) = \delta_{\alpha \beta};
$
  \item   standard modules are dual to the proper standard modules:
\[
([\Delta(\alpha)], [\bar{\Delta}(\beta)]) = \delta_{\alpha \beta};
\]

  \item the pairing between projective and proper standards satisfies
$
([P(\alpha)],  [\bar{\Delta}(\beta)]) =  (P(\alpha) : \Delta(\beta))_q,
$
so that the matrix representing $([P(\alpha)], [ \bar{\Delta}(\beta)])$   is unipotent;

\item the pairing between projectives and proper costandards satisfies
$
([P(\alpha)],  [\bar{\nabla}(\beta)]) = (P(\alpha) : \Delta(\beta))_{q^{-1}},
$
so that the matrix representing $([P(\alpha)], [\bar{\nabla}(\beta)])$   is unipotent.
\end{enumerate}
\end{proposition}

\begin{proof}
The first claim follows from the fact that $L^{\circledast}(\beta)\cong L(\beta)$ and $\dim_q\Hom(P(\alpha), L(\beta)) = \delta_{\alpha \beta}$.  The second claim follows from the isomorphism
$\bar{\Delta}(\beta)^{\circledast} \cong \bar{\nabla}(\beta)$
and \cite[Lemma 7.2 and Lemma 7.3]{Klesh-affine} showing that
\begin{equation} \label{eq:Ext-dual-standard}
\Ext^i(\Delta(\alpha), \bar{\nabla}(\beta))
=
\left\{
  \begin{array}{ll}
    \delta_{\alpha \beta}, & \hbox{if $i=1$;} \\
    0, & \hbox{otherwise.}
  \end{array}
\right.
\end{equation}
For the third claim, observe that
\begin{align*}
 ([P(\alpha)], [ \bar{\Delta}(\beta)]) = \overline{\dim_{q}\Hom(P(\alpha),\bar{\nabla}(\beta) )}
= \dim_{q^{-1}}\Hom_H(P(\alpha),\bar{\nabla}(\beta) )  = (P(\alpha): \Delta(\beta))_q
\end{align*}
where the last equality follows from \cite[Lemma 7.5]{Klesh-affine}.
We also have
\begin{align*}
 ([P(\alpha)], [ \bar{\nabla}(\beta)])
&=  \dim_{q^{-1}}\Hom(P(\alpha),\bar{\Delta}(\beta) )
 = [\bar{\Delta}(\beta):L(\alpha)]_{q^{-1}} = (P(\alpha) : \Delta(\beta))_{q^{-1}}
\end{align*}
where the second equality follows from \cite[Equation (3.5)]{Klesh-affine} and the last by \eqref{eq:BGG-prop-standard}.
\end{proof}
Similarly, it follows from \cite[Lemma 4.3 and Lemma 5.26]{Klesh-affine} that for $i \geq 0$
\begin{equation}
([\bar{\Delta}(\alpha)], [\bar{\nabla}(\beta)]) := \sum_{j \in \Z}(-1)^j \dim_{q^{-1}}\Ext^i(\bar{\Delta}(\alpha),\bar{\Delta}(\beta)))  =0
\end{equation}
if $\alpha \nleq \beta$.

\section{Hypertoric convolution algebras}\label{sec:HypertoricAlg}

In this section we review the definitions of, and establish some facts about, the ``deformed'' hypertoric convolution algebras $\At(\V)$ and $\Bt(\V)$ introduced in \cite{Gale}. In particular, we introduce standard modules over these algebras, which we will use in Section~\ref{sec:HypertoricAlgAreAQH} to show that $\At(\V)$ and $\Bt(\V)$ are affine quasi hereditary algebras.

\subsection{Basic definitions}

\begin{definition}
  A \emph{polarized arrangement} indexed by $I$ with $|I|=k$ is a triple $\cal{V} = (V, \eta, \xi)$
  consisting of
  \begin{itemize}
  \item a vector subspace $V \subset \R^k$,
  \item a vector $\eta \in \R^k/V$, and
  \item a covector $\xi \in V^*=(\R^k)^*/V^\perp$,
  \end{itemize}
such that
  \begin{enumerate}
 \item[(a)] every lift of $\eta$ to $\R^k$ has at least $|I|-\dim V$ non-zero entries, and
\item[(b)] every lift of $\xi$ to $(\R^k)^*$ has at least $\dim V$ non-zero entries.
\end{enumerate}
(Note that for $V$ fixed, a generic $\eta$ will satisfy (a), and a generic $\xi$ will
satisfy (b).)
If $V$, $\eta$, and $\xi$ are all defined over $\Q$, then $\cV$ is called
\emph{rational}.
\end{definition}

\begin{example}\label{ex:n4k2}
Let $I = \{1,2,3,4\}$, so that $k=4$. Take $V$ to be the column span of the matrix $\begin{bmatrix} 1 & 1 \\ 1 & 2 \\ 1 & 3 \\ 1 & 4 \end{bmatrix}$, so that $\dim V = 2$. Let $\eta$ be the equivalence class in $\R^4 / V$ of the vector $\begin{bmatrix} 1 \\ 4 \\ 9 \\ 16 \end{bmatrix}$. Let $\xi \in V^*$ be the linear functional on $V$ sending $\begin{bmatrix} 1 \\ 1 \\ 1 \\ 1 \end{bmatrix}$ to $1$ and sending $\begin{bmatrix} 1 \\ 2 \\ 3 \\ 4 \end{bmatrix}$ to $5$. This example was also mentioned in \cite[Section 3.9.3]{LLM}; it is right cyclic in the sense of Definition~\ref{def:LeftRightCyclic} below.
\end{example}


\begin{figure}
\includegraphics[scale=0.67]{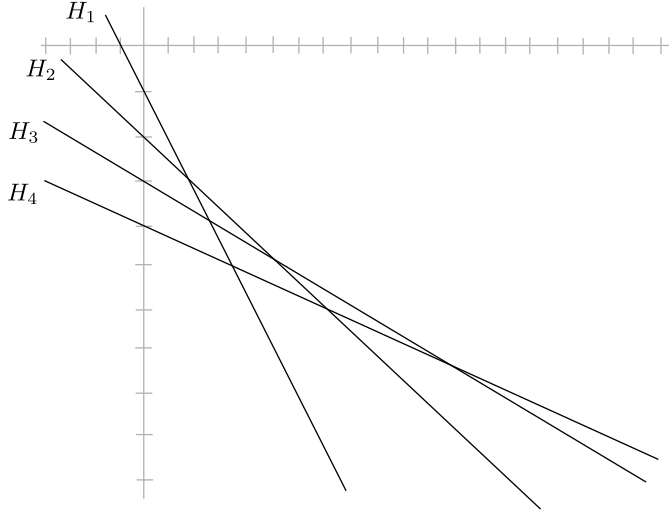}
\caption{The arrangement of hyperplanes in $\R^2$ associated to Example~\ref{ex:n4k2}.}
\label{fig:ArrangementExampleHyperplanesOnly}
\end{figure}

Associated to a (not necessarily rational) polarized arrangement
$\cV = (V,\eta,\xi)$ is an arrangement $\cal{H}$ of  $|I|$ hyperplanes
in the affine space
\[V_\eta = \{x \in \R^k \mid \eta = x + V\},\]
whose $i^\text{th}$
hyperplane is given by $$H_i = \{x \in V_\eta \mid x_i = 0\}.$$
Note that $H_i$ could be empty if $V$ is contained in the
coordinate hyperplane $\{x_i = 0\}$.

\begin{example}
The arrangement $\cal{H}$ associated to Example~\ref{ex:n4k2} is shown in Figure~\ref{fig:ArrangementExampleHyperplanesOnly}; we draw arrangements in $V_{\eta}$ as arrangements in $\R^2$ by using the identification
\[
(x,y) \in \R^2 \leftrightarrow \left( x \begin{bmatrix} 1 \\ 1 \\ 1 \\ 1 \end{bmatrix} + y \begin{bmatrix} 1 \\ 2 \\ 3 \\ 4 \end{bmatrix} + \begin{bmatrix} 1 \\ 4 \\ 9 \\ 16 \end{bmatrix} \right) \in V_{\eta}.
\]
The hyperplane $H_1$ is defined by the equation $x + y + 1 = 0$, the hyperplane $H_2$ is defined by the equation $x + 2y + 4 = 0$, the hyperplane $H_3$ is defined by the equation $x + 3y + 9 = 0$, and the hyperplane $H_4$ is defined by the equation $x + 4y + 16 = 0$.
\end{example}

For any subset $S\subset I$, denote the intersection of hyperplanes in $S$ by $$H_S = \bigcap_{i\in S} H_i.$$
Condition (a) implies that $\cH$ is simple, meaning that
$\operatorname{codim} H_S = |S|$
whenever $H_S$ is nonempty.
Condition (b) implies that $\xi$
is generic with respect to the arrangement,
in the sense that it is not constant on any positive-dimensional flat $H_S$.

Given a sign vector $\a \in \{\pm 1\}^I$, let
\[
	\Delta_\a = V_\eta \cap \{x \in \R^I \mid \a(i)x_i \ge 0 \;\text{for all $i$}\}
\]
and
$$\Sigma_\a = V \cap \{x \in \R^I \mid \a(i)x_i \ge 0 \;\text{for all $i$}\}.$$
In what follows we sometimes write $\alpha(i)=\alpha_i$ to denote the $i^{th}$ term of the sequence $\alpha$.
If $\Delta_\a$ is nonempty, it is the closed chamber of the arrangement
$\cH$ where the defining
equations of the hyperplanes are replaced by
inequalities according to the signs in $\a$.
The cone $\Sigma_\a$ is the corresponding chamber of the central arrangement
given by translating the hyperplanes of $\cH$ to the origin.
It is always nonempty, as it contains $0$. See Figure~\ref{fig:ArrangementExampleChamber} for an example of a chamber $\Delta_{\a}$.

\begin{figure}
\includegraphics[scale=0.67]{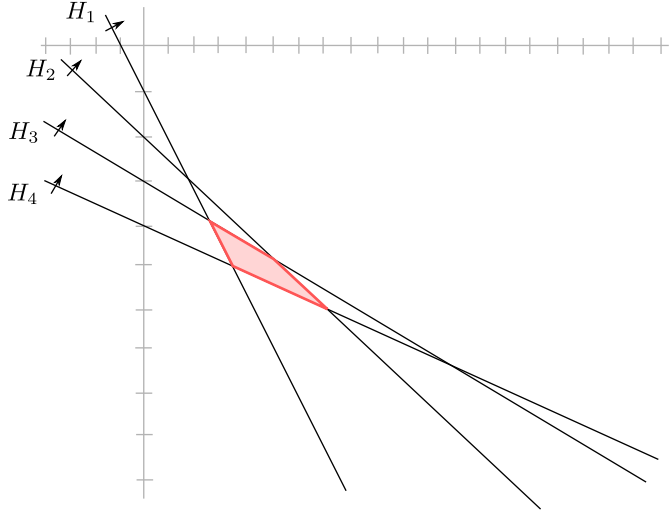}
\caption{The chamber $\Delta_{\a}$, for $\a = (+--+)$, in the arrangement of Example~\ref{ex:n4k2}. The co-orientations on the hyperplanes are determined by the matrix in that example.}
\label{fig:ArrangementExampleChamber}
\end{figure}

Introduce subsets of $\{ \pm 1\}^I$ of \emph{feasible} and \emph{bounded} regions defined by
\[
	\cF = \{\a \in \{\pm 1\}^I \mid \Delta_\a \ne \emptyset\}, \qquad \cal{B} = \{\a \in \{\pm 1\}^I \mid \xi(\Sigma_\a)\,\text{is bounded above}\}.
\]
 It is clear that $\cF$ depends only
on $V$ and $\eta$ and  that $\cal{B}$ depends
only on $V$ and $\xi$.  We let $\cal{K} = \cal{K}(\cal{V}) \subset \cal{F}(\cal{V})$ denote the set of feasible sign sequences $\alpha$ such that $\Delta_{\alpha}$ is compact.
Elements of the intersection
$$\cal{P}:= \cF\cap \cal{B} = \{\a \in \{\pm 1\}^I \mid \xi(\Delta_\a)\,\text{is
nonempty and bounded above}\}$$ are called  \emph{bounded feasible};
here $\xi(\Delta_\a)$ is regarded as a subset of the affine line.

\begin{example}
The union of the bounded feasible regions in Example~\ref{ex:n4k2} is shown as the shaded area in Figure~\ref{fig:ArrangementExample}. In this figure, $\xi$ is indicated by a dotted line whose co-orientation arrow indicates the direction of increasing $\xi$.
\end{example}

\begin{figure}
\includegraphics[scale=0.67]{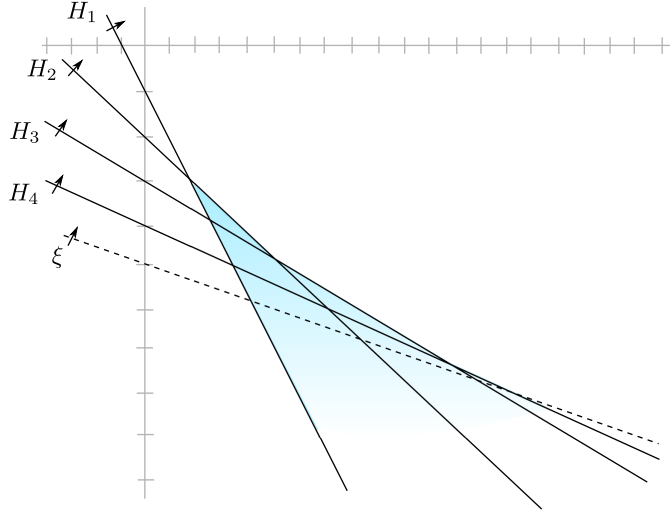}
\caption{The union of the bounded feasible regions in Example~\ref{ex:n4k2}.}
\label{fig:ArrangementExample}
\end{figure}

\subsubsection{Gale duality}
\begin{definition}
The \emph{Gale dual} $\cV^\gd$ of a polarized arrangement
$\cV=(V,\eta,\xi)$ is given by the triple $(V^\perp, -\xi, -\eta)$.
We denote by $\cF^\gd$, $\cB^\gd$, and $\cP^\gd$ the
feasible, bounded, and bounded feasible sign vectors for $\cV^\gd$,
and we denote by $V^\bot_{-\xi}$ the affine space for
the corresponding hyperplane arrangement $\cH^\vee$.
\end{definition}

By \cite[Theorem 2.4]{Gale} we have $\cF^\gd = \cB$, $\,\cB^\gd = \cF$, and therefore $\cP^\gd = \cP$.  We will sometimes use $\Delta_{\alpha}^{\gd}$ to denote the chamber  in $\cV^{\gd}$ associated to $\alpha$.  Likewise, we denote by $H_i^{\vee}$ the $i$th hyperplane in $\cH^{\vee}$ and for some subset $S \subset I$ we denote by $H_S^{\vee}$ the intersection of hyperplanes indexed by $S$.

\subsection{Partial order} \label{subsec:partial_order}
In this section we introduce various structures on polarized arrangements that will be needed for the construction of the affine quasi hereditary structure defined in Section~\ref{sec:HypertoricAlgAreAQH}.  Let $\cal{V}$ be a polarized arrangement. Let $\mathbb{B}$ denote the set of $k$-element subsets $\mathbbm{x}$ of $I=\{1,\dots, n\}$  such that
\[
H_{\mathbbm{x}} = \bigcap_{i \in \mathbbm{x}} H_i \neq \emptyset.
\]
Equivalently, $\mathbb{B}$ is the set of bases of the matroid associated to $\cal{V}$.
There is a bijection $\mu\maps \mathbb{B} \to \cal{P}$ sending $\mathbbm{x}$ to the unique sign sequence $\alpha_{\mathbbm{x}}$ such that $\xi$ obtains its maximum on $\Delta_{\alpha}$ at the point $H_{\mathbbm{x}}$. We write $\mathbbm{x}_{\beta} = \mu^{-1}(\beta)$ for the subset associated to a sign sequence $\beta$.
 The covector $\xi$ induces a partial order\footnote{This partial order is the transitive closure of the relation $\preceq$, where $\mathbbm{x} \preceq \mathbbm{x}'$ if $|\mathbbm{x}\cap \mathbbm{x}'| = |\mathbbm{x}|-1 = k-1$ and $\xi(H_{\mathbbm{x}})<\xi(H_{\mathbbm{x}'})$. The first condition ensures that $H_{\mathbbm{x}}$ and $H_{\mathbbm{x}'}$ lie on the same one dimensional flat, so that $\xi$ cannot take the same value at these two points. } $\leq $ on $\mathbb{B}\cong \cal{P}$.

Write $\mathbbm{x}^c$ for the complement in $I$ of the subset $\mathbbm{x}$.  Letting $\mathbb{B}^{\vee}$ denote the set of matroid bases of $\cal{V}^{\vee}$, the map $\mathbbm{x} \mapsto \mathbbm{x}^c$ defines a bijection from $\mathbb{B} \to \mathbb{B}^{\vee}$.  The bijection $\mu^{\vee} \maps \mathbb{B}^{\vee} \to \cal{P}^{\vee}$ is compatible with the equality $\cal{P}=\cal{P}^{\vee}$, so that $\mu(\mathbbm{x})=\mu^{\vee}(\mathbbm{x}^c)$~\cite[Lemma 2.9]{Gale}.

For $\mathbbm{x} \in \mathbb{B}$, define the \emph{bounded cone} of $\mathbbm{x}$ as
\begin{equation}
\cal{B}_{\mathbbm{x}} = \{ \alpha \in \{\pm\}^{n} \mid \alpha(i) = \mu(\mathbbm{x})(i) \; \text{for all $i \in \mathbbm{x}$} \}.
\end{equation}
Observe that $\cal{B}_{\mathbbm{x}}\subset \cal{B}$ and $\cal{B}_{\mathbbm{x}}$ depends only on $V$ and $\xi$.
Geometrically, the feasible sign sequences $\alpha$ in $\cal{B}_{\mathbbm{x}}$ are those such that $\Delta_{\alpha}$ lies in the negative cone defined by $\xi$ at the vertex $H_{\mathbbm{x}}$ as shown in Figure~\ref{fig:BoundedFeasible}.

\begin{figure}\label{fig:cones}
\includegraphics[scale=0.95]{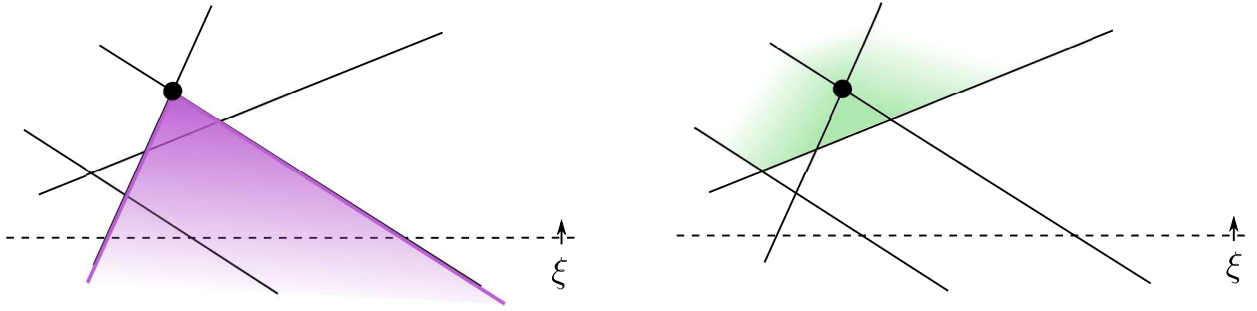}
\caption{Bounded cone $\cal{B}_{\xx}$ (left) and feasible cone $\cal{F}_{\xx}$ (right) of $\xx$; in each case the dot is the point $H_{\xx}$ and feasible regions in the bounded or feasible cone are shaded.}
\label{fig:BoundedFeasible}
\end{figure}

Dually, we have the \emph{feasible cone }
\begin{equation}
\cal{F}_{\mathbbm{x}} = \{ \alpha \in \{\pm\}^{n} \mid \alpha(i) = \mu(\mathbbm{x})(i) \; \text{for all $i \in \mathbbm{x}^c$} \}
\end{equation}
which denotes the set of sign vectors $\alpha$ such that $H_{\xx} \in \Delta_{\a}$ as shown in Figure~\ref{fig:BoundedFeasible} (equivalently, $\cal{F}_{\xx}$ consists of all feasible regions that one can reach from $\alpha_{\xx}$ by crossing hyperplanes incident to $H_{\xx}$). In particular, $\cal{F}_{\mathbbm{x}} \subset \cal{F}$. Under Gale duality we have
\begin{equation}
  \cal{F}_{\mathbbm{x}} = \cal{B}_{\mathbbm{x}^c}^{\vee}, \qquad
   \cal{B}_{\mathbbm{x}} = \cal{F}_{\mathbbm{x}^c}^{\vee},
\end{equation}
justifying the terminology ``feasible cone" as $\cal{F}_{\mathbbm{x}}$ describes a cone in the Gale dual arrangement $\cal{V}^{\vee}$.

\begin{example}
In Example~\ref{ex:n4k2} the bijection $\mu\maps \mathbb{B} \to \cal{P}$ is given by
\begin{alignat}{3}
&\mu(\{1,2 \}) = (+-++), \quad &\mu(\{1,3 \}) =  (+--+), \quad
&\mu(\{2,3\}) = (++-+) ,
\\
& \mu(\{ 1,4\}) = (+---), \quad
&\mu(\{2,4 \}) = (++--), \quad  &\mu(\{3,4 \}) = (+++-),
\end{alignat}
with partial order $\{3,4\} < \{2,4\} < \{1,4\} \& \{2,3\} < \{1,3\} < \{1,2\}$.  The bounded cone $\cal{B}_{\{1,3\}}$ consists of all those $\alpha \in \{\pm\}^{4}$ with $\alpha(1)= \mu(\{ 1,3\})(1) = +$ and $\alpha(3)= \mu(\{ 1,3\})(3) = -$.  In particular, the feasible regions in the bounded cone $\cal{B}_{\{1,3\}}$ are the regions above the hyperplane $H_1$ and below the hyperplane $H_3$, namely $(+--+)$, $(+---)$, $(++-+)$, and  $(++--)$.
\end{example}

For additional examples of polarized arrangements, their Gale duals, and associated partial orders see \cite[Examples 2.2, 2.5, 2.7, and 2.12]{Gale}.


The following is implied by \cite[Lemma 2.11]{Gale} together with Gale duality.

\begin{lemma}[cf. Lemma 2.11 of \cite{Gale}]\label{lem:OrderFromCones}
  If $\mu(\xx) \in \cal{B}_{\yy}$, then $\xx\leq \yy$ and if $\mu(\zz) \in \cal{F}_{\yy}$ then $\yy \leq \zz$.
\end{lemma}

The following lemma identifies feasible chambers in the bounded cone $\cal{B}_{\mathbbm{x}}$ in $\cal{V}$ with  chambers in $\cal{V}^{\vee}$ having nonempty intersection with $\alpha_{\mathbbm{x}}=\mu(\mathbbm{x})$.

\begin{lemma} \label{lem:cone-adjacent}
Let $\gamma \in \cal{F}$ and $\mathbbm{x} \in \mathbb{B}$.  Then $\gamma \in \cal{F} \cap \cal{B}_{\mathbbm{x}}$ in $\cal{V}$ if and only  $H_{\mathbbm{x}^c}^{\vee} \cap \Delta_{\gamma}^{\vee} \neq \emptyset$ in $\cal{V}^{\vee}$.
\end{lemma}

\begin{proof}
This is immediate since  $\gamma \in \cal{F} \cap \cal{B}_{\mathbbm{x}} \leftrightarrow
\gamma \in \cal{B}^{\vee} \cap \cal{F}_{\mathbbm{x}^c}^{\vee}$ and the feasible cone $\cal{F}_{\mathbbm{x}^c}^{\vee}$ consists of those sign sequences satisfying  $H_{\mathbbm{x}^c}^{\vee} \cap \Delta_{\gamma}^{\vee} \neq \emptyset$ (note that $\gamma \in \cal{B}^{\vee}$ since $\gamma \in \cal{F}$).

\end{proof}

\subsection{Convolution algebras}\label{sec:ABAlgebraDefs}

We now recall the algebras associated to a polarized arrangement $\cal{V}=(V,\eta,\xi)$.   Almost all of the material in this section is taken directly from the original sources \cite{Gale,HypertoricCatO,BLPPW}.

\subsubsection{The $A$ algebras} \label{subsec:defA}

For sign sequences $\alpha,\beta\in \{\pm\}^n$, we write
 \[
	\alpha \leftrightarrow \beta \iff \alpha \text{ and } \beta \text{ differ in exactly one entry.}
\]
If $\alpha , \beta \in \cal{F}$ this means that $\Delta_{\alpha}$ and $\Delta_{\beta}$ are related by crossing a single hyperplane $H_i$, in which case we write $\beta=\alpha^i$.

Define a quiver $Q = Q(\cal{V})$ whose vertex set is $\cal{F}$ and which has arrows $p(\alpha,\beta)$ from $\alpha$ to $\beta$ and $p(\beta,\alpha)$ from $\beta$ to $\alpha$ if and only if $\alpha \leftrightarrow \beta$. Let $P(Q)$ be the path algebra of this quiver over $\Z$; $P(Q)$ has a distinguished idempotent $e_{\alpha}$ for all $\alpha \in \cal{F}$.

\begin{definition}[Definition 3.1 and Remark 3.1 of \cite{Gale}]
The $\Z$-algebra $\tilde{A}(\cal{V})$ is defined to be $P(Q) \otimes_{\Z} \Z[t_1,\ldots,t_n]$ modulo the two-sided ideal generated by the following relations:
\begin{enumerate}
\item[A1]: $e_{\alpha}$ for all $\alpha \in \cal{F}\setminus \cal{B}$, that is those feasible $\alpha$ that are not bounded,
\item[A2]: $p(\alpha,\beta) p(\beta,\gamma) - p(\alpha,\delta) p(\delta,\gamma)$ for all distinct $\alpha,\beta,\gamma,\delta \in \cal{F}$ with $\alpha \leftrightarrow \beta \leftrightarrow \gamma \leftrightarrow \delta \leftrightarrow \alpha$,
\item[A3]: $p(\alpha,\beta,\alpha) - t_i e_{\alpha}$ for all $\alpha,\beta \in \cal{F}$ with $\alpha \leftrightarrow \beta$ via a sign change in coordinate $i$.
\end{enumerate}
We give $\tilde{A}(\cal{V})$ a grading by setting $\deg(p(\alpha,\beta)) = 1$ and $\deg(t_i) = 2$. While $\tilde{A}(\cal{V})$ is a $\Z$-algebra a priori, we can view it as a $\Z[t_1,\ldots,t_n]$-algebra.
\end{definition}

Over $\R$ (or $\Q$ given a rational arrangement), the infinite-dimensional algebra $\tilde{A}(\cal{V})$ can be viewed as the universal graded flat deformation in the sense of \cite{BLPPW} of a finite-dimensional quasi-hereditary Koszul algebra $A(\cal{V})$; see \cite[Remark 4.5]{Gale}. We briefly recall the definition of $A(\cal{V})$ below.

We have $\R[t_1,\ldots,t_n] \cong \Sym((\R^n)^*)$, where $\Sym()$ denotes the symmetric algebra of a vector space; the isomorphism identifies $t_i$ with the $i$-th coordinate function on $\R^n$.  We can then identify $V^*$ with $(\R^n)^*/V^{\perp}$. It follows that $\Sym(V^*)$ is the quotient of $\R[t_1,\ldots,t_n]$ by the ideal generated by all linear combinations of $t_1,\ldots,t_n$ whose coefficient vectors annihilate $V$; equivalently, we have $\Sym(V^*) \cong \R[t_1,\ldots,t_n] \otimes_{\Sym(V^{\perp})} \R$. The algebra $A(\cal{V})$ is defined similarly to $\tilde{A}(\cal{V})$, except that we take a quotient of $P(Q) \otimes_{\R} \Sym(V^*)$ instead of $P(Q) \otimes_{\Z} \Z[t_1,\ldots,t_n]$. The algebra $A(\cal{V})$ inherits a grading from $\At(\cal{V})$. It follows from \cite[Theorem 8.7]{BLPPW} that
\[
\Sym(V^{\perp}) \xrightarrow{j} \tilde{A}(\cal{V}) \xrightarrow{\pi} A(\cal{V})
\]
is a graded flat deformation that is universal in the sense of \cite[Remark 4.2]{BLPPW}, where $j$ includes an element of $\Sym(V^{\perp})$ into $\R[t_1,\ldots,t_n]$ and then multiplies by $1 \in \tilde{A}(\cal{V})$, while $\pi$ is the natural quotient map from $\tilde{A}(\cal{V})$ to $A(\cal{V})$.

\begin{example}
For the polarized arrangement in Example~\ref{ex:n4k2}, the quiver for $\tilde{A}(\cal{V})$ has vertices the feasible regions of the hyperplane arrangement
\[
\cal{F} = \left\{\begin{array}{c}
             (++++), (-+++), (--++), (---+), (----), \\
             (+-++),(+--+), (++-+), (+---), (++--), (+++-)
          \end{array} \right\}
\]
There is an edge between feasible regions whenever they differ by exactly one sign.
The $e_{\alpha}$ set equal to zero by A1 are the unbounded feasible regions  $\alpha \in \cal{F}\setminus \cal{B}$ listed in the top row in the list above.
\end{example}

\subsubsection{The $B$ algebras}\label{subsec:defB}

For $S = \{i_1, \ldots, i_m\} \subset \{1,\ldots n\}$, let $u_S := u_{i_1} \cdots u_{i_m} \in \Z[u_1, \ldots, u_n]$, and let $H_S \subset V + \eta$ denote the intersection of the hyperplanes corresponding to elements of $S$. For $\alpha, \beta \in \cal{P}$, set
\begin{equation} \label{eq:Rab}
\tilde{R}_{\alpha \beta} := \frac{\Z[u_1,\ldots,u_n]}{(u_S: S \subset \{1,\ldots,n\} \textrm{ with } \Delta_{\alpha} \cap \Delta_{\beta} \cap H_S = \emptyset)}.
\end{equation}
Let $f_{\alpha,\beta} \in \tilde{R}_{\alpha \beta}$ be the element corresponding to $1 \in \Z[u_1,\ldots,u_n]$. For $\alpha, \beta, \gamma \in \cal{P}$, let $S(\alpha \beta \gamma) = \{i \in \{1,\ldots,n\} : \alpha(i) = \gamma(i) \neq \beta(i)\}$, where $\alpha(i), \beta(i), \gamma(i)$ denote the $i^{th}$ sign of $\alpha, \beta, \gamma$ respectively.

\begin{example}
For the polarized arrangement in Example~\ref{ex:n4k2},
let $\alpha = (+--+)$, $\beta=(+---)$, and $\gamma = (++--)$.  To compute $\tilde{R}_{\alpha \beta}$ observe that $\Delta_{\alpha} \cap \Delta_{\beta}$ is the portion of the line on $H_4$ connecting the vertices $H_1 \cap H_4$ and $H_2 \cap H_4$. This line is disjoint from $H_3$ and $H_1\cap H_2$, so that $\tilde{R}_{\alpha \beta} = \Z[u_1,u_2, u_3, u_4] / (u_1u_2, u_3)$.   One can also check that $\tilde{R}_{\alpha \gamma} = \Z[u_1,u_2,u_3,u_4]/(u_1,u_3)$.
\end{example}

\begin{definition}\label{def:BStyle}
The $\Z$-algebra $\tilde{B}(\cal{V})$ is defined to be
$
\tilde{B}(\cal{V}):= \bigoplus_{\alpha,\beta \in \cal{P}} \tilde{R}_{\alpha,\beta}
$
with multiplication given by
\[
f_{\alpha,\beta} \cdot f_{\beta,\gamma} := u_{S(\alpha \beta \gamma)} f_{\alpha,\gamma}
\]
and extended bilinearly over $\Z[u_1,\ldots,u_n]$.
The algebra $\tilde{B}(\cal{V})$ admits a grading by setting $\deg(f_{\alpha,\beta}) = d_{\alpha,\beta}$, where $d_{\alpha,\beta}$ is the number of sign changes required to turn $\alpha$ into $\beta$, and $\deg(u_i) = 2$. We can view $\tilde{B}(\cal{V})$ as an algebra over $\Z[u_1,\ldots,u_n]$.
\end{definition}

To define the finite-dimensional version $B(\cal{V})$ over $\R$ (or $\Q$ if $\cal{V}$ is rational), write $\R[u_1,\ldots,u_n] = \Sym(\R^n)$ by identifying $u_i$ with the $i^{th}$ coordinate function on $(\R^n)^*$. The inclusion of $V$ into $\R^n$ gives us a ring homomorphism from $\Sym(V)$ into $\R[u_1,\ldots,u_n]$ and thus into the quotient $\tilde{R}^{\R}_{\alpha \beta}$. For $\alpha, \beta \in \cal{P}$ we set
\[
R_{\alpha \beta} := \tilde{R}^{\R}_{\alpha \beta} \otimes_{\Sym(V)} \R,
\]
where the action of $\Sym(V)$ on $\R$ has all elements of $V$ acting as zero, so that $R_{\alpha \beta}$ can be viewed as a further quotient of $\tilde{R}^{\R}_{\alpha \beta}$ by $(c_1 u_1 + \cdots + c_n u_n : (c_1,\ldots, c_n) \in V)$.
We define $B(\cal{V})$ using $R_{\alpha \beta}$ in place of $\tilde{R}_{\alpha \beta}$ in the definition of $\tilde{B}(\cal{V})$. By \cite[Theorem 8.7]{BLPPW},
\begin{equation} \label{eq:Bt-flat}
  \Sym(V) \xrightarrow{j} \tilde{B}(\cal{V}) \xrightarrow{\pi} B(\cal{V})
\end{equation}
is a universal graded flat deformation.

\begin{theorem}[Theorem 4.14 and Corollary 4.15 of \cite{Gale}]
For a polarized arrangement $\cal{V}$, we have graded algebra isomorphisms
$\tilde{B}(\cal{V}) \cong \tilde{A}(\cal{V}^{\vee})$ and $B(\cal{V}) \cong A(\cal{V}^{\vee})$.
\end{theorem}

As a consequence, we have the following description of $\tilde{B}(\cal{V})$.
\begin{proposition}\label{prop:BTildeGensRels}
For a polarized arrangement $\cal{V} = (V,\eta,\xi)$, let $Q$ be the quiver with vertices $e_{\alpha}$ given by $\alpha\in \cal{B}$
 and arrows $p(\alpha,\beta)$ from $\alpha$ to $\beta$ when $\alpha \leftrightarrow \beta$. The algebra $\tilde{B}(\cal{V})$ is $P(Q) \otimes_{\Z} \Z[u_1,\ldots,u_n]$ modulo the two-sided ideal generated by the following relations:
\begin{enumerate}
\item[B1]: $e_{\alpha}$ if $\alpha \in \cal{B}\setminus \cal{F}$, that is $\alpha$ bounded and infeasible,
\item[B2]: $p(\alpha,\beta) p(\beta,\gamma) - p(\alpha,\delta) p(\delta,\gamma)$ for all distinct bounded $\alpha,\beta,\gamma,\delta$ with $\alpha \leftrightarrow \beta \leftrightarrow \gamma \leftrightarrow \delta \leftrightarrow \alpha$,
\item[B3]: $p(\alpha,\beta,\alpha) - u_i e_{\alpha}$ for all bounded $\alpha,\beta$ with $\alpha \leftrightarrow \beta$ via a sign change in coordinate $i$.
\end{enumerate}
The grading on $\tilde{B}(\cal{V})$ defined above matches the one defined as for $\tilde{A}(\cal{V})$.
\end{proposition}

\subsection{Taut paths} \label{sec:taut}
In this section we develop some helpful results on the structure of the algebras $\At(\cal{V})$, largely adapting techniques from \cite[Section 3.2]{Gale}.

\begin{definition}[Definition 3.6 of \cite{Gale}] \label{def:GDKD-3.6}
Given a sequence of elements $\und{\alpha} =(\alpha_1, \dots, \alpha_{\ell})$ of $\{\pm 1\}^n$ and an index $i \in \{1,2, \dots, n\}$,  define
\[
\theta_i(\und{\alpha})
= \left| \{ 1 < j <  \ell  \mid \alpha_j(i) \neq \alpha_{j+1}(i) = \alpha_1(i) \} \right|.
\]
\end{definition}
The number $\theta_i(\und{\alpha})$ counts the number of times a sequence crosses the $i$th hyperplane and returns to the original side.

Let $\cal{V}$ be a polarized arrangement. Let $Q = Q(\cal{V})$ be the quiver defined in Section~\ref{sec:ABAlgebraDefs} with vertex set the set $\cal{F}$ of feasible sign sequences for $\cal{V}$; all paths below will be paths in $Q$. Recall that $\cal{P} \subset \cal{F}$ denotes the set of bounded feasible sign sequences. For a path $\alpha_1 \leftrightarrow \alpha_2 \leftrightarrow \dots \leftrightarrow \alpha_{\ell}$, we write $p(\alpha_1,\ldots,\alpha_{\ell})$ for the element of $\tilde{A}$ represented by the path.

\begin{definition}[Definition 3.7 of \cite{Gale}]
A path $\alpha_1 \leftrightarrow \alpha_2 \leftrightarrow \dots \leftrightarrow \alpha_{\ell}$ is \emph{taut} if it has minimal length among paths from $\alpha_1$ to $\alpha_{\ell}$.  This is equivalent to saying that the sign vectors of $\alpha_1$ and $\alpha_{\ell}$ differ in exactly ${\ell}-1$ entries.
\end{definition}

The next two statements are proved in \cite{Gale} for the algebras $A(\cal{V})$, but the proofs adapt without modification to $\tilde{A}(\cal{V})$.

\begin{proposition}[Proposition 3.8 of \cite{Gale}] \label{prop:GDKD-Prop3.9}
Let $\alpha_1 \leftrightarrow \alpha_2 \leftrightarrow \dots \leftrightarrow \alpha_{\ell}$ be a path.  There is a taut path $\alpha_1 = \beta_1 \leftrightarrow \beta_2 \leftrightarrow \dots \leftrightarrow \beta_d = \alpha_{\ell}$ such that
\[
p(\alpha_1,\dots, \alpha_{\ell}) = p(\beta_1, \dots ,\beta_d) \cdot \prod_{i \in I} t_i^{\theta_i(\alpha_1,\dots, \alpha_{\ell})}
\]
in $\tilde{A}(\cal{V})$, with $\theta_i(\alpha_1,\dots, \alpha_{\ell})$ defined in Definition~\ref{def:GDKD-3.6}.
\end{proposition}

\begin{corollary}[Corollary 3.9 of \cite{Gale}] \label{cor:GDKD-Cor3.9}
Let $\alpha=\alpha_1 \leftrightarrow \alpha_2 \leftrightarrow \dots \leftrightarrow \alpha_d = \beta$ and $\alpha=\beta_1 \leftrightarrow \beta_2 \leftrightarrow \dots \leftrightarrow \beta_d = \beta$ be two taut paths between fixed $\alpha,\beta \in \cal{P}$.  Then
\[
p(\alpha_1,\dots, \alpha_d) = p(\beta_1,\dots,\beta_d).
\]
\end{corollary}

\begin{corollary}\label{cor:GDKD-Cor3.10}
Consider an element
\[
a = p_{\alpha}^{\beta} \cdot \prod_{i \in I} t_i^{d_i} \in e_{\alpha} \tilde{A}(\cal{V}) e_{\beta}
\]
where $p_{\alpha}^{\beta}$ is represented by a taut path from $\alpha$ to $\beta$.
Suppose that $\gamma \in \cal{F}$ satisfies $\gamma(i)=\alpha(i)$ whenever $\alpha(i)=\beta(i)$ and $d_i=0$\footnote{Alternatively, for those hyperplanes where $d_i=0$, if $\alpha$ and $\beta$ are on the same side of the hyperplane, so is $\gamma$.}.  Then $a$ can be written as a monomial in the $t_i$ variables times an element represented by a path in $Q$ that passes through $\gamma$. In particular, if $\gamma \in \cal{F}\setminus \cal{P}$, then $a=0$.
\end{corollary}

\begin{proof}
This is analogous to \cite[Corollary 3.10]{Gale}; we include the proof here for completeness.  By Proposition~\ref{prop:GDKD-Prop3.9}, the composition of a taut path $p_{\alpha}^{\gamma}$ from $\alpha$ to $\gamma$ and $p_{\gamma}^{\beta}$ from $\gamma$ to $\beta$ is equal in $\tilde{A}(\cal{V})$ to $p_{\alpha}^{\beta} \cdot \prod_{i\in I}t_i^{d_i'}$, where $p_{\alpha}^{\beta}$ is a taut path and $d_i'=1$ if $\alpha(i)=\beta(i)\neq \gamma(i)$ and zero otherwise.  Then for $\gamma$ satisfying the hypothesis of the corollary, we always have $d_i \geq d_{i}'$ for all $i$, so that
\[
a = p_{\alpha}^{\beta} \cdot \prod_{i \in I} t_i^{d_i}
= p_{\alpha}^{\gamma} p_{\gamma}^{\delta} \cdot \prod_{i \in I} t_i^{d_i-d_i'}.
\]
\end{proof}

\subsection{Standard modules}\label{sec:HypertoricStandardMods}
%

In this section we write $\tilde{A}$ as shorthand for $\tilde{A}(\cal{V})$; we also write $\At_+$ for the strictly positive-degree elements of $\At$.

For $\alpha \in \cal{P}$, the simple $\At(\cal{V})$-module $L_{\alpha}$ is defined by
\begin{equation}
  L_{\alpha}:= \At/(\At_+ \oplus \Z[e_{\beta} \mid \beta \neq \alpha]).
\end{equation}
We denote by $\Pt_{\alpha}:=\At e_{\alpha}$
the projective cover of $L_{\alpha}$.
Define a submodule $K_{>\alpha}$ of $\Pt_{\alpha}$ by
\begin{equation} \label{eq:Kalpha}
K_{>\alpha} = \sum_{i \in \mathbbm{x}_{\alpha}} \tilde{A}\cdot p(\alpha^i,\alpha),
\end{equation}
where $\mathbbm{x}_{\alpha}=\mu^{-1}(\alpha)$.
Using the partial order defined above, consider the idempotents
\begin{equation}
  \varepsilon_{\alpha} := \sum_{\gamma > \alpha} e_{\gamma},
  \qquad \varepsilon_{\alpha}' = \varepsilon_{\alpha} + e_{\alpha}
\end{equation}
as well as
\begin{equation}
  \overline{\varepsilon}_{\alpha} := \sum_{\gamma \nleq \alpha} e_{\gamma},
  \qquad \overline{\varepsilon}_{\alpha}' = \overline{\varepsilon}_{\alpha} + e_{\alpha}.
\end{equation}

\begin{proposition} \label{prop:varepsilon}
We have
$
K_{>\alpha} =\At\varepsilon_{\alpha}\At e_{\alpha} =\At \overline{\varepsilon}_{\alpha}\At e_{\alpha}.
$
\end{proposition}

\begin{proof}
It is clear that
\[
K_{>\alpha} \subset\At\varepsilon_{\alpha}\At e_{\alpha} \subset\At \overline{\varepsilon}_{\alpha}\At e_{\alpha}.
\]
For $\gamma \nleq \alpha$, we have $\gamma \notin \B_{\mathbbm{x}_{\alpha}}$, so for at least one $i \in \mathbbm{x}_{\alpha}$, there is a sign change from $\alpha$ to $\gamma$ in position $i$. It follows that
$p_{\gamma}^{\a} = p^{\alpha^i}_{\gamma} p^{\alpha}_{\alpha^i} \in K_{> \alpha}$ and thus that $\At \overline{\varepsilon}_{\alpha}\At e_{\alpha} \subset K_{\alpha}$.
\end{proof}

The left
standard module associated to $\alpha$ is given by
\[
\tilde{V}_{\alpha}:= \tilde{P}_{\alpha}/ K_{>\alpha} =\At e_{\alpha}/\At\varepsilon_{\alpha} \At e_{\alpha} = \At e_{\alpha} / \At\overline{\varepsilon}_{\alpha} \At e_{\alpha}.
\]

\begin{remark}\label{rem:VtVersusDelta}
The standard module $\Vt_{\a}$ defined here agrees with the standard module $\Delta(\a)$ defined in \eqref{eq:affine-standard}. To see this, observe that
$\Hom_{\At}(\Pt_{\gamma},\Pt_{\a}) =e_{\gamma}\tilde{A}e_{\alpha}$
so that summing over all $\gamma \nleq \alpha$ produces
$\overline{\varepsilon}_{\alpha}\tilde{A}e_{\alpha}$
and the image in $\tilde{A}$ is just
$\tilde{A}\overline{\varepsilon}_{\alpha}\tilde{A}e_{\alpha}$.
\end{remark}

\begin{remark}[Geometric interpretation of modules]

When $\cal{V}$ is rational, the indecomposable projective, standard, and simple modules $P_{\alpha}$, $V_{\alpha}$, $L_{\alpha}$ over  $B(\cal{V}) \cong A(\cal{V}^{\vee})$ acquire natural geometric interpretations in the hypertoric variety $\M_{\cal{V}}$, see \cite[Proposition 5.22]{Gale}. The modules $\Pt_{\a}, \Vt_{\a}, \tilde{L}_{\a}$ for $\Bt(\cal{V}) \cong \At(\cal{V}^{\vee})$ have similar descriptions using $T^k$-equivariant cohomology; see \cite[Corollary 5.5 \& Section 8]{BLPPW}. In the case of standard modules, let $x_{\alpha}$ denote the toric fixed point whose image under the moment map $\bar{\mu}_{\R}$ is the vertex of $\Delta_{\alpha}$ on which $\xi$ attains its maximum. Then
\begin{align}
V_{\alpha} \cong \bigoplus_{\beta \in \cal{P}} H^{\ast}(\{x_{\alpha}\}\cap X_{\beta})[-d_{\alpha\beta}]
\qquad
\tilde{V}_{\alpha} \cong \bigoplus_{\beta \in \cal{P}} H^{\ast}_{T^k}(\{x_{\alpha}\}\cap X_{\beta})[-d_{\alpha\beta}] \nn
\end{align}
with action given by convolution.

\end{remark}

\subsection{Structure of standard modules}\label{sec:HypertoricStandardStructure}

The study of the infinite dimensional modules $\tilde{V}_{\alpha}$ for $\tilde{A}$ requires new techniques, not adapted from \cite{Gale}, which we develop in this subsection. 


For $\xx \in \mathbb{B}$, we let $\Z[t_{\xx}] := \Z[t_i : i \in \xx]$.

\begin{lemma} \label{lem:Xnozero}
Fix $\alpha \in\cal{P}$ and let $\mathbbm{x}=\mathbbm{x}_{\alpha}$ and $\mathbbm{y}=\mathbbm{x}^{c}$;  we have $\mathbbm{x},\mathbbm{y} \subset I = \{1,\ldots,n\}$ with $|\mathbbm{x}| = k$ and $|\mathbbm{y}| = n-k$.
Let $p^{\alpha}_{\gamma}$ be a taut path from $\gamma \in \cal{F}\cap \cal{B}_{\mathbbm{x}}$ to $\a$.  Then any $X\in \Z[t_{\xx}]$ acts trivially on $p^{\alpha}_{\gamma}$ and any  $Y \in \Z[t_{\yy}]$ acts nontrivially on $p^{\alpha}_{\gamma}$ in $\Vt_{\alpha}$.
\end{lemma}

\begin{proof}
First we show that any $Y \in \Z[t_{\yy}]$ acts nontrivially on $p^{\alpha}_{\gamma}$ in
\[
e_{\gamma}\At(\cal{V})e_{\alpha} \cong e_{\gamma}\Bt(\cal{V}^{\vee}) e_{\alpha}
= \Z[t_1,t_2,\dots, t_n]/ ( t_S\mid S\subset I, \quad H_{S}^{\vee}\cap{\Delta}_{\alpha}^{\vee}\cap \Delta_{\gamma}^{\vee} = \emptyset )
\]
for all $\gamma \in \cal{F}\cap \cal{B}_{\mathbbm{x}}$. Since the ideal in the denominator is a squarefree monomial ideal, we can assume $Y$ is a squarefree monomial in the $t_{\yy}$ variables, i.e. that $Y = t_S$ for some $S \subset \yy$.

Lemma~\ref{lem:cone-adjacent} implies that the only $\gamma' \in \cal{F}$ such that $H^{\vee}_{\mathbbm{y}} \cap \Delta^{\vee}_{\gamma} \neq \emptyset$ are the $\gamma' \in \cal{F}\cap \cal{B}_{\mathbbm{x}}$.
 Since $H_{\mathbbm{y}}^{\vee}$ is the maximum of $\Delta_{\alpha}^{\vee}$ under $\xi^{\vee}$ in $\cal{H}^{\vee}$, this implies $H^{\vee}_{\mathbbm{y}} \cap \Delta_{\alpha}^{\vee} \cap \Delta_{\gamma'}^{\vee} \neq \emptyset$ if and only if $\gamma' \in \cal{F}\cap\cal{B}_{\mathbbm{x}}$,
so that each subset $S \subset \mathbbm{y}$ has $H_{S}^{\vee}\cap \Delta_{\alpha}^{\vee}\cap \Delta_{\gamma}^{\vee} \neq \emptyset$. Hence $Y p^{\alpha}_{\gamma} = t_S p^{\alpha}_{\gamma}$ is nonzero in $e_{\gamma}\At(\cal{V})e_{\alpha}$.

Now let $Y$ be an arbitrary nonzero polynomial in $\Z[t_{\mathbbm{y}(1)}, \dots, t_{\mathbbm{y}(n-k)}]$; it follows that $Y p^{\alpha}_{\gamma}$ is nonzero in $e_{\gamma}\At(\cal{V})e_{\alpha}$. To see that $Y p^{\alpha}_{\gamma}$ is nonzero in $\Vt_{\alpha}$, we show that $Y p^{\alpha}_{\gamma} \notin\At\varepsilon_{\alpha}\At e_{\alpha}$.  Suppose for contradiction that
\[
Y p^{\alpha}_{\gamma}   =\sum_{\beta > \a }Y'_{\beta} p_{\gamma}^{\beta} p_{\beta}^{\a}
=
 \sum_{\beta > \a }Y'_{\beta}\prod_i t_i^{d_i^{\beta}} p_{\gamma}^{\a},
\]
for taut paths $p^{\alpha}_{\beta}$, $p^{\beta}_{\gamma}$ and $Y_{\beta}' \in \Z[t_1,\dots, t_n]$.
 The assumption  $\beta > \a$ implies that for some $i \in \mathbbm{x}$ we have $\alpha(i)=\gamma(i) \neq \beta(i)$. In particular, for each $\beta$ appearing with nonzero coefficient $Y'_{\beta}$ on the right-hand-side, we have $d_i^{\beta} \neq 0$ for some $i\in \mathbbm{x}$.  Hence,
\[
\prod_{i\in I}  t_i^{d_i^{\beta}}  = \prod_{i\in \mathbbm{y}} t_i^{d_i^{\beta}}  \prod_{i\in \mathbbm{x}} t_i^{d_i^{\beta}}
\]
with $\prod_{i\in \mathbbm{x}} t_i^{d_i^{\beta}} \neq 1$. By the nonvanishing in $e_{\gamma}\At(\cal{V})e_{\alpha}$ proved above, we have
\[
Y =  \sum_{\beta > \gamma }Y'_{\beta}\prod_{i\in \mathbbm{y}} t_i^{d_i^{\beta}}  \prod_{i\in \mathbbm{x}} t_i^{d_i^{\beta}}
\]
with at least one $Y'_{\beta} \neq 0$, contradicting that $Y \in \Z[t_{\yy}]$.

Finally, to see that all $X\in \Z[t_{\xx}]$ act by zero, observe that for $i\in \mathbbm{x}$, $t_ip^{\alpha}_{\gamma} =p^{\alpha}_{\gamma}t_i = p^{\alpha}_{\gamma}p(\alpha,\alpha^i,\alpha)$, so that $t_ip^{\alpha}_{\gamma}=0$ in $\Vt_{\alpha}$ since $\alpha^i > \alpha$.
\end{proof}

\begin{proposition} \label{prop:Vtaut}
Let $\mathbbm{y}=\mathbbm{x}_{\alpha}^c$ for some $\alpha \in \cal{P}$.  The standard module $\tilde{V}_{\alpha}$ is a free module over $\Z[t_{\yy}]$ with basis consisting of taut paths to
$\alpha$ from
each element in $\cal{F}\cap \cal{B}_{\mathbbm{x}_{\alpha}}$.
\end{proposition}

\begin{proof}
By Lemma~\ref{lem:Xnozero}, these collections of paths are linearly independent; note that one can deduce independence for paths starting at different points $\gamma$ by left multiplying by idempotents $e_{\gamma}$. Any taut path which originates outside of $\cal{F}\cap \cal{B}_{\mathbbm{x}_{\alpha}}$ must cross a hyperplane $H_i$ for some $i \in \mathbbm{x}_{\alpha}$. By the proof of Corollary~\ref{cor:GDKD-Cor3.10} it can be replaced by a multiple of a path which crosses this hyperplane last, thus lies in $K_{>\alpha}$. Any non-taut path from  an element of $\cal{F}\cap \cal{B}_{\mathbbm{x}_{\alpha}}$ to $\alpha$ that does not cross a hyperplane $H_i$ with $i\in \mathbbm{x}_{\alpha}$, can be written as a $\Z[t_{\yy}]$-multiple of a taut path by Proposition~\ref{prop:GDKD-Prop3.9}.
\end{proof}

It follows that
\begin{equation}\label{eq:StandardMod}
\Vt_{\a} \cong \left( \bigoplus_{\beta \in \cal{F} \cap \cal{B}_{\xx_{\a}}} e_{\beta} \At(\cal{V}) e_{\alpha} \right) / (t_i : i \in \xx_{\a}).
\end{equation}

We want to show that the kernel of the natural map $\Pt_{\a} \to \Vt_{\a}$ has a $\Delta$-filtration whose subquotients are shifts of standard modules $\Vt_{\b}$; to do this, we need some preliminary definitions and results.\footnote{We would like to thank Hugh Thomas for his generous assistance with the following material.} Suppose $H_{\xx}$ and $H_{\xx'}$ are two vertices of the arrangement that are connected by an edge. The maximum of $\xi$ on this edge is attained at one of its two endpoints (say $H_{\xx}$). Then we must have $\xx' < \xx$ by Lemma~\ref{lem:OrderFromCones}, because $H_{\xx'}$ is in the negative cone of $H_{\xx}$ and we thus have $\cal{B}_{\xx} \cap \cal{F}_{\xx'} \neq \emptyset$. It follows that any two vertices connected by an edge correspond to elements of $\mathbb{B}$ that are comparable in the partial order.

\begin{definition}
Let $\gamma \in \cal{P}$, so that $H_{\xx_{\gamma}}$ is a vertex of the arrangement. We say an edge $E$ of the arrangement, with endpoints $\{H_{\xx_{\gamma}}, H_{\xx}\}$, is ascending with respect to $H_{\xx_{\gamma}}$ if $\xx_{\gamma} < \xx$ and descending if $\xx_{\gamma} > \xx$. We say an unbounded edge of the arrangement (with only one endpoint $H_{\xx_{\gamma}}$) is ascending if $\xi$ is unbounded on the edge, and descending otherwise.
\end{definition}

Since we are working with simple arrangements, $H_{\xx_{\gamma}}$ is an endpoint of $2k$ edges of the arrangement. For each $i \in \xx_{\gamma}$, there are two edges $E,E'$ containing $H_{\xx_{\gamma}}$ and contained in $H_i$, and one of $\{E,E'\}$ is ascending while the other is descending. We will say $E$ and $E'$ are opposite to each other.

For any edge $E$ of the arrangement, the two endpoints of $E$ lie on $k-1$ common hyperplanes; each vertex lies on one hyperplane that the other vertex does not lie on.
\begin{definition}
For an edge $E$ of the arrangement with endpoints $\{H_{\xx_{\gamma}},H_{\xx}\}$, let $i = \phi(E)$ denote the unique index such that $i \in \xx_{\gamma} \cap \xx^c$, i.e. such that $H_{\xx_{\gamma}}$ is contained in $H_i$ but $H_{\xx}$ is not.
\end{definition}

Note that $\phi$ restricts to a bijection from the set of ascending edges at $H_{\xx_{\gamma}}$ to the set $\xx_{\gamma}$.

\begin{lemma}\label{lem:SignChangesAscending}
Let $\a,\gamma \in \cal{P}$ and assume that $H_{\xx_{\gamma}} \in \Delta_{\alpha}$. Let $E$ be an edge of the arrangement with one endpoint $H_{\xx_{\gamma}}$; assume $E$ is ascending and let $i = \phi(E)$. Then $E \subset \Delta_{\a}$ if and only if $\alpha(i) \neq \gamma(i)$. Furthermore, any index $i$ with $\alpha(i) \neq \gamma(i)$ is equal to $\phi(E)$ for a unique ascending edge $E$ with one endpoint $H_{\xx_{\gamma}}$.
\end{lemma}

\begin{proof}
First assume $E \subset \Delta_{\alpha}$, and let $H_{\xx_{\gamma'}}$ be the other endpoint of $E$. If $\gamma(i) = \gamma'(i)$, then $H_{\xx_{\gamma'}}$ and thus $E$ are on the same side of the hyperplane $H_i$ as is the region $\Delta_{\gamma}$. But since $E$ is ascending, its opposite edge is descending and thus contained in $\Delta_{\gamma}$, so since $E \cup E'$ is a subset of a straight line, it must be contained in $H_i$. This contradicts the definition of $i = \phi(E)$, so we have $\gamma'(i) \neq \gamma(i)$. Since the convex polytope $\Delta_{\alpha}$ is contained on one side of $H_i$ and has a vertex $H_{\xx_{\gamma'}}$ with $\gamma'(i) \neq \gamma(i)$, we must also have $\alpha(i) \neq \gamma(i)$.

Conversely, assume $\alpha(i) \neq \gamma(i)$ for some index $i$. If $i \notin \xx_{\gamma}$, i.e. $H_i$ is not one of the hyperplanes intersecting to form $H_{\xx_{\gamma}}$, then $H_{\xx_{\gamma}}$ lies on the strictly positive side or strictly negative side of $H_i$. For all points in $\Delta_{\alpha}$ that do not lie on $H_i$, the sign of the point with respect to $H_i$ is $\alpha(i)$, so we have $\gamma(i) = \alpha(i)$, a contradiction. It follows that $i \in \xx_{\gamma}$, so that $i = \phi(E)$ for some edge $E \subset \Delta_{\alpha}$ with one endpoint $H_{\xx_{\gamma}}$.

If $E$ is descending, let its endpoints be $\{H_{\xx_{\gamma}}, H_{\xx_{\gamma'}}\}$. The point $H_{\xx_{\gamma'}}$ is on the strictly positive or strictly negative side of $H_i$, and the sign is given by $\alpha(i)$ since $H_{\xx_{\gamma'}} \in \Delta_{\alpha}$. On the other hand, the sign is also given by $\gamma(i)$ since $E$ is descending and thus contained in $\Delta_{\gamma}$. We get $\alpha(i) = \gamma(i)$, a contradiction, so $E$ is ascending. Uniqueness in the statement of the lemma follows from the injectivity of $\phi$ on ascending edges.
\end{proof}

\begin{corollary}\label{cor:AscendingEdgesAlphaBeta}
Let $\a,\b,\gamma \in \cal{P}$ with $H_{\xx_{\gamma}} \in \Delta_{\alpha} \cap \Delta_{\beta}$. The number of ascending edges with respect to $H_{\xx_{\gamma}}$ that are contained in $\Delta_{\alpha} \cap \Delta_{\beta}$ is equal to $S_{\a \gamma \b}$ as defined in Section~\ref{subsec:defB}, i.e. the number of indices $i$ with $\alpha(i) = \beta(i) \neq \gamma(i)$.
\end{corollary}

\begin{proof}
By Lemma~\ref{lem:SignChangesAscending}, the ascending edges with one vertex $H_{\xx_{\gamma}}$ that are contained in both $\Delta_{\alpha}$ and $\Delta_{\beta}$ are in bijection with the set of all indices $i$ such that $\alpha(i) = \beta(i) \neq \gamma(i)$.

\end{proof}

\begin{lemma}\label{lem:Kfiltration}
The kernel of the map $\Pi_{\alpha} \maps \Pt_{\alpha} \to \Vt_{\alpha}$ has a filtration whose subquotients are isomorphic to $\Vt_{\beta}$ for those $\beta = \beta_{\mathbbm{x}}$ such that $\alpha \in \cal{B}_{\mathbbm{x}}$ (note that $\beta > \alpha$ for all such $\beta$), each with multiplicity one, and with grading shifts given by the number $d_{\alpha, \beta}$ of sign changes between $\alpha$ and $\beta$.

\end{lemma}

\begin{proof}
The kernel of $\Pi_{\alpha} \maps \Pt_{\alpha} \to \Vt_{\alpha}$ is given by $K_{>\alpha}$ from \eqref{eq:Kalpha}.  We define a filtration on $K_{>\alpha}$. Choose a total ordering $\{ \gamma \in \cal{P} : \gamma > \alpha \} = \{\gamma_1, \ldots, \gamma_{\ell}\}$ such that $\gamma_i < \gamma_j$ only if $i < j$. Write
\[
K_{\alpha} = \At \sum_{i=1}^{\ell} e_{\gamma_i} \At e_{\alpha} \supset \At \sum_{i=2}^{\ell} e_{\gamma_i} \At e_{\alpha} \supset \cdots \supset \At e_{\gamma_{\ell}} \At e_{\alpha} \supset 0.
\]
Let
\[
M_{\alpha}^{\gamma_j} := \frac{\At \sum_{i=j}^{\ell} e_{\gamma_i} \At e_{\alpha}}{\At \sum_{i=j+1}^{\ell} e_{\gamma_i} \At e_{\alpha}}
\]
be the subquotients of this filtration for $1 \leq j \leq \ell$. It suffices to show that for each $j$, $M_{\alpha}^{\gamma_j}$ is zero when $\alpha \notin \cal{B}_{\xx_{\gamma_j}}$ and isomorphic to $q^{d_{\alpha,\gamma_j}} \Vt_{\gamma_j}$ otherwise.

Fix $j$ between $1$ and $\ell$ and let $\mathbbm{x}=\mathbbm{x}_{\gamma_j}=\mu^{-1}(\gamma_j)$. If $\alpha$ is not in the bounded cone $\cal{B}_{\mathbbm{x}}$, then there exists $i\in \mathbbm{x}$ such that $\alpha(i)\neq \gamma_j(i)$.  For this $i$, we have $\gamma_j<\gamma_j^{i}$ and any path from $\gamma_j$ to $\a$ can be written as a path that passes through $\gamma_j^i$. Since $\gamma_j < \gamma_j^i$, we have $\gamma_j^i = \gamma_{j'}$ for some $j' > j$, so $M_{\alpha}^{\gamma_j} = 0$.

If $\alpha \in\cal{F} \cap \cal{B}_{\mathbbm{x}}$, then postcomposition with a taut path $p_{\gamma_j}^{\a}$ from $\gamma_j$ to $\a$ defines a degree-zero map
\[
q^{d_{\alpha, \gamma_j}} \Pt_{\gamma_j} \rightarrow \At \sum_{i=j}^{\ell} e_{\gamma_i} \At e_{\alpha} \twoheadrightarrow M_{\alpha}^{\gamma_j}.
\]
This map annihilates $\At\varepsilon_{\gamma_j}\At e_{\gamma_j}$ since for any path passing through $\gamma' > \gamma_j$ and ending at $\gamma_j$, we must have $\gamma' = \gamma_{j'}$ for some $j' > j$. Thus, we get a map on the quotient
\begin{align}
 \phi\maps q^{d_{\alpha, \gamma_j}} \Vt_{\gamma_j} = &\Pt_{\gamma_j}/\At\varepsilon_{\gamma_j}\At e_{\gamma_j}
 \longrightarrow M_{\alpha}^{\gamma_j} \nn
\\
& [\delta \to\gamma_j] \mapsto [ \delta \to\gamma_j \stackrel{p_{\gamma_j}^{\a}}{\to}\alpha] , \qquad \text{$\delta \in \cal{F} \cap \cal{B}_{\mathbbm{x}}$.}   \nn
\end{align}

Now, $M_{\alpha}^{\gamma_j}$ is spanned as a $\Z[t_1,\dots, t_n]$-module by paths through $\gamma_j$ ending at $\a$, and by Proposition~\ref{prop:GDKD-Prop3.9} any such path can be written as $p'' p'$ where $p'$ is a taut path from $\gamma_j$ to $\a$. By Corollary~\ref{cor:GDKD-Cor3.9} we can assume the taut path $p'$ is $p_{\gamma_j}^{\a}$.  Hence, the map $\phi$ is surjective.

For injectivity, it is enough to show that
\[
\dim_q \ker \Pi_{\alpha} = \sum_{\gamma>\alpha \,\, : \,\, \a \in \cal{B}_{\xx_{\gamma}}}\dim_q q^{d_{\alpha, \gamma}} \Vt_{\gamma},
\]
or
\begin{align*}
\dim_q \Pt_{\alpha}
    &= \dim_q {\rm im} \Pi_{\alpha} \oplus \dim_q \ker\Pi_{\alpha}
    = \dim_q \Vt_{\alpha} \oplus \sum_{\gamma>\alpha \,\, : \,\, \a \in \cal{B}_{\xx_{\gamma}}} q^{d_{\alpha, \gamma}} \dim_q\Vt_{\gamma} \\
    &= \sum_{\mathbbm{x} : \alpha \in \cal{B}_{\mathbbm{x}}} q^{d_{\alpha, \mu(\xx)}} \dim_q\Vt_{\mu(\mathbbm{x})}
\end{align*}
(we implicitly pass to $\Q$ before taking graded dimensions).
Since surjectivity establishes one inequality, it is sufficient to show we have an equality when we sum over $\alpha$,
\[
\dim_q \At =  \sum_{\alpha \in \cal{P}} \sum_{\mathbbm{x}: \alpha \in \cal{B}_{\mathbbm{x}}} q^{d_{\alpha, \mu(\xx)}} \dim_q\Vt_{\mu(\mathbbm{x})}.
\]
Note that
\[
\dim_q \At = \sum_{\alpha,\beta \in \cal{P}^{\vee}} \dim_q \tilde{R}^{\vee}_{\alpha,\beta}.
\]
By \cite[Lemma 4.2]{Gale}, $\tilde{R}_{\alpha,\beta}^{\vee}$ is a free ${\rm Sym}(V^{\perp})$-module whose rank is the number of $\mathbbm{x} \in \mathbb{B}$ such that $H_{\mathbbm{x}^c}^{\vee} \subset \Delta_{\alpha}^{\vee} \cap \Delta_{\beta}^{\vee}$. We have $\dim_q {\rm Sym}(V^{\perp}) = \frac{1}{(1-q^2)^{n-k}}$; the generator of the Stanley--Reisner ring $\tilde{R}_{\alpha,\beta}^{\vee}$ over $\Sym(V^{\perp})$ corresponding to a given $\xx$ has degree given as follows.

Let $D$ be the dual polytope of $\Delta_{\a}^{\vee} \cap \Delta_{\b}^{\vee}$ (called the polar polytope in \cite{ZieglerPoly}). By \cite[Exercise 8.10]{ZieglerPoly}, if we make a small modification to our objective function $\xi$ so that the partial order is preserved but the modified $\xi$ now takes distinct values on all vertices of $\Delta_{\a}^{\vee} \cap \Delta_{\b}^{\vee}$, we get a shelling of $D$ such that the ordering on facets of $D$ (vertices of $\Delta_{\a}^{\vee} \cap \Delta_{\b}^{\vee}$) is compatible with our usual partial order.

Now, the degree of the generator of $\tilde{R}_{\alpha,\beta}^{\vee}$ corresponding to $\xx$, or to the vertex $H_{{\xx}^c}^{\vee}$, is encoded in the contribution of this vertex (or facet $F$ of $D$) to the $h$-vector of $D$. By \cite[Theorem 8.19]{ZieglerPoly} this contribution is the size of the ``restriction'' of $F$ as defined in \cite[Section 8.3]{ZieglerPoly}, so the degree of the corresponding generator is $d_{\a, \b}$ plus the size of this restriction (note that the minimal-degree element of $\tilde{R}_{\alpha,\beta}^{\vee}$ has degree $d_{\a,\b}$).

Since the polytope $\Delta_{\a}^{\vee} \cap \Delta_{\b}^{\vee}$ is simple, its dual polytope $D$ is simplicial (see \cite[Proposition 2.16]{ZieglerPoly}). Thus, by \cite[Exercise 8.10]{ZieglerPoly}, the degree of the generator corresponding to $\xx$ is $d_{\a,\b}$ plus the number of ascending edges in $\Delta_{\a}^{\vee} \cap \Delta_{\b}^{\vee}$ with respect to $H_{{\xx}^c}^{\vee}$ (note that since $\Delta_{\a}^{\vee} \cap \Delta_{\b}^{\vee}$ is simple, the number of ascending edges in $\Delta_{\a}^{\vee} \cap \Delta_{\b}^{\vee}$ and the number of descending edges in $\Delta_{\a}^{\vee} \cap \Delta_{\b}^{\vee}$ at any vertex add up to the dimension of $\Delta_{\a}^{\vee} \cap \Delta_{\b}^{\vee}$). Corollary~\ref{cor:AscendingEdgesAlphaBeta} then implies that the degree of the generator corresponding to $\xx$ is $d_{\a,\b} + S_{\a \gamma \b}$.

We see that
\[
\dim_q \At = \frac{1}{(1-q^2)^{n-k}} \sum_{ \left\{ (\alpha, \beta, \mathbbm{x}) \in \cal{P} \times \cal{P} \times \mathbb{B} \,\, | \,\, H_{\mathbbm{x}^c}^{\vee} \subset \Delta_{\alpha}^{\vee} \cap \Delta_{\beta}^{\vee} \right\} } q^{d_{\a,\b} + S_{\a \gamma \b}}.
\]
On the other hand, for $\mathbbm{x} \in \mathbb{B}$, $\Vt_{\mu(\mathbbm{x})}$ is a free $\Z[t_{\yy}]$-module whose rank is the number of elements $\beta \in \cal{F} \cap \cal{B}_{\mathbbm{x}}$ by Proposition~\ref{prop:Vtaut}, and $\dim_q \Z[t_{\yy}]$ is also $\frac{1}{(1-q^2)^{n-k}}$. Thus,
\begin{align*}
\sum_{\alpha \in \cal{P}} \sum_{\mathbbm{x} : \alpha \in \cal{B}_{\mathbbm{x}}} q^{d_{\alpha, \mu(\xx)}} \dim_q\Vt_{\mu(\mathbbm{x})}
&= \frac{1}{(1-q^2)^{n-k}}  \sum_{ \left\{ (\alpha, \beta, \mathbbm{x}) \in \cal{P} \times \cal{P} \times \mathbb{B} \,\, |\,\, \alpha, \beta \in \cal{B}_{\mathbbm{x}} \right\} } q^{d_{\alpha, \mu_{\xx}}} q^{d_{\mu_{\xx},\beta}} \\
&= \frac{1}{(1-q^2)^{n-k}}  \sum_{ \left\{ (\alpha, \beta, \mathbbm{x}) \in \cal{P} \times \cal{P} \times \mathbb{B} \,\, |\,\, \alpha, \beta \in \cal{B}_{\mathbbm{x}} \right\} } q^{d_{\alpha, \beta} +  S_{\a \gamma \b}}.
\end{align*}
Injectivity follows from Lemma~\ref{lem:cone-adjacent}.
\end{proof}

\begin{proposition} \label{prop:endV}
Let $\mathbbm{y}=\mathbbm{x}_{\alpha}^c$ for some $\alpha \in \cal{P}$.  The endomorphism algebra $\End_{\At}(\tilde{V}_{\alpha})$ is the graded polynomial ring $\Z[t_{\yy}]$.
\end{proposition}

\begin{proof}
   For each $\gamma \in \cal{F}\cap \cal{B}_{\mathbbm{x}}$,  fix a choice of taut path $p^{\alpha}_{\gamma}$ from $\gamma$ to $\alpha$. By  Lemma~\ref{lem:Xnozero},
 the set $\{ p^{\alpha}_{\gamma} \mid \gamma \in \cal{F}\cap \cal{B}_{\mathbbm{x}}\}$ forms a basis for $\tilde{V}_{\alpha}$ as a free module over  $\Z[t_{\yy}]$ and any $t_i$ for $i \in \mathbbm{x}_{\alpha}$ acts by zero.
Any $\At$-module homomorphism $f\maps \tilde{V}_{\alpha} \to \tilde{V}_{\alpha}$ is determined by the image of these basis elements.

Let $f \maps \tilde{V}_{\alpha} \to \tilde{V}_{\alpha}$ be an $\At$-module map, so that for $\delta \in \cal{F}\cap B_{\mathbbm{x}_{\alpha}}$ we have
\[
f(p^{\alpha}_{\delta}) = \sum_{\beta \in \cal{F}\cap B_{\mathbbm{x}_{\alpha}}}c_{\beta}p^{\alpha}_{\beta},
\]
for some $c_{\beta} \in \Z[t_{\yy}]$.
Acting on the left by $e_{\beta}$, we see that $c_{\beta}=0$ unless $\beta = \delta$, so that any $\At$-module map must send $p^{\alpha}_{\delta}$ to a multiple of itself.
For $\delta, \gamma \in \cal{F}\cap B_{\mathbbm{x}_{\alpha}}$ let $p^{\delta}_{\gamma}$ be a taut path from $\gamma$ to $\delta$.  Then $f$ being an $\At$ module map implies the following are equal:
\begin{align}
  f\left(p^{\delta}_{\gamma} \cdot p^{\alpha}_{\delta}\right )
    &= f\left(\prod_{i \in I} t_i^{\theta_i} p^{\alpha}_{\gamma} \right)
    = c_{\gamma} \prod_{i \notin \mathbbm{x}_{\alpha}} t_i^{\theta_i} p^{\alpha}_{\gamma}
    \\
    p_{\gamma}^{\delta} \cdot f(p^{\alpha}_{\delta}) &= c_{\delta}p^{\delta}_{\gamma}p^{\alpha}_{\delta}
= c_{\delta} \prod_{i \notin \mathbbm{x}_{\alpha}} t_i^{\theta_i}p^{\alpha}_{\gamma}
\end{align}
with $\theta_i$ as in Proposition~\ref{prop:GDKD-Prop3.9}. Proposition~\ref{prop:Vtaut} then implies that $c_{\delta}=c_{\gamma}$ for all $\delta, \gamma \in \cal{F}\cap B_{\mathbbm{x}_{\alpha}}$.  In particular, $f$ is given by multiplication by an element in $\Z[t_{\yy}]$, all of which act nontrivially.

\end{proof}

\section{Hypertoric algebras are affine quasi hereditary}\label{sec:HypertoricAlgAreAQH}

In this section we show that the algebras $\At(\cal{V})$ are affine (polynomial) quasi hereditary, giving their Grothendieck groups the rich structure of Section~\ref{sec:Grothendieck-affine}. We also show that the distinguished classes of modules arising from the affine quasi hereditary structure allow us to  define a ``canonical'' and ``dual canonical'' basis for $K_0(\At(\cal{V}))$.

\subsection{Hypertoric convolution algebras are affine quasi hereditary}
We show that the categories $\At(\cal{V})\gmod$ of finitely generated graded left
$\At(\cal{V})$-modules are polynomial highest weight categories. Since $\At(\cal{V})$ has a finite number of simple objects, this is equivalent to $\At(\cal{V})$ being affine (polynomial) quasi hereditary. Recall that in \cite[Section 5]{Gale} it is shown that the undeformed algebras $A(\cal{V})$ are quasi-hereditary algebras~\cite[Theorem 5.23]{Gale}, so that their categories of modules are highest weight categories. The results in this section can be viewed as an infinite-dimensional analogue of this result.

For this section we work over an algebraically closed field $\Bbbk$, rather than over $\Z$.

\begin{lemma}
The algebra $\At(\cal{V})$ is (graded) Noetherian Laurentian, so that $\At(\cal{V})\gmod$ is a Noetherian Laurentian category.
\end{lemma}

\begin{proof}
Since $\At(\cal{V})$ is positively graded and finite-dimensional in each degree, it is a Laurentian algebra; to show $\At(\cal{V})$ is left (graded) Noetherian, it suffices to show $\At(\cal{V})$ is left Noetherian as an ungraded ring. Indeed, consider the subring $S = \Pi_{\a \in \cal{P}} e_{\a} \At(\cal{V}) e_{\a}$ of $\At(\cal{V})$. Each factor $e_{\a} \At(\cal{V}) e_{\a}$ is a quotient of a polynomial ring in $n$ variables and is thus Noetherian, so the direct product $S$ is Noetherian. Furthermore, $\At(\cal{V})$ is finitely generated (by a set of chosen taut paths from $\a$ to $\b$ for all $\a,\b \in \cal{P}$) as a left $S$-module. By \cite[Corollary 1.5]{GoodearlWarfield}, $\At(\cal{V})$ is left Noetherian.
\end{proof}

 Recall that for $\alpha , \beta \in \cal{F}$ such that $\Delta_{\alpha}$ and $\Delta_{\beta}$ are related by crossing a single hyperplane $H_i$,  we write $\beta=\alpha^i$.

\begin{lemma} \label{lem:find-neighbor}
Let $\mathbbm{x}=\mu^{-1}(\beta)$  for $\beta \in \cal{P}$.  Suppose that for $\alpha \in \cal{P}$ there does not exist an $i \in \mathbbm{x}^c$ such that $\alpha^i \in \cal{F}$ and either $\alpha^i$ is unbounded or $\alpha < \alpha^i$. Then $\alpha=\beta$.
\end{lemma}

\begin{proof}
If there is no $i \in \mathbbm{x}^c$ such that $\alpha^i \in \cal{F}$ and $\alpha^i$ is unbounded or $\alpha < \alpha^i$, then passing to the deletion arrangement $\cal{V}_{\mathbbm{x}^c}$, defined in \cite[Section 2.4]{Gale}, the projection of $\alpha$ to   $\cal{V}_{\mathbbm{x}^c}$ remains bounded.  But this means that $\alpha \in \cal{F}\cap \cal{B}_{\mathbbm{x}}$;  since the unique $\gamma \in \cal{F}\cap \cal{B}_{\mathbbm{x}}$ such that $\gamma \nless \gamma^i$ for all $i \in \xx^c$ is $\gamma = \beta$, this implies $\alpha=\beta$.
\end{proof}

\begin{theorem} \label{thm:affine}
Let    $\Pi=\cal{P}$
with partial order defined in Section~\ref{subsec:partial_order}.
The data $(\tilde{A}(\cal{V})\gmod, \Pi,\leq)$
defines a polynomial highest weight category with standard modules $\{\Vt_{\alpha}\}$.
\end{theorem}

\begin{proof}
$(\mathsf{PHW1})$:   This is proven in Lemma~\ref{lem:Kfiltration}.
\smallskip

\noindent $(\mathsf{PHW2})$: Proposition~\ref{prop:endV} implies that $\End(\Vt_{\alpha})$ is a positively graded polynomial ring.

\smallskip
\noindent $(\mathsf{PHW3})$:
Let $\mathbbm{y}=\mathbbm{x}_{\beta}^c$.
Fix a choice of taut path $p^{\alpha}_{\gamma}$ from $\gamma$ to $\a$ for each $\alpha,\gamma \in \cal{P}$ (some of which will be zero in $\At(\cal{V})$). The set $\{ p^{\beta}_{\delta} \mid \delta \in \cal{F}\cap\cal{B}_{\mathbbm{x}_{\beta}}\}$ forms a basis of $\Vt_{\beta}$ as a free $\Z[t_{\yy}]$-module and any $t_i$ for $i \in \mathbbm{x}_{\beta}$ acts by zero on $\Vt_{\beta}$.

Now, any $\At$-module map $f\maps \Pt_{\alpha} \to \Vt_{\beta}$ sends the generator $e_{\alpha}$ of $\Pt_{\alpha}$ to some element of $\Vt_{\beta}$ which is equal to its product on the left with $e_{\alpha}$, and any such element of $\Vt_{\beta}$ uniquely determines an $\At$-module map $f\maps \Pt_{\alpha} \to \Vt_{\beta}$. If $\alpha \notin \cal{F}\cap \cal{B}_{\mathbbm{x}_{\beta}}$, all such elements of $\Vt_{\beta}$ are zero; otherwise, these elements $c p_{\alpha}^{\beta}$ for $c \in \Z[t_1,\ldots,t_n]$ form a rank one $\Z[t_1,\ldots,t_n]$-submodule of $\Vt_{\beta}$.

\end{proof}

Theorem~\ref{thm:highestwt-qh} and the fact that $\Bt(\cal{V}) \cong \At(\cal{V}^{\vee})$ then give us the following corollary.

\begin{corollary}\label{cor:ABPQH}
For a polarized arrangement $\cal{V}$, the algebras $\At(\cal{V})$ and $\Bt(\cal{V})$ are polynomial quasi hereditary.
\end{corollary}

The proof of Theorem~\ref{thm:affine} also implies the following corollary that is slightly stronger than the usual affine highest weight condition of $\Hom(\Vt_{\alpha},\Vt_{\beta})=0$ for $\alpha \nleq \beta$.
\begin{corollary}
For any $\beta \in \cal{P}$ and $\alpha\in \cal{F}\setminus \cal{B}_{\mathbbm{x}_{\beta}}$, then
$
\Hom(\Vt_{\alpha},\Vt_{\beta})=0.
$
\end{corollary}

By Kleshchev~\cite[Prop 5.6]{Klesh-affine}, in a highest weight category the proper standard module can be computed as
\[
\bar{\Delta}_{\alpha} \cong \Delta(\alpha)/\Delta(\alpha) {\rm rad}(B_{\alpha}).
\]
For $\At(\cal{V})$ the endomorphism algebra $B_{\alpha}$ of $\Vt_{\alpha}$ is isomorphic to the polynomial ring $\Z[t_{\yy}]$ for $\mathbbm{y} = \mathbbm{x}_{\alpha}^c$ by Proposition~\ref{prop:endV}.  Hence, ${\rm rad}(B_{\alpha})$ is the ideal of strictly positive degree polynomials.  This implies
$
\bar{\Delta}_{\alpha}
= \Vt_{\alpha} / \Vt_{\alpha} \cdot (\Z[t_{\yy}])_{>0}.
$
Proposition~\ref{prop:Vtaut} then implies that the proper standard module $\bar{\Delta}_{\alpha}$ for $\At(\cal{V})$ is a free $\Z$-module with basis the taut paths to $\alpha$ from elements of $\cal{F}\cap \cal{B}_{\mathbbm{x}_{\alpha}}$.  Then \cite[Lemma 5.21]{Gale} implies the following.

\begin{corollary}
After passing to $\R$ or $\Q$ as appropriate, the proper standard module $\bar{\Delta}_{\alpha}$ of $\At(\cal{V})$ is isomorphic to the inflation of the standard module $V_{\alpha}$ for $A(\cal{V})$   along the projection map
$
\At(\cal{V}) \to A(\cal{V}).
$
\end{corollary}

\subsection{Further properties of hypertoric convolution algebras} \label{subsec:affine-cellular}

\begin{definition} \label{def:hypertoric-involution}
The algebra $\At(\cal{V})$  has a canonical anti-automorphism $\psi^{\cal{V}}$ fixing idempotents and taking $p(\alpha,\beta)$ to $p(\beta,\alpha)$ for all $\alpha \leftrightarrow \beta$ in $\cal{P}$.  Under the identification $\At(\cal{V}) \cong \Bt(\cal{V}^{\vee})$ this corresponds to identifying $\tilde{R}_{\alpha\beta}^{\vee}$ with $\tilde{R}_{\beta\alpha}^{\vee}$ for all $\alpha,\beta \in \cal{P}=\cP^{\vee}$, where $\tilde{R}_{\alpha\beta}^{\vee}$ is defined as in \eqref{eq:Rab} using $\cV^{\vee}$.
\end{definition}

The following is immediate from Corollary~\ref{cor:affineQHisSmooth} and Corollary~\ref{cor:ABPQH}.
\begin{corollary} \label{cor:hypSmooth}
The algebras $\At(\cal{V})$ and $\Bt(\cal{V})$ are homologically smooth.
\end{corollary}

Note that because all the simple $\At(\cal{V})$-modules $L(\alpha)$ are one dimensional, concentrated in degree zero, the anti-automorphism $\psi^{\cal{V}}$ is balanced in the sense of Section~\ref{subsec:affine-cellular}.

\begin{corollary}\label{cor:hyp-affine-cellular}
The algebras $\At(\cal{V})$ and $\Bt(\cal{V})$ are affine cellular.
\end{corollary}

Another consequence of Theorem~\ref{thm:affine} is that our algebras are positively graded standardly stratified algebras in the sense of \cite{Maz-standardly} (see Section~\ref{sec:ProperCostandardTilting}). This implies the following result.

\begin{theorem}[cf. Theorem 14 of \cite{Maz-standardly}]
Consider the positively graded standardly stratified algebra $\At(\cal{V})$ with associated category $\cal{T}$ of tilting modules.
 \begin{enumerate}
   \item The category $\cal{T}$ is closed with respect to direct sums and direct summands.
   \item For every  $\alpha \in \cal{P}$
   there is a unique indecomposable object $T(\alpha) \in \cal{T}$ such that there is a short exact sequence
\[
\Delta(\alpha) \hookrightarrow T(\alpha) \twoheadrightarrow {\rm Coker}
\]
with ${\rm Coker}\in \mathsf{F}(\Delta)$.
\item Every indecomposable object in $\cal{T}$ has the form $q^s T(\alpha)$ for some  $\alpha \in \cal{P}$
and $s \in \Z$.
 \end{enumerate}
 In particular,   $(T(\alpha):\Delta(\alpha))_q=1$ and $(T(\alpha):\Delta(\beta))_q=0$ for any $\beta \nleq \alpha$.
\end{theorem}

Tilting modules for affine highest weight categories satisfying an additional assumption are studied in \cite[Lemma 3.8]{Fuj-tilt}.  By the results above, the categories $\At(\cal{V})\gmod$ satisfy this additional assumption, so that the results in that paper also apply to $\At(\cal{V})$.

\subsection{Grothendieck groups of algebras \texorpdfstring{$\Bt(\cal{V})$}{B-tilde(V)} and \texorpdfstring{$B(\cal{V})$}{B(V)}}
Recall that the algebras $\Bt(\cal{V})$ are graded flat deformations
\[
S:={\rm Sym}(V) \longrightarrow \Bt(\cal{V}) \longrightarrow B(\cal{V})
\]
of finite-dimensional algebras $B(\cal{V})$ over ${\rm Sym}(V)$.   In \cite[Section 5]{Gale} it is shown that the algebras $B(\cal{V})$ are quasi hereditary   with finite-dimensional projective, standard, costandard, and simple modules
$(P_{\alpha}, V_{\alpha}, \Lambda_{\alpha}, L_{\alpha})$ for $\alpha \in \cal{P}$.
Then we have functors
\[
\pr \maps \Bt(\cal{V})\gmod \to B(\cal{V})\gmod, \qquad \infl \maps B(\cal{V})\gmod \to \Bt(\cal{V})\gmod
\]
where, for $V\in \Bt(\cal{V})\gmod$, we write $\pr(V)$ (the projection of $V$) for the $B(\cal{V})$-module
$V \otimes_{\Bt(\cal{V})} B(\cal{V})$, and for $W\in B(\cal{V})\gmod$, we write $\infl(W)$ (the inflation of $W$) for the $\Bt(\cal{V})$-module with underlying vector space $W$ on which elements of $\Bt(\cal{V})$ act by their images under the quotient map $\Bt(\cal{V}) \to \B(\cal{V})$.  We have that $\pr \circ \infl = \Id$ and that $\pr$ is left adjoint to $\infl$; the functors restrict to give
\[
\pr \maps \Bt(\cal{V})\pmod \to B(\cal{V})\pmod, \qquad \infl \maps B(\cal{V})\fmod \to \Bt(\cal{V})\fmod.
\]
The functor $\pr$ commutes with the duality $\#$ from Section~\ref{subsec:bilinearform}, while $\infl$ commutes with the duality $\circledast$ from  Section~\ref{subsec:QHinvolution}.

The simple modules for $\Bt(\cal{V})$ and $B(\cal{V})$, denoted $L_{\alpha}$, are both one-dimensional. Let $\bar{\Vt}_{\alpha}$ and $\bar{\tilde{\nabla}}_{\alpha}$ denote the proper standard and proper costandard modules over $\Bt(\cal{V})$ associated to $\alpha$. We have that
\[
\pr \Pt_{\alpha} = P_{\alpha},
\quad \pr \Vt_{\alpha} = \pr \bar{\Vt}_{\alpha} = V_{\alpha},
\quad \pr \bar{\tilde{\nabla}}_{\alpha} = \Lambda_{\alpha},
\quad \pr L_{\alpha} = L_{\alpha},
\]
\[
\infl P_{\alpha} =\Pt_{\alpha}/{\rm rad}({\rm Sym}(V)),
\quad \infl V_{\alpha} = \bar{\Vt}_{\alpha},
\quad \infl \Lambda_{\alpha} = \bar{\tilde{\nabla}}_{\alpha},
\quad \infl L_{\alpha} = L_{\alpha}.
\]
The above discussion implies the following proposition.
\begin{proposition} \label{prop:pr-infl}
The functors $\pr \maps \Bt(\cal{V}) \to B(\cal{V})$ and $\infl \maps B(\cal{V}) \to \Bt(\cal{V})$ induce isomorphisms
\begin{equation}
[\pr] \maps K_0(\Bt(\cal{V})) \to K_0(B(\cal{V})), \qquad
[\infl] \maps G_0(B(\cal{V})) \to G_0(\Bt(\cal{V})).
\end{equation}
Under this isomorphism,
\[
(-,-)_{\Bt(\cal{V})} = \dim_q({\rm Sym}(V))\cdot  (-,-)_{B(\cal{V})} = \frac{1}{(1-q^2)^k}\cdot  (-,-)_{B(\cal{V})}
\]
where $(-,-)_{\Bt(\cal{V})}$ denotes the form \eqref{eq:form-HOM} on $K_0(\Bt(\cal{V}))$  and $(-,-)_{B(\cal{V})}$ the analogous form on $K_0(B(\cal{V}))$.
\end{proposition}

\begin{remark}
The inflation functor produces $\Bt(\cal{V})$-modules from $B(\cal{V})$-modules, but it cannot produce the standard modules over $\Bt(\cal{V})$ (standard modules over $\Bt(\cal{V})$ are infinite-dimensional whereas the corresponding modules over $B(\cal{V})$ are finite-dimensional, and the inflation functor preserves the underlying vector spaces of the modules it acts on).
\end{remark}

\begin{remark} \label{rem:form-compare}
In Section~\ref{sec:bases-gl} we identify $K_0(\Bt(\cal{V}))$ for a certain choice of polarized arrangement with a weight space of the $U_q(\mf{gl}(1|1))$-module $V^{\otimes n}$.  Under this identification, the form coming from the finite-dimensional algebra $B(\cal{V})$ agrees with the form on $V^{\otimes n}$ coming from the Hecke algebra under super Schur-Weyl duality, see for example \cite[Section 3]{Sar-tensor}.
\end{remark}

\subsubsection{Grothendieck groups and cohomology}\label{sec:BtGGCohomology}

It follows from the above discussion that $K_0(\Bt(\cal{V}))$ and $K_0(B(\cal{V}))$ are free $\Z[q,q^{-1}]$-modules of the same rank, with bases given by $\{[\Pt_{\alpha}]\}$ and $\{[P_{\alpha}]\}$ respectively. Similarly, the Grothendieck group of finitely generated (equivalently, finite-dimensional) ungraded $B(\cal{V})$-modules is a free $\Z$-module with basis given by $\{[P_{\alpha}]\}$. Identifying this latter Grothendieck group with $K_0(\Bt(\cal{V})) \otimes_{\Z[q,q^{-1}]} \Z$ where $q$ acts as $1$ on $\Z$ and we identify $[\Pt_{\alpha}]$ with $[P_{\alpha}]$, it follows from \cite[Theorem 1.2(5)]{HypertoricCatO} that if the hypertoric variety $\M_{\cal{V}}$ is smooth then
\[
K_0(\Bt(\cal{V})) \otimes_{\Z[q,q^{-1}]} \C \cong H^{2k}_{\mathbb{T}}(\M_{\cal{V}}; \C).
\]
The class of $[\Pt_{\alpha}]$ is identified with the equivariant cohomology class of $X_{\alpha}$ for $\alpha \in \cal{P}$.

\subsection{Canonical bases for hypertoric convolution algebras} \label{sec:can-aqh}
Recall that $\bar{}$ denotes the bar involution of $\Z[q,q^{-1}]$ with $\bar{q}=q^{-1}$.
Following Webster~\cite{Web-canonical}\footnote{Webster works with sesquilinear forms $\la,\ra$.  Here we prefer to use bilinear forms, so our presentation is adapted accordingly.  The sesquilinear form $\la-,-\ra$ is given from the bilinear form $(-,-)$ by \[
\la u,v \ra  = ( \psi(u),v )  = \overline{( u,\psi^{\ast}(v) ) }
\]

} , let $V$ be a free $\Z[q,q^{-1}]$-module. A \emph{pre-canonical structure} on $V$ is a choice of
\begin{enumerate}[a.)]
  \item a ``bar involution" $\psi \maps V \to V$ which is $\Z[q,q^{-1}]$-antilinear (so that $\psi(fv) = \bar{f}\psi(v)$ for $f\in \Z[q,q^{-1}]$ and $v\in V$) and satisfies $\psi^2=1$,
  \item \label{eq:sesqi-flip} a bilinear symmetric inner product $( -, - ) \maps V \times V \to \Z((q^{}))$,

 \item a partially ordered index set $(C,\leq)$ indexing a ``standard basis'' $\{v_c\}$   such that
\[
 \psi(v_c) \in v_c + \sum_{c' < c} \Z[q,q^{-1}]\cdot v_{c'}.
\]
\end{enumerate}

We call a basis $\{b_c\}$ of $V$ \emph{canonical} if
\begin{enumerate}[I.)]
  \item \label{pcan-invariant}each vector $b_c$ is invariant under $\psi$;

  \item \label{pcan-lower}each vector $b_c$ in the basis satisfies
$
b_c - v_c \in \sum_{c'<c} \Z[q,q^{-1}]\cdot v_{c'};
$

\item \label{pcan-orthog} the vectors $b_c$ are almost orthogonal, i.e. $(b_c,b_{c'}) \in \delta_{cc'} + q\Z[[q]]$.
\end{enumerate}
By \cite{Web-canonical}, a precanonical structure admits at most one canonical basis.

\subsubsection{Dual canonical basis}
Webster considers dual canonical bases using certain assumptions that do not apply for the algebras $\Bt(\cal{V})$, while noting that his results hold under more general situations. Here we spell out the setting for the notion of dual canonical basis considered here.

Given a free $\Z[q,q^{-1}]$-module $V$, let $W$ be a $\Z[q,q^{-1}]$-submodule of $V$.
Let $V$ have a canonical basis $\{(-,-), \psi, \{v_c\}, \{b_c\} \}$ on an indexing set $\Pi$
with partial order $\preceq$.

Assume that the bilinear form on $V$ restricts to give a perfect pairing
\[
(- , - ) \maps V \times W \to \Z[q,q^{-1}].
\]
This perfect pairing define a $\Z[q,q^{-1}]$-antilinear involution $\psi^{\ast} \maps W \to W$ by the formula
$
(  \psi v, w ) = \overline{( v, \psi^{\ast} w )},
$
and bases $\{v_c^{\ast}\}$ and $\{b_c^{\ast}\}$ for $W$ dual to the standard and canonical bases, i.e.
\[
(v_c,v_{c'}^{\ast}) = \delta_{cc'}, \qquad \quad (b_c,b_{c'}^{\ast}) = \delta_{cc'}.
\]
Restricting the form from $V$  to $W$ gives a symmetric bilinear form $(-,-)  \maps W \times W \to \Z[q,q^{-1}]$.

\begin{proposition}\label{prop:DualCanonical}
Under the assumptions above, the $\Z[q,q^{-1}]$-module $W$ acquires the structure of a canonical basis $\{  (-,-), \psi^{\ast}, \{v_c^{\ast}\}, \{b_c^{\ast}\} \}$ with index set  $\Pi$
and opposite partial order $\succeq$.
\end{proposition}

\begin{proof}
To show we have a pre-canonical structure, we show that $\psi^{\ast}(v_{c}^{\ast}) \in v_{c}^{\ast} + \sum_{c \prec c'} \Z[q,q^{-1}] v_{c'}^*$ (here $\preceq$ refers to the order on  $\Pi$
rather than its reverse). Indeed, the $v_{c'}^*$-coefficient of the dual standard basis expansion of $\psi^{\ast}(v_{c}^{\ast})$ is $(v_{c'},\psi^{\ast}(v_{c}^{\ast}))$ and we have
\begin{align*}
&\overline{(v_{c'}, \psi^{\ast}(v_{c}^{\ast}))} \\
 &= (\psi(v_{c'}), v_{c}^{\ast}) \\
 &= (v_{c'},v_c^{\ast}) + \sum_{c'' \prec c'} \pi_{c''} ( v_{c''}, v_{c}^{\ast})
\end{align*}
for some elements $\pi_{c''} \in \Z[q,q^{-1}]$. Thus, the coefficient under consideration is $1$ if $c' = c$ and is zero unless $c \preceq c'$.

Now we establish the canonical basis conditions for $\{b_c^{\ast}\}$. Since $\psi(b_c) = b_c$, we have
\[
\overline{(b_c, \psi^{\ast}(b_{c'}^{\ast}))} = (\psi(b_c), b_{c'}^{\ast}) = ( b_c , b_{c'}^{\ast}) =
\overline{(b_c,  b_{c'}^{\ast})}
\]
so that $\psi^{\ast}(b_{c'}^{\ast}) = b_{c'}^{\ast}$ and condition~\eqref{pcan-invariant} holds.

For all  $c \in \Pi$,
$v_{c} \in b_{c} + \sum_{c' \prec c} \Z[q,q^{-1}] b_{c'}$ which follows by induction over the partial order $\prec$. We can thus compute the coefficient of $v_{c'}^{\ast}$ in $b_c^{\ast}$ as follows:
\[
(v_{c'}, b_c^{\ast}) = (b_{c'}, b_c^{\ast}) + \sum_{c'' \prec c'} \pi_{c''} (b_{c''},b_c^{\ast})
\]
for some elements $\pi_{c''} \in \Z[q,q^{-1}]$, so that the coefficient is $1$ if $c' = c$ and is zero unless $c \preceq c'$. Thus, condition~\eqref{pcan-lower} holds.

Finally, for condition~\eqref{pcan-orthog}, it suffices to show that the inverse of the matrix $A := [(b_c, b_{c'})]_{c,c'}$ has entries in $\Z[[q]]$ and reduces to the identity matrix when $q$ is set to zero. Since power series with constant term $1$ (such as the determinant of $A$) are invertible in $\Z[[q]]$, Cramer's rule shows that the entries of the inverse matrix (call it $B$) are in $\Z[[q]]$. We have $AB = BA = \id$ over $\Z[[q]]$ so the same equation holds when setting $q = 0$, but since $A$ becomes the identity matrix when $q=0$, so must $B$.

\end{proof}

In this case, we say that the canonical basis $\{  (-,-)  , \psi^{\ast}, \{v_c^{\ast}\}, \{b_c^{\ast}\} \}$ on $W$ is a dual canonical basis to the canonical basis $\{(-,-), \psi, \{v_c\}, \{b_c\} \}$ on $V$.

\subsubsection{Canonical bases for finite-dimensional hypertoric algebras}

Webster has observed in \cite[Theorem 2.6]{Web-canonical} that hypertoric category $\cal{O}$, and the associated finite-dimensional standardly Koszul algebras $B(\cal{V})$, endow $K_0(B(\cal{V}))$ with a canonical and dual canonical basis such that
\[
\psi = [\#], \quad \psi^{\ast} = [\circledast], \quad  ([M],[N] ) = \sum_{j \in \Z}(-1)^j\overline{\dim_q(\Ext^j(M,N^{\circledast}))} = \sum_{j \in \Z}(-1)^j \dim_{q^{-1}}(\Ext^j(M,N^{\circledast})),
\]
 where $v_c=[V_c]$ are the classes of standard $B(\cal{V})$-modules. The canonical basis corresponds to the classes of $\#$-self dual indecomposable projective modules and dual canonical basis is given by the classes of $\circledast$-self dual simple modules.  Further, the canonical basis of the Ringel dual pre-canonical structure (see \cite[Section 2.1]{Web-canonical}) is given by the classes of $\circledast$-self dual indecomposable tilting $B(\cal{V})$-modules from  \cite[Section 5]{Gale}.

\subsubsection{Equivariant canonical bases for deformed hypertoric algebras}

The following result is essentially contained in \cite[Theorem 2.6]{Web-canonical}, though that proof assumes properties that hold for $B(\cal{V})\pmod$ and $B(\cal{V})\fmod$ that do not hold for $\Bt(\cal{V})\pmod$ or $\Bt(\cal{V})\fmod$.  However, the natural affine analogues of these results do hold and suffice to establish the theorem suitably modified.

\begin{theorem}\label{thm:canonical}
Let $\cal{V}$ be a polarized arrangement. Let $V:=K_0(\Bt(\cal{V}))$, let $\psi = [\#]\maps K_0(\Bt(\cal{V})) \to K_0(\Bt(\cal{V}))^{{\rm op}}$, and let $v_{\alpha}= [\Vt_{\alpha}]$ and $b_{\alpha}= [\Pt_{\alpha}]$ for
 $\alpha \in \Pi := \cal{P}$.
 Give  $\Pi$
 the partial order $\preceq$ from \eqref{subsec:partial_order}. Let $
 ( - , -)  \maps V \times V \to \Z((q))
$
be the bilinear form from \eqref{eq:form-HOM}.
The data $((-,-),\psi,\{v_{\alpha}\}, \{b_{\alpha}\})$ define a canonical basis on $V=K_0(\Bt(\cal{V}))$.
\end{theorem}

\begin{proof}
Recall from Section~\ref{subsec:bilinearform} that the anti-automorphism $\#$ on $\Bt(\cal{V})\pmod$ satisfies $(qP)^{\#} = q^{-1}(P^{\#})$, so it induces an anti-linear automorphism on $K_0$. Since $\Bt(\cal{V})$ is affine quasi hereditary (see Lemma~\ref{lem:Kfiltration} in particular), the projectives $\Pt_{\alpha}$ are filtered by standards $q^j \Vt_{\beta}$ with $\beta \leq \alpha$ and $j=0$ when $\beta = \alpha$. We get $[\Pt_{\alpha}] \in [\Vt_{\alpha}] + \sum_{\beta < \alpha} \Z[q,q^{-1}] [\Vt_{\beta}]$, establishing condition~\ref{pcan-lower} for a canonical basis, and we can deduce that $[\Vt_{\alpha}] \in [\Pt_{\alpha}] + \sum_{\beta < \alpha} \Z[q,q^{-1}] [\Pt_{\beta}]$. By \cite[Lemma 1.6]{Web-canonical}, adapted to the affine highest weight setting, the free $\Z[q,q^{-1}]$-module $V$ acquires a pre-canonical structure with
 $(\psi,(-,-),\cal{P},\preceq)$.
Furthermore, $\#$ fixes indecomposable projectives, establishing condition \eqref{pcan-invariant} for a canonical basis. Condition \eqref{pcan-orthog} follows because $\Bt(\cal{V})$ is nonnegatively graded and the minimal degree element $f_{\alpha \beta}$ of $\tilde{R}_{\alpha\beta}$ has strictly positive degree except when $\alpha=\beta$, in which case it is the unique degree zero map.
\end{proof}

Let $W:=G_0(\Bt(\cal{V}))$, viewed as a $\Z[q,q^{-1}]$ submodule of $V$ as in Section~\ref{sec:Grothendieck-affine}. The form $(-,-) \maps V \times V \to \Z((q))$ restricts to a perfect pairing
$
( - , -)  \maps V  \times W \to \Z[q,q^{-1}]
$
as in \eqref{eq:perfect-pairing}. Under the resulting identification of $W$ with the $\Z[q,q^{-1}]$-linear dual of $V$, Proposition~\ref{prop:pairing} implies that the dual basis elements $v_{\alpha}^{\ast}$ and $b_{\alpha}^{\ast}$ in $V^*$ correspond to $[\bar{\Vt}_{\alpha}]$ and $[L_{\alpha}]$ in $W$ respectively. Give the index set
$\Pi = \cal{P}$
for these bases the reverse of the partial order from \eqref{subsec:partial_order}. By equation~\eqref{eq:PsiAdjointness}, we have $(\psi([X]),[N]) = \overline{([X],[N^{\circledast}])}$ for all $[X] \in V$ and $[N] \in W$, so that $\psi^{\ast} \maps V^* \to V^*$ corresponds to the anti-linear automorphism $\circledast$ of $W = G_0(\Bt(\cal{V}))$. Proposition~\ref{prop:DualCanonical} then implies the following corollary.

\begin{corollary}\label{cor:dualcanonical}
The data $(\{v_{\alpha}^{\ast}\}, \{b_{\alpha}^{\ast}\}, \psi^{\ast} )$ define a dual canonical basis on $W = G_0(\Bt(\cal{V}))$.
\end{corollary}

\begin{remark}
Observe that the results of Theorem~\ref{thm:canonical} and Corollary~\ref{cor:dualcanonical} apply more generally in the context of a polynomial quasi hereditary algebra with balanced anti-involution.
\end{remark}

\section{Projective resolutions of standard modules}\label{sec:ProjRes}

In this section we describe projective resolutions for standard modules over the hypertoric algebras $\At(\V)$ and $\Bt(\V)$ and use these resolutions to compute the Ext groups between the standard modules. In Section~\ref{sec:HeegaardDiags} we will revisit these projective resolutions in the case of cyclic $\V$ (see Section~\ref{sec:Cyclic}), but for now we discuss them in general.

\subsection{The \texorpdfstring{$\At$}{A-tilde} side}\label{sec:ASide}

We start with a technical lemma.

\begin{lemma}\label{lem:Technical}
Let $\cal{V}$ be a polarized arrangement. If $\alpha,\beta \in \cal{P}$ with $\alpha(i) = \beta(i)$ for all $i \in \xx_{\a} \cap \xx_{\b}^c$, then $\beta \leq \alpha$.
\end{lemma}

\begin{proof}
It suffices to show that the point $H_{\xx_{\b}}$ lies in the negative cone of the point $H_{\xx_{\a}}$, i.e. in the union of regions $\Delta_{\gamma}$ for $\gamma \in \cal{B}_{\xx_{\a}}$. Indeed, if $H_{\xx_{\b}} \cap \Delta_{\gamma} \neq 0$ for some $\gamma \in \cal{B}_{\xx_a}$, then we also have $\gamma \in \cal{F}_{\xx_{\b}}$ by definition of $\cal{F}_{\xx_{\b}}$, so $\beta \leq \gamma \leq \alpha$ by Lemma~\ref{lem:OrderFromCones}.

Let $\sigma(H_{\xx_{\b}})$ be the sequence of elements of $\{+,-,0\}$ associated to the point $H_{\xx_{\b}}$ (at an index $i$, this sequence is $0$ if $H_{\xx_{\b}}$ lies on hyperplane $i$ and it is $\b(i)$ otherwise). To show that $H_{\xx_{\b}}$ lies in the negative cone of $H_{\xx_{\a}}$, it suffices to show that
\[
\sigma(H_{\xx_{\b}})(i) \in \{0,\a(i)\}
\]
for all $i \in \xx_{\a}$. We know that $\sigma(H_{\xx_{\b}})(i) = 0 \in \{0,\a(i)\}$ for all $i \in \xx_{\b}$ and thus for all $i \in \xx_{\a} \cap \xx_{\b}$. For $i \in \xx_{\b}^c$, we have $\sigma(H_{\xx_{\b}})(i) = \b(i)$, so for $i \in \xx_{\a} \cap \xx_{\b}^c$ the assumption $\a(i) = \b(i)$ gives us $\sigma(H_{\xx_{\b}})(i) = \a(i) \in \{0,a(i)\}$, proving the lemma.
\end{proof}

Recall that a linear projective resolution of a graded module is one in which each projective summand appearing in homological degree $i$ of the resolution is degree-shifted upwards by $i$ in its internal grading (see \cite[Definition 5.15]{Gale}).

\begin{theorem}\label{thm:proj-resolution}
Let $\cal{V}$ be a polarized arrangement.   Then the standard module $\Vt_{\alpha}$ for $\At(\cal{V})$ has a linear projective resolution by projectives $\Pt_{\beta}$ for $\beta\in \cal{F}_{\mathbbm{x}_{\alpha}} \cap \cal{B}$.
\end{theorem}

\begin{proof} This proof is a straightforward extension of \cite[Theorem 5.24]{Gale} replacing projectives and standards for $A(\cal{V})$ with their deformed analogs $\Pt_{\alpha}$ and $\Vt_{\beta}$ for $\At(\cal{V})$.

Let $\mathbbm{x}_{\alpha} = \mu^{-1}(\alpha)$ be the basis associated with the sign vector
$\alpha$.  For any subset $S\subset \mathbbm{x}_{\alpha}$, let $\alpha^S$ be the sign vector that
differs from $\alpha$ in exactly the indices in $S$.  Thus, for example,
$\alpha^\emptyset = \alpha$, and $\alpha^{\{i\}} = \alpha^i$ for all $i\in \mathbbm{x}_{\alpha}$.
(Note that the sign vectors that arise this way are exactly those in the set $\cal{F}_{\mathbbm{x}_{\alpha}}$.)
If $S = S' \sqcup \{i\} \subset \mathbbm{x}_{\alpha}$, then we have a degree-preserving map
$\vp_{S,i}\colon q\Pt_{\alpha^{S}}\to \Pt_{\alpha^{S'}}$ given by  right
multiplication by
the element $p(\alpha^{S}, \alpha^{S'})$.  We adopt the convention that
$\Pt_{\alpha^S} = 0$ if $\alpha^S\notin\cal{P}$ and $\vp_{S,i}=0$ if $i\notin S$.

Let
\[
\Pi_\alpha=\bigoplus_{S\subset \xx_{\alpha}} q^{|S|} \Pt_{\alpha^S}
\]
be the sum of all of
the projective modules associated to the sign vectors $\alpha^S$, with indicated grading shifts. Besides the internal grading, this module also has a multi-grading
by the abelian group $\Z^{\mathbbm{x}_{\alpha}} = \Z\{\ep_i\mid i\in \mathbbm{x}_{\alpha}\}$,
with the summand $q^{|S|} \Pt_{\alpha^S}$ sitting in multi-degree $\ep_S = \sum_{i\in S}\ep_i$.
For each $i\in \xx_{\alpha}$, we define a differential
\[
\partial_i = \sum_{i\in S\subset \xx_{\alpha}}\vp_{S,i}
\]
of degree $-\ep_i$.
These differentials commute because of the relation (A2), and thus
define a multi-complex structure on $\Pi_\alpha$.  The total complex $\tPi$
of this multi-complex is linear and projective; we claim that it is a resolution
of the standard module $\Vt_{\alpha}$.

Choose an ordering $(\beta_1, \ldots,\beta_{\ell})$ of $\cal{P}$ such that if $\beta_i < \beta_j$ then $i < j$. We may filter each projective module $\Pt_{\alpha^S}$ by submodules
\[
\Pt_{\alpha^S} = \At e_{\alpha^S} = \At \left( \sum_{i=1}^{\ell} e_{\beta_i} \right) \At e_{\alpha^S} \supset \cdots \supset \At e_{\beta_{\ell}} \At e_{\alpha^S};
\]
correspondingly, we can filter $\Pi_{\alpha}$ by letting
\[
\Pi_{\alpha}^{h, \geq j} := \bigoplus_{S \subset \xx_{\a}, |S| = h} q^h \At \left( \sum_{i=j}^{\ell} e_{\beta_i} \right) \At e_{\alpha^S}
\]
where $h$ denotes the homological degree and $\geq j$ denotes the filtration level for $1 \leq j \leq \ell$. The differentials $\partial_i$ on $\Pi_{\alpha}$ (and thus the differential on the total complex $\tPi$) respect the filtration, i.e. they send $\Pi_{\alpha}^{h, \geq j}$ to $\Pi_{\alpha}^{h-1, \geq j}$ for all $h,j$.

Writing $\widetilde{\boldsymbol{\Pi}}_{\alpha}^{\bullet,\bullet}$ for the associated graded complex, we have
\[
\widetilde{\boldsymbol{\Pi}}_{\alpha}^{h,j} = \frac{\bigoplus_{S \subset \xx_{\a}, |S| = h} q^h \At \left( \sum_{i=j}^{\ell} e_{\beta_i} \right) \At e_{\alpha^S}}{\bigoplus_{S \subset \xx_{\a}, |S| = h} q^h \At \left( \sum_{i=j+1}^{\ell} e_{\beta_i} \right) \At e_{\alpha^S}} = \bigoplus_{S \subset \xx_{\a}, |S| = h} q^h Q_{\alpha^S}^{h,j}
\]
where
\[
Q_{\alpha^S}^{h,j} := \frac{\At \left( \sum_{i=j}^{\ell} e_{\beta_i} \right) \At e_{\alpha^S}}{\At \left( \sum_{i=j+1}^{\ell} e_{\beta_i} \right) \At e_{\alpha^S}}.
\]
For a given $S$ (with $h := |S|$) and $j$, if $\alpha^S \notin \cal{B}_{\xx_{\beta_j}}$ then there exists some index $i \in \xx_{\beta_j}$ with $\alpha^S(i) \neq \beta_j(i)$.  Since $\beta_j^i > \beta_j$, we have $\beta_j^i = \beta_{j'}$ for some $j' > j$, and any path from $\beta_j$ to $\alpha^S$ is equal to a path that goes through $\beta_{j'}$, so the quotient $Q_{\alpha^S}^{h,j}$ vanishes.

On the other hand, if $\alpha^S \in \cal{B}_{\xx_{\beta_j}}$, then postcomposition with a taut path $p_{\beta_j}^{\alpha^S}$ from $\beta_j$ to $\alpha^S$ defines a degree-zero map
\[
q^{d_{\alpha^S,\beta_j}} \Pt_{\beta_j} \rightarrow \At \left( \sum_{i=j}^{\ell} e_{\beta_i} \right) \At e_{\alpha^S} \twoheadrightarrow Q_{\alpha^S}^{h,j}
\]
which annihilates $\At \varepsilon_{\beta_j} \At e_{\beta_j} \subset \Pt_{\beta_j} = \At e_{\beta_j}$, so we get a map from $q^{d_{\alpha^S,\beta_j}} \Vt_{\beta_j}$ to $Q_{\alpha^S}^{h,j}$. As in Lemma~\ref{lem:Kfiltration}, this map is surjective, and injectivity follows from
\[
\dim_q \Pt_{\alpha^S} = \sum_{\xx: \alpha^S \in \beta_{\xx}} q^{d(\alpha^S,\mu_{\xx})} \dim_q \Vt_{\mu(\xx)}
\]
which was proved in Lemma~\ref{lem:Kfiltration}. We get that
\[
\widetilde{\boldsymbol{\Pi}}_{\alpha}^{h,j} \cong \bigoplus_{S \subset \xx_{\a}, |S| = h, \alpha^S \in \cal{B}_{\xx_{\beta_j}}} q^{h + d_{\alpha^S,\beta_j}} \Vt_{\beta_j}.
\]

Now, if $\beta_j = \alpha$, only the $S = \emptyset$ term contributes to $\widetilde{\boldsymbol{\Pi}}_{\alpha}^{h,j}$ (with $h = 0$); we have $\widetilde{\boldsymbol{\Pi}}_{\alpha}^{0,j} \cong \Vt_{\alpha}$ and $\widetilde{\boldsymbol{\Pi}}_{\alpha}^{h,j} = 0$ for $h \neq 0$.

For all other $\beta_j$, we will show the complex $\widetilde{\boldsymbol{\Pi}}_{\alpha}^{\bullet,j}$ is acyclic.  Indeed, if $\beta_j \neq \alpha$ and $\alpha(i) = \beta_j(i)$ for all $i \in \xx_{\a} \cap \xx_{\beta_j}^c$, then by Lemma~\ref{lem:Technical} we have $\beta_j < \alpha$. It follows that $\cal{F}_{\xx_{\alpha}} \cap \cal{B}_{\xx_{\beta}} = \emptyset$, so that for all $S \subset \xx_{\a}$ we have $\alpha^S \notin \cal{B}_{\xx_{\beta}}$. In this case, $\widetilde{\boldsymbol{\Pi}}_{\alpha}^{\bullet,j}$ is the zero complex.

 Now assume that $\beta_j \neq \alpha$ and we have $i \in \xx_{\a} \cap \xx_{\b_j}^c$ with $\alpha(i) \neq \beta_j(i)$.
For $S \subset \xx_{\a}$ with $i \in S$, write $S = S' \sqcup \{i\}$; since $i \in \xx_{\beta_j}^c$, we have $\alpha^S \in \cal{B}_{\xx_{\beta_j}}$ if and only if $\alpha^{S'} \in \cal{B}_{\xx_{\beta_j}}$.

The complex $\widetilde{\boldsymbol{\Pi}}_{\alpha}^{\bullet,j}$ is the total complex of a multi-complex $\widetilde{\Pi}_{\alpha}^{\bullet,j}$ whose differential has degree $-\varepsilon_i$ part $\tilde{\partial}_i^{\beta_j}$ induced by $\partial_i$. For $S \subset \xx_{\a}$ (with $|S| = h$ and $\alpha^S \in \cal{B}_{\xx_{\beta_j}}$), the differential $\tilde{\partial}_i^{\beta_j}$ is only nonzero on the corresponding summand $q^h Q_{\alpha^S}^{h,j}$ of $\widetilde{\Pi}_{\alpha}^{h,j}$ if $i \in S$, so that $S = S' \sqcup \{i\}$. If neither of $\{\alpha^S, \alpha^{S'}\}$ is in $\cal{B}_{\xx_{\beta_j}}$, then $\tilde{\partial}_i^{\beta_j}$ on this summand maps $0 \to 0$. If both of $\{\alpha^S, \alpha^{S'}\}$ are in $\cal{B}_{\xx_{\beta_j}}$, then $\tilde{\partial}_i^{\beta_j}$ on this summand is the isomorphism
\[
q^h Q_{\alpha^S}^{h,j} \cong q^{h + d_{\alpha^S,\beta_j}} \Vt_{\beta_j} \cong q^{h-1} Q_{\alpha^{S'}}^{h-1,j}
\]
given by right multiplication by $p_{\alpha^S}^{\alpha^{S'}}$ (the map is an isomorphism since $\beta_j(i) = \alpha_S(i)$ ($\neq (\alpha_{S'}(i) = \alpha(i))$ by assumption).

It follows that the homology of $\widetilde{\Pi}_{\alpha}^{\bullet,j}$ with respect to its degree $-\varepsilon_i$ differential vanishes, so the total complex $\widetilde{\boldsymbol{\Pi}}_{\alpha}^{\bullet,j}$ is acyclic. In other words, in the unique filtration degree $j$ such that $\beta_j = \alpha$, the associated graded complex of $\tPi$ has homology $\Vt_{\alpha}$ concentrated in homological degree zero, while in all other filtration degrees the associated graded complex of $\tPi$ is acyclic. For degree reasons, the spectral sequence induced by the filtration collapses at this level. We conclude that $\tPi$ has homology $\Vt_{\alpha}$ as desired.
\end{proof}

\subsection{The \texorpdfstring{$\Bt$}{B-tilde} side}\label{sec:BSide}

In this section, for $\alpha \in \cal{P}$ let $\Pt_{\a}$ and $\Vt_{\a}$ denote the indecomposable projective and standard modules over $\Bt(\cal{V}) \cong \At(\cal{V}^{\vee})$ associated to $\alpha$. By \cite[Lemma 2.10]{Gale}, the partial order on $\cal{P} = \cal{P}^{\vee}$ induced by $\cal{V}^{\vee}$ is the opposite of the one induced by $\cal{V}$, so for $\a \in \cal{P}$ we have
\[
\Vt_{\a} = (\Bt(\cal{V}) e_{\a}) / (\Bt(\cal{V}) \cal{I}_{\a} \Bt(\cal{V}) e_{\a})
\]
by definition, where $\cal{I}_{\a} := \sum_{\b < \a} e_{\b}$.

For $\xx \in \mathbb{B}$, write $\Z[u_{\xx}]:=\Z[u_i: i \in \xx]$. By Gale duality and Proposition~\ref{prop:Vtaut}, $\Vt_{\a}$ is a free $\Z[u_{\xx_{\a}}]$-module with basis indexed by the set $\cal{F}_{\xx_{\a}} \cap \cal{B}$. As in equation~\eqref{eq:StandardMod}, we can write
\[
\Vt_{\a} = \left(\bigoplus_{\beta \in \cal{F}_{\xx_{\a}} \cap \cal{B}} e_{\b} \Bt(\cal{V}) e_{\a} \right)/\left( u_i \mid i \notin \xx_{\a} \right).
\]
Proposition~\ref{prop:endV} implies that
\begin{equation}
\End_{\Bt(\cal{V})}(\Vt_{\a}) = \Z[u_{\xx_{\a}}]
 \end{equation}
and that $\Z[u_{\xx_{\a}^c}]$ acts trivially on $\Vt_{\a}$. The proof of Theorem~\ref{thm:affine} implies that
\begin{equation} \label{eq:HomPV}
\Hom_{\Bt(\cal{V})}(\Pt_{\a},\Vt_{\b})
\cong
\left\{
  \begin{array}{ll}
    0, & \hbox{if $\a \notin \cal{F}_{\b}$;} \\
    \Z[u_{\xx_{\b}}], & \hbox{if $\a \in \cal{F}_{\b}$,}
  \end{array}
\right\},
 \end{equation}
where if $\a \in \cal{F}_{\xx_{\b}} \cap \cal{B}$, then $1 \in \Z[u_{\xx_{\b}}]$ corresponds to a map from $\Pt_{\a}$ to $\Vt_{\b}$ of degree $q^{d_{\a,\b}}$. Theorem~\ref{thm:proj-resolution} implies that the standard $\Bt(\cal{V})$-module $\Vt_{\a}$ has a linear projective resolution with projectives indexed by the set $\B_{\xx_{\a}} \cap \cal{F}$; indeed,
\[
\Vt_{\a} \simeq \bigoplus_{\beta \in \B_{\xx_{\a}}} q^{d_{\a,\b}} \tilde{P}_{\beta} = \bigoplus_{S \subset \mathbbm{x}_{\alpha}^c} q^{|S|} \tilde{P}_{\alpha^S}
\]
with a differential defined by, for $S = S' \sqcup \{i\} \subset \xx_{\a}^c$, mapping $q^{|S|} \tilde{P}_{\alpha^S}$ to $q^{|S|-1}\tilde{P}_{\alpha^{S'}}$ by $\varphi_{S,i}$ (as in Section~\ref{sec:ASide} we set $\tilde{P}_{\b} = 0$ when $\b \notin \cal{P}$).

\subsection{Ext groups between standard modules}

In this section we continue to work over $\Bt(\cal{V})$.   The material in this section departs from the discussion in \cite{Gale}, as we are able to leverage the analysis of standard modules in the previous sections to make explicit computations of Ext-groups. 

\begin{lemma}\label{lem:BFIntersection}
If $\cal{B}_{\xx_{\a}}\cap \cal{F}_{\xx_{\b}} \neq \emptyset$ then
\[
\cal{B}_{\xx_{\a}}\cap \cal{F}_{\xx_{\b}}  =
\{
\gamma \in \{\pm\}^n \mid \gamma(i) = \alpha(i) \quad  \forall i \in\xx_{\a} , \;\;  \text{and } \gamma(i) = \beta(i) \quad \forall i \in \xx_{\b}^c
\} \subset \cal{P}.
\]
In particular, $ \gamma(i)$ can take either value $\pm$ for all $i \in \xx_{\a}^c \cap \xx_{\b}$.
\end{lemma}

\begin{proof}
For $\gamma \in \cal{B}_{\xx_{\a}}\cap \cal{F}_{\xx_{\b}}$ we have
\[
\gamma(i) =
\left\{
  \begin{array}{ll}
    \alpha(i), & \hbox{$i \in \xx_{\a} \cap \xx_{\b}$;} \\
\alpha(i)=\beta(i), & \hbox{$i \in \xx_{\a}  \cap \xx_{\b}^c$;} \\
      \beta(i), & \hbox{$i \in \xx_{\a}^c \cap \xx_{\b}^c$;} \\
     \pm , & \hbox{$i \in \xx_{\a}^c \cap \xx_{\b}$.}
  \end{array}
\right.
\]
All such $\gamma$ are in $\cal{P}$ since $\gamma \in \cal{B}_{\xx_{\a}} \subset \cal{B}$ and $\gamma \in \cal{F}_{\xx_{\b}}\subset \cal{F}$.
\end{proof}

\begin{proposition}\label{prop:ExtComplexResolveLeft}
Let $S_{\min} = \{i \in \xx_{\a}^c \cap \xx_{\b}^c: \alpha(i) \neq \beta(i)\}$. The group $\Ext_{\Bt(\cal{V})}(\Vt_{\a},\Vt_{\b})$ is the homology of the total complex of the multicomplex formed by
\[
\mf{X}_{\a} = \bigoplus_{S \textrm{  with  } S_{\min} \subset S \subset S_{\min} \cup (\xx_{\a}^c \cap \xx_\b)} q^{d_{\a^S, \b} - |S|} \Z[u_{\xx_{\b}}]
\]
with differential from the summand indexed by $S'$ to the summand indexed by $S = S'\sqcup \{ j\}$ given by multiplication by
$\begin{cases}
u_j, & \alpha(j) \neq \beta(j) \\
1, & \alpha(j) = \beta(j).
\end{cases}$
\end{proposition}

\begin{proof}
By \eqref{eq:HomPV}, for each term $\Pt_{\gamma}$ in the projective resolution of $\Vt_\a$ with $\gamma\in \cal{B}_{\xx_{\a}}\cap \cal{F}_{\xx_{\b}}$), the space of maps from $\Pt_{\gamma}$ to $\Vt_\b$ is a free module of rank one over $\Z[u_{\xx_{\b}}]$ generated by a map of degree $d_{\gamma, \b}$. Hence, for each term in the submodule
\[
\mf{X}_{\a}' = \bigoplus_{S \textrm{  with  } S_{\min} \subset S \subset S_{\min} \cup (\xx_{\a}^c \cap \xx_\b)} q^{|S|} \Pt_{\alpha^{S}}
\]
of the resolution of $\Vt_{\a}$ there is a free rank-one $\Z[u_{\xx_{\b}}]$-module of maps to $\Vt_{\b}$ generated by a map of degree $d_{\alpha^S,\b} - |S|$. This module $\mf{X}_{\a}'$ also has a multi-grading by $\Z\{\epsilon_i \mid i \in \xx_{\a}^c \cap \xx_{\b} \}$, with summand $q^{|S|}\Pt_{\a^{S}}$ sitting in multi-degree $\epsilon_{S}=\sum_{i\in S \setminus S_{\min}} \epsilon_i$.  We write $\kappa_S$ for the minimal map in $\Hom(\mf{X}_{\a}',\Vt_{\b})$ mapping $q^{|S|} \Pt_{\alpha^{S}}$ to $\Vt_{\b}$. We have $\kappa_S = - \cdot p_{\alpha^S}^{\b}$ where $p_{\alpha^S}^{\b}$ denotes a taut path from $\alpha^S$ to $\b$ in $Q(\cal{V}^{\vee})$ (see Section~\ref{sec:ABAlgebraDefs}; we are using the identification $\Bt(\cal{V}) \cong \At(\cal{V}^{\vee})$ to view $p_{\alpha^S}^{\b}$ as an element of $\Bt(\cal{V})$).

In $\mf{X}_{\a}'$, the component of the differential mapping $q^{|S|} \Pt_{\alpha^{S}} \to q^{|S|-1} \Pt_{\alpha^{S'}}$ with $S=S'\sqcup \{ j\}$ is given by right multiplication by $p_{\alpha^S}^{\alpha^{S'}}$ (again this denotes a taut path in $Q(\cal{V}^{\vee})$). Hence, the differential of the map $\kappa_{S'}\maps \Pt_{\alpha^{S'}} \to \Vt_{\b}$ is given by
\[
\bigoplus_{\stackrel{j \in (\xx_{\a}^c \cap \xx_{\b}) \setminus S'}{S_j:=S'\sqcup \{ j\}}} - \cdot p_{\alpha^{S_j}}^{\alpha^{S'}} \cdot p_{\alpha^{S'}}^{\b}.
\]
If $\alpha(j) \neq \beta(j)$, then
\[
\bigoplus_{\stackrel{j \in (\xx_{\a}^c \cap \xx_{\b}) \setminus S'}{S_j:=S'\sqcup \{ j\}}} - \cdot p_{\alpha^{S_j}}^{\alpha^{S'}} \cdot p_{\alpha^{S'}}^{\b}
=
\bigoplus_{\stackrel{j \in (\xx_{\a}^c \cap \xx_{\b}) \setminus S'}{S_j:=S'\sqcup \{ j\}}}   - \cdot u_j p_{\alpha^{S_j}}^{\b},
\]
while if $\alpha(j) = \beta(j)$ then
\[
\bigoplus_{\stackrel{j \in (\xx_{\a}^c \cap \xx_{\b}) \setminus S'}{S_j:=S'\sqcup \{ j\}}} - \cdot p_{\alpha^{S_j}}^{\alpha^{S'}} \cdot p_{\alpha^{S'}}^{\b}
=
\bigoplus_{\stackrel{j \in (\xx_{\a}^c \cap \xx_{\b}) \setminus S'}{S_j:=S'\sqcup \{ j\}}}   - \cdot p_{\alpha^{S_j}}^{\b}.
\]
\end{proof}

\begin{corollary}\label{cor:BTildeExtComputation}
$
\Ext^i_{\Bt(\cal{V})}(\Vt_{\a},\Vt_{\b})
=0
$ unless $\cal{B}_{\xx_{\a}}\cap \cal{F}_{\xx_{\b}}\neq \emptyset$, $\alpha(j) \neq \beta(j)$ for all $j \in \xx_{\a}^c \cap \xx_{\b}$, and $i=|\{i \in \xx_{\a}^c \cap \xx_{\b}^c: \alpha(i) \neq \beta(i)\}| + |\xx_{\a}^c \cap \xx_{\b}|$, in which case
\[
\Ext^i_{\Bt(\cal{V})}(\Vt_{\a},\Vt_{\b})\cong q^i \Z[u_{\xx_{\b}}]/( u_j \mid j \in \xx_{\a}^c \cap \xx_{\b}) \cong q^i \Z[u_{\xx_{\a}\cap\xx_{\b}}].
\]
In particular, $\Hom_{\Bt(\cal{V})}(\Vt_{\alpha},\Vt_{\beta}) = \Ext^0_{\Bt(\cal{V})}(\Vt_{\a},\Vt_{\b}) = \delta_{\a\b}$.
\end{corollary}

\begin{corollary} The Ext groups for standard modules of the undeformed algebras $B(\cal{V})$ satisfy
$
\Ext^i_{B(\cal{V})}(V_{\a},V_{\b})
=0
$ unless $\cal{B}_{\xx_{\a}}\cap \cal{F}_{\xx_{\b}}\neq \emptyset$, $\alpha(j) \neq \beta(j)$ for all $j \in \xx_{\a}^c \cap \xx_{\b}$, and $i=|\{i \in \xx_{\a}^c \cap \xx_{\b}^c: \alpha(i) \neq \beta(i)\}| + |\xx_{\a}^c \cap \xx_{\b}|$.
\end{corollary}

\subsection{Chain maps between projective resolutions}

\begin{remark}
This section will be used only in Section~\ref{sec:Ainfty} below, where we work over $\F_2$ and ignore gradings, so we will do the same here.
\end{remark}

Assume that $\cal{B}_{\xx_{\a}}\cap \cal{F}_{\xx_{\b}}\neq \emptyset$ and that $\alpha(j) \neq \beta(j)$ for all $j \in \xx_{\a}^c \cap \xx_{\b}$. We can view $\Ext_{\Bt(\cal{V})}(\Vt_{\a},\Vt_{\b})$ as the homology of the space of $\Bt(\cal{V})$-linear maps from the projective resolution of $\Vt_{\a}$ to the projective resolution of $\Vt_{\b}$, with differential given (over $\F_2$) by $d \circ (-) + (-) \circ d$. In this section, for each basis element of $\Ext_{\Bt(\cal{V})}(\Vt_{\a},\Vt_{\b})\cong \F_2[u_{\xx_{\b}}]/( u_j \mid j \in \xx_a^c \cap \xx_{b})$, we will define a chain map between projective resolutions whose class in homology is the given basis element of the Ext group.

\begin{definition}\label{def:ChainMapForExt}
Let $\alpha, \beta \in \cal{P}$; assume that $\cal{B}_{\xx_{\a}^c} \cap \cal{F}_{\xx_{\b}} \neq \emptyset$ and that $\alpha(j) \neq \beta(j)$ for all $j \in \xx_{\a}^c \cap \xx_{\b}$. Define a linear map $\varphi_{\a,\b}$ of $\Bt(\cal{V})$-modules from the projective resolution of $\Vt_{\a}$ to the projective resolution of $\Vt_{\b}$ as follows.
\begin{itemize}
\item Writing
\[
\bigoplus_{S \subset \xx_{\a}^c} \tilde{P}_{\alpha^{S}},
\]
for the projective resolution of $\Vt_{\a}$, the map $\varphi_{\a,\b}$ is nonzero only on the summands $\tilde{P}_{\alpha^{S}}$ for
\[
S = S_{\min} \cup (\xx_{\a}^c \cap \xx_{\b}) \cup S',
\]
where $S_{\min} = \{ i \in \xx_{\a}^c \cap \xx_{\b}^c: \alpha(i) \neq \beta(i) \}$ and $S'$ is any subset of $\{ i \in \xx_{\a}^c \cap \xx_{\b}^c : \alpha(i) = \beta(i) \}$.

\item Writing
\[
\bigoplus_{T \subset \xx_{\b}^c} \tilde{P}_{\beta^{T}},
\]
for the projective resolution of $\Vt_{\b}$, the map $\varphi_{\a,\b}$ sends the summand $\tilde{P}_{\alpha^{S}}$ in the previous item to the summand $\tilde{P}_{\beta^{S'}}$ (i.e. the summand for $T := S'$; this makes sense because $S' \subset \xx_{\b}^c$).

\item As a map from $\tilde{P}_{\alpha^{S}}$ to $\tilde{P}_{\beta^{S'}}$, the map $\varphi_{\a,\b}$ sends the generator $e_{\alpha^{S}}$ of $\tilde{P}_{\alpha^{S}}$ to $p_{\alpha^{S}}^{\beta^{S'}} \in \tilde{P}_{\beta^{S'}}$ (recall that $p_{\alpha^{S}}^{\beta^{S'}}$ denotes any taut path from $\alpha^{S}$ to $\beta^{S'}$ in $Q(\cal{V}^{\vee})$, where we are using the identification $\Bt(\cal{V}) \cong \At(\cal{V}^{\vee})$).

\end{itemize}
\end{definition}

\begin{proposition}\label{prop:VarphiChainMap}
The map $\varphi_{\a,\b}$ above intertwines the differential on the projective resolution of $\Vt_{\a}$ with the differential on the projective resolution of $\Vt_{\b}$.
\end{proposition}

\begin{proof}
Denote the differentials on both projective resolutions by $d$. We get a nonzero term of $\varphi_{\a,\b} \circ d$ when we have $S = S_{\min} \cup (\xx_{\a}^c \cap \xx_{\b}) \cup S'$ for some subset $S'$ of $\{i \in \xx_{\a}^c \cap \xx_{\b}^c: \alpha(i) = \beta(i) \}$ such that for some element $j$ of $\{i \in \xx_{\a}^c \cap \xx_{\b}^c: \alpha(i) = \beta(i) \}$, we have $j \notin S''$. Then the composition
\[
\tilde{P}_{\alpha^{S_{\min} \cup (\xx_{\a}^c \cap \xx_{\b}) \cup S' \cup \{j\}}} \xrightarrow{d} \tilde{P}_{\alpha^{S_{\min} \cup (\xx_{\a}^c \cap \xx_{\b}) \cup S'}} \xrightarrow{\varphi_{\a,\b}} \tilde{P}_{\beta^{S'}}
\]
sends the generator $e_{\alpha^{S_{\min} \cup (\xx_{\a}^c \cap \xx_{\b}) \cup S' \cup \{j\}}}$ of $\tilde{P}_{\alpha^{S_{\min} \cup (\xx_{\a}^c \cap \xx_{\b}) \cup S' \cup \{j\}}}$ to the element
\[
p_{\alpha^{S_{\min} \cup (\xx_{\a}^c \cap \xx_{\b}) \cup S' \cup \{j\}}}^{\alpha^{S_{\min} \cup (\xx_{\a}^c \cap \xx_{\b}) \cup S'}} p_{\alpha^{S_{\min} \cup (\xx_{\a}^c \cap \xx_{\b}) \cup S'}}^{\beta^{S'}}
\]
of $\tilde{P}_{\beta^{S'}}$. Since $\alpha(j) = \beta(j)$, $j \notin S'$ so $\beta^{S'}(j) = \beta(j)$, and the element $p_{\alpha^{S_{\min} \cup (\xx_{\a}^c \cap \xx_{\b}) \cup S' \cup \{j\}}}^{\alpha^{S_{\min} \cup (\xx_{\a}^c \cap \xx_{\b}) \cup S'}}$ is a single-step path, the above product is equal to
\[
p_{\alpha^{S_{\min} \cup (\xx_{\a}^c \cap \xx_{\b}) \cup S' \cup \{j\}}}^{\beta^{S''}}.
\]

On the other hand, the composition
\[
\tilde{P}_{\alpha^{S_{\min} \cup (\xx_{\a}^c \cap \xx_{\b}) \cup S' \cup \{j\}}} \xrightarrow{\varphi_{\a,\b}} \tilde{P}_{\beta^{S' \cup \{j\}}} \xrightarrow{d} \tilde{P}_{\beta^{S'}}
\]
sends $e_{\alpha^{S_{\min} \cup (\xx_{\a}^c \cap \xx_{\b}) \cup S' \cup \{j\}}}$ to the element
\[
p_{\alpha^{S_{\min} \cup (\xx_{\a}^c \cap \xx_{\b}) \cup S' \cup \{j\}}}^{\beta^{S' \cup \{j\}}} p_{\beta^{S' \cup \{j\}}}^{\beta^{S'}};
\]
indeed, since $j \in \xx_{\b}^c \setminus S'$, there is a term of $d$ mapping the generator of $\tilde{P}_{\beta^{S' \cup \{j\}}}$ to $p_{\beta^{S' \cup \{j\}}}^{\beta^{S'}} \in \tilde{P}_{\beta^{S'}}$. By the same reasoning as above, the product is equal to
\[
p_{\alpha^{S_{\min} \cup (\xx_{\a}^c \cap \xx_{\b}) \cup S'' \cup \{j\}}}^{\beta^{S'}}.
\]

It thus suffices to show that the terms of $d \circ \varphi_{\a,\b}$ appearing above are all the terms of $d \circ \varphi_{\a,\b}$. Indeed, for an arbitrary subset $S'$ of $\{i \in \xx_{\a}^c \cap \xx_{\b}^c: \alpha(i) = \beta(i)\}$, all terms of $d$ applied to $\tilde{P}_{\beta^{S'}}$ arise from indices $j' \in S'$; the term for $j' \in S'$ sends $e_{\beta^{S'}} \in \tilde{P}_{\beta^{S'}}$ to $p_{\beta^{S'}}^{\beta^{S' \setminus \{j'\}}} \in \tilde{P}_{\beta^{S' \setminus \{j'\}}}$.
\end{proof}

\begin{proposition}\label{prop:ChainMapInducesRightExtElt}
The element of $\Ext_{\Bt(\cal{V})}(\Vt_{\a},\Vt_{\b})$ induced by the map $\varphi_{\a,\b}$ is the element corresponding to $1 \in \F_2[U_{\xx_{\a} \cap \xx_{\b}}]$ under the isomorphism of Corollary~\ref{cor:BTildeExtComputation}.
\end{proposition}

\begin{proof}
The Ext groups computed by resolving both $\Vt_{\a}$ and $\Vt_{\b}$ are identified with the Ext groups computed by resolving only $\Vt_{\a}$ by sending a morphism between projective resolutions to the map from the resolution of $\Vt_{\a}$ to $\Vt_{\b}$ obtained by postcomposition with the map
\[
\bigoplus_{T \subset \xx_{\b}^c} \Pt_{\b^T} \to \Vt_{\b}
\]
which is the quotient projection $\Pt_{\b} \to \Vt_{\b}$ from Section~\ref{sec:HypertoricStandardMods} on the summand $P_{\b}$ (for $T = \emptyset$) and is zero on all other summands. In particular, $\varphi_{\a,\b}$ gets sent to the map
\[
\bigoplus_{S \subset \xx_{\a}^c} \Pt_{\a^S} \to \Vt_{\b}
\]
which, for $S = S_{\min} \cup (\xx_{\a}^c \cap \xx_{\b})$ with $S_{\min} = \{i \in \xx_{\a}^c \cap \xx_{\b} : \alpha(i) \neq \beta(i) \}$, maps the generator $e_{\a^S}$ of $\Pt_{\a^S}$ to the class of $p_{\a^S}^{\b}$ in the quotient $\Vt_{\b}$ of $\Pt_{\b}$, and is zero on the summands $\Pt_{\a^S}$ for all other $S \subset \xx_{\a}^c$. This map is the one called $\kappa_{\xx_{\a}^c \cap \xx_{\b}}$ in the proof of Proposition~\ref{prop:ExtComplexResolveLeft}, corresponding to $1 \in \F_2[u_{\xx_{\a} \cap \xx_{\b}}]$ under the isomorphism of Corollary~\ref{cor:BTildeExtComputation}.
\end{proof}

More generally, if $\mu$ is any monomial in the variables $u_{\xx_{\a} \cap \xx_{\b}}$, define a differential-intertwining morphism $\mu \varphi_{\a,\b}$ from the projective resolution of $\Vt_{\a}$ to the projective resolution of $\Vt_{\b}$ by multiplying the outputs of $\varphi_{\a,\b}$ by $\mu$. Proposition~\ref{prop:ChainMapInducesRightExtElt} implies the following.

\begin{corollary}\label{cor:VarphiInducesRightExtElts}
The element of $\Ext_{\Bt(\cal{V})}(\Vt_{\a},\Vt_{\b})$ induced by $\mu\varphi_{\a,\b}$ is the element corresponding to $\mu \in \F_2[u_{\xx_{\a} \cap \xx_{\b}}]$ under the isomorphism of Corollary~\ref{cor:BTildeExtComputation}.
\end{corollary}

\section{Cyclic arrangements}\label{sec:Cyclic}

Now we restrict attention to certain special arrangements, related to total positivity, for which the bounded feasible regions admit alternate combinatorial descriptions. We also characterize the partial order as well as bounded and feasible cones for such arrangements.

\subsection{Definitions}

\subsubsection{Cyclic arrangements}

We let $\Grr_{k,n}^{> 0}$ denote the positive Grassmannian consisting of positive (i.e. totally positive) $k$-dimensional subspaces of $\R^n$, i.e. the set of subspaces whose Pl{\"u}cker coordinates are all nonzero and have the same sign.
An element in $\Grr_{k,n}^{> 0}$ can be represented as the column span of  an $n\times k$ matrix with strictly positive maximal minors.

\begin{definition}[cf. \cite{Shannon, Ziegler, RamirezAlfonsin, FRA, KarpWilliams, LLM}]
An arrangement $(V,\eta)$ is called \emph{cyclic} if:
\begin{itemize}
\item $V \in \Grr_{k,n}^{> 0}$,
\item $V + \langle \eta \rangle \in \Grr_{k+1,n}^{> 0}$, and
\item $\eta$ is \emph{positively oriented} with respect to $V$, which means that the first coordinate of the orthogonal projection of some, or equivalently every, representative $w \in \R^n$ of $\eta$ onto $V^{\perp}$ is positive.
\end{itemize}
\end{definition}

\subsubsection{Left and right cyclic polarized arrangements}\label{sec:LRCyclicDefs}
In \cite{LLM} we introduced two types of cyclicity for polarized arrangements; we review these below.

\begin{definition}\label{def:LeftRightCyclic}
Let $\cal{V} = (V,\eta,\xi)$ be a polarized arrangement. We say that $\V$ is \emph{left cyclic} if:
\begin{itemize}
\item $V + \langle \eta \rangle \in \Grr_{k+1,n}^{>0}$,
\item $(\xi,\id)(V) \in \Grr_{k,n+1}^{>0}$, and
\item $\eta$ is positively oriented with respect to $V$.
\end{itemize}
where $(\xi,\id)$ is the linear map from $V$ to $\R^{n+1}$ whose first coordinate is given by the linear functional $\xi$ on $V$. Similarly, we say that $\V$ is \emph{right cyclic} if:
\begin{itemize}
\item $V + \langle \eta \rangle \in \Grr_{k+1,n}^{>0}$,
\item $(\id,(-1)^k \xi)(V) \in \Grr_{k,n+1}^{>0}$, and
\item $\eta$ is positively oriented with respect to $V$.
\end{itemize}
\end{definition}

\begin{remark}
By \cite[Corollary 3.21]{LLM}, left cyclic arrangements $\cal{V}$ constitute an equivalence class of polarized arrangements as defined in \cite[Section 2.2.1]{LLM}, and thus the algebras $\Bt(\cal{V})$ for left cyclic $\cal{V}$ are all isomorphic. The same holds for right cyclic $\cal{V}$.
\end{remark}

An example of a right cyclic arrangement is given in Example~\ref{ex:n4k2}.  For more examples, see \cite[Corollary 3.21]{LLM}. 

\subsection{Sign variation}

\begin{definition}
For $\alpha \in \{+,-\}^n$, we say $\var(\alpha) = k$ if the signs in $\alpha$ change from $+$ to $-$ or $-$ to $+$ exactly $k$ times when reading from left to right (or from right to left).
\end{definition}

If $\alpha \in \{+,-\}^n$, we write $+\alpha$ and $-\alpha$ for the elements of $\{+,-\}^{n-1}$ with a plus or minus appended in the first entry; we define $\alpha+$ and $\alpha-$ similarly using the last entry.

\begin{definition}
For $\alpha \in \{+,-\}^n$, and given $0 \leq k \leq n$, we define $\var_l(\alpha) := \var(+\alpha)$ and $\var_r(\alpha) := \var(\alpha(-1)^k)$.
\end{definition}

The following results were demonstrated in \cite{LLM} using results of \cite{KarpWilliams}.

\begin{corollary}
A sign sequence $\alpha \in \{+,-\}^n$ is feasible for a left cyclic arrangement $(V,\eta,\xi)$ (with $\dim(V) = k$) if and only if $\var_l(\alpha) \leq k$ and is bounded if and only if $\var_l(\alpha) \geq k$.
\end{corollary}

\begin{corollary}
A sign sequence $\alpha \in \{+,-\}^n$ is feasible for a right cyclic arrangement $(V,\eta,\xi)$ (with $\dim(V) = k$) if and only if $\var_r(\alpha) \leq k$ and is bounded if and only if $\var_r(\alpha) \geq k$.
\end{corollary}

\subsection{Dots in regions and partial orders}

\subsubsection{Sign sequences and dots in regions}\label{sec:SignSeqAndOSzIdems}

We review an alternate combinatorial description of bounded feasible sign sequences $\alpha$ in the left and right cyclic cases from \cite{LLM}. Let $V_l(n,k)$ denote the set of $k$-element subsets of $\{0,\ldots,n-1\}$ and let $V_r(n,k)$ denote the set of $k$-element subsets of $\{1,\ldots,n\}$. Following Ozsv{\'a}th--Szab{\'o} \cite{OSzNew}, we draw elements of $V_l(n,k)$ as sets of $k$ dots in the regions $\{0,\ldots,n-1\}$ on the left of Figure~\ref{fig:Idems}; see the right of Figure~\ref{fig:Idems} for an example. Elements of $V_r(n,k)$ are drawn similarly as sets of $k$ dots in the regions $\{1,\ldots,n\}$ on the left of Figure~\ref{fig:Idems}.

\begin{figure}
\includegraphics[scale=0.62]{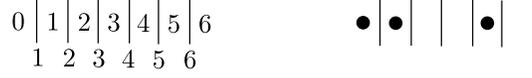}
\caption{Left: $n$ lines and $n+1$ regions between them. Right: a set of $3$ dots in the regions $\{0,1,4\} \subset \{0,\ldots,6\}$.}
\label{fig:Idems}
\end{figure}

\begin{definition}\label{def:StatesFromAlpha}[cf. Definition 3.24 of \cite{LLM}] \hfill
\begin{enumerate}[(i.)]
  \item \label{def:LeftIStatesFromAlpha} Let $\cal{V}$ be a left cyclic polarized arrangement so that $\cal{P}$ is the set of $\alpha \in \{+,-\}^n$ with $\var_l(\alpha) = k$.  There is a bijection  $\kappa_l\maps \cal{P}\to V_l(n,k)$ given by sending $\alpha \in \cal{P}$ to $\x_{\alpha} \subset \{0,\ldots,n-1\}$ with  $i \in \x_{\alpha}$ if there is a change after sign $i+1$ in the sequence $+\alpha$. The inverse sends $\x \in V_l(n,k)$ to $\alpha_{\x}$, obtained by omitting the first entry of the sign sequence defined by starting with $+$ as the leftmost entry (``step zero'') and writing $n$ signs to the right, introducing a sign change at step $i$ if and only if $i-1 \in \x$.

      \item \label{def:RightIStatesFromAlpha}Let $\cal{V}$ be a right cyclic polarized arrangement so that $\cal{P}$ is the set of $\alpha \in \{+,-\}^n$ with $\var_r(\alpha) = k$.  There is a bijection
           $\kappa_r\maps \cal{P} \to V_r(n,k)$ sending $\alpha$ to $\x_{\alpha} \subset \{1,\ldots,n\}$ defined by $i \in \x_{\alpha}$ if there is a change after sign $i$  in the sequence $\alpha(-1)^k$. The inverse sends $\x \in V_r(n,k)$ to $\alpha_{\x}$, obtained by omitting the last entry of the sign sequence defined by starting with $(-1)^k$ as the rightmost entry (``step zero'') and writing $n$ signs from right to left, introducing a sign change at step $i$ if and only if $n-i+1 \in \x$.
\end{enumerate}
\end{definition}

\subsubsection{The partial order for cyclic arrangements} \label{sec:partial-dots}

Let $(V,\eta,\xi)$ be a left cyclic polarized arrangement. We have a partial order on the set $\cal{P}$ of bounded feasible sign sequences for $(V,\eta,\xi)$ from Section~\ref{subsec:partial_order}. Identifying $\cal{P}$ with the set $V_l(n,k)$ of $k$-element subsets of $\{0,\ldots,n-1\}$ from Section~\ref{sec:SignSeqAndOSzIdems}, viewed as sets of dots in regions, we also have the lexicographic partial order on $V_l(n,k)$ generated by the relations $\x < \y$ when $\y$ is obtained from $\x$ by moving a dot one step to the right.

\begin{proposition}\label{prop:VandermondeOrders}[cf. Proposition 3.25 of \cite{LLM}]
 For a left cyclic arrangement, the partial order on $\cal{P}$ induced by $\xi$ agrees with the order induced from the lexicographic order on $V_l(n,k)$ from the bijection $\kappa_l \maps \cal{P} \to V_l(n,k)$.
\end{proposition}

An analogous result holds in the right cyclic case: for $\alpha \leftrightarrow \beta$ corresponding to $\x,\y \subset \{1,\ldots,n\}$, we have $\alpha < \beta$ if and only if $\y$ is obtained from $\x$ by moving a dot one step to the left.

\begin{corollary}\label{cor:BijectionsAgree}[cf. Corollary 3.26 of \cite{LLM}]
If $(V,\eta,\xi)$ is a left cyclic polarized arrangement and we have $\alpha \in \cal{P}$, then the element $\mathbbm{x} \in \mathbb{B}$ associated to $\alpha$ in Section~\ref{subsec:partial_order} is obtained from the element $\x \in V_l(n,k)$ associated to $\alpha$ in Definition~\ref{def:StatesFromAlpha} by adding one to each $i \in \x$.
\end{corollary}

If $(V,\eta,\xi)$ is a right cyclic arrangement, one can show similarly that the elements $\mathbbm{x} \in \mathbb{B}$ and $\x \in V_r(n,k)$ associated to $\alpha$ agree as subsets of $\{1,\ldots,n\}$.

\subsection{Bounded and feasible cones for cyclic arrangements}

\begin{lemma}\label{lem:ConesForCyclic} Let $(V,\eta,\xi)$ be left cyclic and let $\alpha,\beta \in \mathcal{P}$; let $\x,\y$ be the corresponding elements of $V_l(n,k)$ (viewed as sets of $k$ dots in regions) and let $\mathbbm{x}$, $\mathbbm{y}$ be the corresponding elements of $\mathbb{B}$. Then:
\begin{itemize}
\item We have $\beta \in \mathcal{B}_{\mathbbm{x}}$ if and only if $\y$ is obtained from $\x$ by moving each dot of $\x$ some number of steps to the left, such that no dot moves far enough to reach the starting point of its leftmost neighbor.

\item We have $\beta \in \mathcal{F}_{\mathbbm{x}}$ if and only if $\y$ is obtained from $\x$ by moving each dot of $\x$ at most one step to the right.
\end{itemize}
\end{lemma}

\begin{proof}
By definition, $\mathcal{B}_{\mathbbm{x}}$ is the set of those regions $\beta$ that can be reached from $\alpha$ by successively crossing hyperplanes, none of which is incident with the maximal point of $\alpha$. Thus, by Corollary~\ref{cor:BijectionsAgree}, $\mathcal{B}_{\mathbbm{x}}$ is the set of those $\beta$ corresponding to elements $\y \in V_l(n,k)$ that can be reached by successively moving dots (one step at a time) such that the line immediately to the right of each dot of $\x$ is never crossed. This set of $\y$ is the same as the one described in the statement of the lemma.

Similarly, $\mathcal{F}_{\mathbbm{x}}$ is the set of those regions $\beta$ that can be reached from $\alpha$ by successively crossing hyperplanes, each of which is incident with the maximal point of $\alpha$. Thus, by Corollary~\ref{cor:BijectionsAgree}, $\mathcal{F}_{\mathbbm{x}}$ is the set of those $\beta$ corresponding to elements $\y \in V_l(n,k)$ that can be reached by successively moving dots (one step at a time) such that each step crosses the line immediately to the right of some dot of $\x$. This set of $\y$ is the same as the one described in the statement of the lemma.
\end{proof}

\section{Applications to bordered Heegaard Floer homology}\label{sec:BorderedFloerApps}

We now discuss various applications of the above results arising from the isomorphism exhibited in \cite{LLM} between $\tilde{B}(\cal{V})$ for left and right cyclic $\cal{V}$ and algebras introduced by Ozsv{\'a}th--Szab{\'o} related to bordered Heegaard Floer homology.

\subsection{Ozsv{\'a}th--Szab{\'o} algebras as hypertoric convolution algebras}
We recall the definition of the graded algebra $\B(n,k)$ from \cite{OSzNew} using the generators-and-relations description from \cite{MMW1}. Let $V(n,k)$ be the set of $k$-element subsets $\x \subset \{0,\ldots,n\}$.

\begin{definition}\label{def:SmallStep}
Let $\B(n,k)$ be the path algebra of the quiver with vertex set $V(n,k)$ and arrows
\begin{itemize}
\item for $1 \leq i \leq n$, $R_i$ from $\x$ to $\y$ and $L_i$ from $\y$ to $\x$ if $\x \cap \{i-1,i\} = \{i-1\}$ and $\y = (\x \setminus \{i-1\}) \cup \{i\}$,
\item for $1 \leq i \leq n$, $U_i$ from $\x$ to $\x$ for all $\x \in V(n,k)$
\end{itemize}
modulo the relations
\begin{enumerate}
\item\label{it:UCommute} $R_i U_j = U_j R_i$, $L_i U_j = U_j L_i$, $U_i U_j = U_j U_i$,
\item\label{it:RLU} $R_i L_i = U_i$, $L_i R_i = U_i$,
\item\label{it:DistantCommute} $R_i R_j = R_j R_i$, $L_i L_j = L_j L_i$, $L_i R_j = R_j L_i$ ($|i-j|>1$),
\item\label{it:TwoLinePass} $R_{i-1} R_{i} = 0$, $L_i L_{i-1} = 0$,
\item\label{it:UVanish} $U_i \Ib_{\x} = 0$ if $\x \cap \{i-1,i\} = \emptyset$.
\end{enumerate}
The relations are assumed to hold for any linear combination of quiver paths with the same starting and ending vertices and labels $R_i,L_i,U_i$ as described; $\Ib_{\x}$ denotes the trivial path at $\x \in V(n,k)$.  The elements $\Ib_{\x} \in \B(n,k)$ give a complete set of orthogonal idempotents. We define a grading on $\B(n,k)$ by setting $\deg(R_i) = \deg(L_i) = 1$ and $\deg(U_i) = 2$.
\end{definition}

\begin{remark}
In \cite{MMW2} the usual grading setup in bordered Heegaard Floer homology is shown to recover Ozsv{\'a}th--Szab{\'o}'s ``Alexander multi-grading'' by $\Z^n$ on $\B(n,k)$ (related to the multi-variable Alexander polynomial in the same way that our single $\Z$ grading is related to the single-variable Alexander polynomial).
\end{remark}

Recall from Section~\ref{sec:SignSeqAndOSzIdems} that $V_l(n,k)$ denotes the subset of $V(n,k)$ consisting of $k$-element subsets of $\{0,\ldots,n-1\}$ and that $V_r(n,k)$ denotes the set of $k$-element subsets of $\{1,\ldots,n\}$. Similarly, we let $V'(n,k)$ denote the set of $k$-element subsets of $\{1,\ldots,n-1\}$; by \cite[Section 3.6.1]{LLM}, for $(V,\eta,\xi)$ left or right cyclic, $\alpha \in \cal{K} \subset \cal{P}$ if and only if $\x_{\alpha} \in V'(n,k)$ as a subset of $V_l(n,k)$ or $V_r(n,k)$.

\begin{definition} \label{def:variants-OSz}
Define the following variants of the Ozsv{\'a}th--Szab{\'o} algebra by
\begin{align}
  \B_l(n,k) &= \bigoplus_{\x, \x' \in V_l(n,k)} \Ib_{\x} \cdot \B(n,k) \Ib_{\x'} \nn \\
  \B_r(n,k) &= \bigoplus_{\x, \x' \in V_r(n,k)} \Ib_{\x} \cdot \B(n,k) \Ib_{\x'} \nn \\
  \B'(n,k) &= \bigoplus_{\x, \x' \in V'(n,k)} \Ib_{\x} \cdot \B(n,k) \Ib_{\x'}. \nn
\end{align}
\end{definition}

In \cite[Section 3.6]{OSzNew}, Ozsv{\'a}th--Szab{\'o} define an anti-automorphism of $\B(n,k)$, restricting to anti-automorphisms of $\B_l(n,k)$, $\B_r(n,k)$ and $\B'(n,k)$, as follows.

\begin{definition} \label{def:PsiOS}
The anti-automorphism $\psi_{OSz}\colon \B(n,k) \to \B(n,k)^{{\rm opp}}$ sends the generators of Definition~\ref{def:SmallStep} by $R_i \mapsto L_i$, $L_i \mapsto R_i$, and $U_i \mapsto U_i$.
\end{definition}

\begin{theorem}[cf. Theorems 4.9, 4.13 of \cite{LLM}] \label{thm:CyclicIsom}
Let $(V,\eta,\xi)$ be left cyclic (respectively right cyclic) for a given $(n,k)$. There exists an isomorphism of graded algebras $\tilde{B}(\cal{V}) \cong \B_l(n,k)$ (respectively $\tilde{B}(\cal{V}) \cong \B_r(n,k)$) such that:
\begin{itemize}
\item the basic idempotent $e_{\alpha}$ of $\tilde{B}(\cal{V})$ for left cyclic $(V,\eta,\xi)$ is sent to the basic idempotent $\Ib_{\x_{\alpha}}$ of $\B_l(n,k)$ (respectively, the basic idempotent $e_{\alpha}$ for right cyclic $(V,\eta,\xi)$ is sent to $\Ib_{\x_{\alpha}} \in \B_r(n,k)$)
\item the isomorphism intertwines the $\Z[u_1,\ldots,u_n]$ action on $\Bt(\cal{V})$ with the $\Z[U_1,\ldots,U_n]$ action on $\B_l(n,k)$ (respectively $\B_r(n,k)$) under the identification $u_i \leftrightarrow U_i$.
\item the isomorphism intertwines the anti-involution $\psi^{\cal{V}}$ of $\Bt(\cal{V})$ from Definition~\ref{def:hypertoric-involution} with the involution $\Psi_{OSz}$ from Definition~\ref{def:PsiOS}.
\end{itemize}
\end{theorem}

\begin{remark}
While in this paper we work only with $\Z$ gradings, in \cite{LLM} we also considered a multi-grading by $\Z^n$ on $\Bt(\cal{V})$ recovering Ozsv{\'a}th--Szab{\'o}'s Alexander multi-grading for cyclic $\cal{V}$.
\end{remark}

\subsection{Oszv{\'a}th-Szab{\'o} algebras are affine quasi hereditary}
Theorem~\ref{thm:affine} established that hypertoric convolution algebras are affine (polynomial) quasi hereditary; Corollary~\ref{cor:hyp-affine-cellular}  shows these are affine cellular with anti-automorphism $\psi^{\cal{V}}$.  The isomorphism from Theorem~\ref{thm:CyclicIsom} shows that the  Oszv{\'a}th-Szab{\'o} algebras  $\cal{B}_l(n,k)$ (respectively $\cal{B}_r(n,k)$) are isomorphic to hypertoric convolution algebras associated to left (respectively right) cyclic arrangements. This isomorphism intertwines the anti-automorphism $\psi^{\cal{V}}$ with the anti-automorphism $\Psi_{OSz}$ from Definition~\ref{def:PsiOS}.
Section~\ref{sec:partial-dots} identifies the partial order of any left (resp. right) cyclic polarized arrangement with the lexicographic (resp. reverse lexicographic order) on the set of bases $\mathbb{B}$.

\begin{theorem} \label{thm_OSz-aqh}
The Oszv{\'a}th-Szab{\'o} algebras $\cal{B}'(n,k)$, $\cal{B}_l(n,k)$, $\cal{B}_r(n,k)$, with partially ordered sets $(V'(n,k), \leq_{{\rm lex}})$ or $(V'(n,k), \leq_{{\rm rev lex}})$, $(V_l(n,k), \leq_{{\rm lex}})$, and $(V_r(n,k), \leq_{{\rm rev lex}})$ and anti-automorphism $\Psi_{OSz}$ from Definition~\ref{def:PsiOS}
are polynomial quasi hereditary, standardly stratified, and affine cellular algebras (the partial orders indicate the ordering used in constructing polynomial hereditary chains; see Remark~\ref{rem:hereditary-po}).
\end{theorem}

\begin{proof}
The only case that does not follow from the remarks above is that of $(\cal{B}'(\cal{V}), \leq_{{\rm lex}})$ and $(\cal{B}'(\cal{V}), \leq_{{\rm rev lex}})$.

With respect to the lexicographic or reverse lexicographic order arising in the left/right cyclic case, there exists a choice of total order   $\cal{P}=\{ \alpha_1,\dots ,\alpha_m\}$ with $\alpha_i < \alpha_j$ only if $i<j$ with the property that there
exists some $\ell \in \{1, \dots, n\}$ such that $\alpha_i \in \cal{P}\setminus \cal{K}$ for $i <\ell$   and $\alpha_j \in \cal{K}$ for $\ell \leq j$.

For $1 \leq j \leq m$, let $e^{(j)} = \sum_{i=j}^m e_{\alpha_i}$. By \cite[Theorem 6.7]{Klesh-affine}
the polynomial quasi hereditary structure on $H:=\Bt(\cal{V})$ implies that
\[
H
=
H e^{(1)} H
\supsetneq H e^{(2)} H
\supsetneq \dots
\supsetneq H e^{(m)} H \supsetneq 0
\]
is a polynomial stratifying chain of ideals.
Let $e:= \sum_{\alpha \in \cal{K}}e_{\alpha} = \sum_{i=\ell}^m e_{\alpha_i}$ so that $\Bt'(\cal{V}):= e\Bt(\cal{V})e =: eHe$.
Then $e e^{(i)}e=e$ for $i\leq \ell$  and $e e^{(j)} e = e^{(j)}$ for $\ell \leq j$, so that
\[
\Bt'(\cal{V}):= eHe = eH e^{(\ell)} He
\supsetneq
eH e^{(\ell+1)} He
\supsetneq eH e^{(\ell+2)} He
\supsetneq \dots
\supsetneq eH e^{(m)} He \supsetneq 0
\]
is  a polynomial stratifying chain of ideals for $\Bt'(\cal{V})$.
\end{proof}
We will describe the resulting standard modules for $\B_l(n,k)$ more explicitly in Section~\ref{subsec:OSz-standard}. The above corollary has a number of homological implications and provides additional insights into the Grothendieck groups of these algebras  elaborated on  in the following sections.

\subsection{Finite global dimension}
The \emph{global dimension} $\gldim(H)$ of an algebra $A$ is the supremum over all $A$-modules of the minimum length of a projective resolution.
By \cite[Theorem 5.24]{Klesh-affine} the global dimension of an affine highest weight category is at most
 $2\ell(\Pi) + d_{\Pi}$, where  $d_{\alpha} = {\gldim B_{\alpha}}$, $d_{\Pi}=\max\{ d_{\alpha} \mid \alpha \in \Pi \}$,
and
\[
\ell(\Pi) = \max\{l \mid \text{there exist elements $\alpha_0<\alpha_1 < \dots < \alpha_{l}$ in $\Pi$} \}.
\]
For the Ozsv{\'a}th--Szab{\'o} algebras we have
$d_{\Pi} = \gldim(\Z[y_1,\dots, y_{k}]) = k$.
We also have
 \begin{align*}
\ell(\Pi_{\cal{B}_{\ell}'(n,k)})  =  (n-k-1)\cdot k,
\qquad
\ell(\Pi_{\cal{B}_{\ell}(n,k)})=\ell(\Pi_{\cal{B}_{r}(n,k)}) =  (n-k)\cdot k,
\qquad
\ell(\Pi_{\cal{B}(n,k)})  &=  (n-k+1)\cdot k.
\end{align*}

\begin{corollary} \label{cor:OSz-finite-gd}
The global dimensions of Ozsv{\'a}th--Szab{\'o} bordered algebras satisfy the inequalities
\begin{align*}
\gldim \cal{B}'(n,k)  &\leq  2(n-k-1)\cdot k +k
\\
\gldim \cal{B}_{\ell}(n,k) =\gldim \cal{B}_{r}(n,k) &\leq   2(n-k)\cdot k +k
\\
\gldim  \cal{B}(n,k)   &\leq  2(n-k+1)\cdot k +k
\end{align*}
\end{corollary}

\subsection{Categorification of \texorpdfstring{$\mf{gl}(1|1)$}{gl(1|1)} canonical bases} \label{sec:bases-gl}
In \cite{ManionDecat} it is shown that the algebras $\B_{l}(n,k)$ categorify the tensor product $V^{\otimes n}$ of the two-dimensional $U_{q}(\mf{gl}(1|1))$-module $V$ of highest weight $\epsilon_1=(1,0)$ in the weight lattice $\Z^2$ of $\mf{gl}(1|1)$.  More explicitly, using the conventions of
\cite[Section 8]{LaudaManion} we have an isomorphism
\begin{align*}
   \bigoplus_{k=0}^n K_0^{\C(q)}(\B_{l}(n,k)) &\longrightarrow V^{\otimes n} \\
   [P_{\x}] &\mapsto v^{\diamondsuit}_{\x}
\end{align*}
where $\{ v^{\diamondsuit}_{\x} \}$ denotes the $U_q(\mf{gl}(1|1))$ canonical or crystal basis of $V^{\otimes n}$ from \cite[Section 4.3]{Sar-tensor}.
Under this isomorphism, in \cite[Section 8]{LaudaManion} the classes in $K_0^{\C(q)}(\B_l(n,k))$ of several families of modules over $\B_l(n,k)$ were identified with elements of various bases for $V^{\otimes n}$. Indecomposable projectives and simples gave the $U_q(\mf{gl}(1|1))$ canonical basis for $V^{\otimes n}$ and its dual with respect to the bilinear pairing on $V^{\otimes n}$ induced by \eqref{eq:form-HOM}.

Sartori gave an alternative categorification of $V^{\otimes n}$ using $\mf{q}$-presentable quotients of parabolic category $\cal{O}$ \cite{Sar-diagrams,Sar-tensor}; this categorification can be described using categories of modules over certain  properly stratified algebras.  In \cite[Theorem 7.6]{LaudaManion} it was shown that the algebras $\B_l(n,k)$ are graded flat deformations of Sartori's algebras $\mathbb{A}_{n,k}$ over the ring of symmetric functions in $k$ variables. Inflations of Sartori's proper standard modules from \cite{Sar-diagrams} along the quotient map were identified with the dual of the standard basis of $V^{\otimes n}$ with respect to the bilinear form on $K_0^{\C(q)}(\B_l(n,k))$; inflations of Sartori's standard and indecomposable projective modules gave slightly less natural elements of $V^{\otimes n}$. In particular, while we would like to categorify at least the canonical and standard bases of $V^{\otimes n}$ and their duals using $\B_l(n,k)$, only three of these bases were categorified in \cite{LaudaManion}; a categorification of the (most natural) standard basis was missing, unlike in Sartori's picture where his algebras admit a different bilinear form, see Remark~\ref{rem:form-compare} .

Sartori's categorification of the standard basis comes from the properly stratified structure
of his algebras; similarly, we show here that the standard modules over $\B_l(n,k)$ coming from its affine quasi-hereditary structure categorify the standard tensor-product basis of $V^{\otimes n} \cong K_0^{\C(q)}(\B_l(n,k))$.

\begin{lemma}\label{lem:Sartori-inflate}
Let $\cal{V}$ be a left cyclic polarized arrangement.
The inflation of Sartori's proper standard module $\bar{\Delta}^S(\alpha)$ along the quotient map $\Bt(\cal{V}) \to \mathbb{A}_{n,k}$ from \cite{LaudaManion} is isomorphic to the proper standard module $\bar{\Vt}(\alpha)$ for $\Bt(\cal{V})$.
\end{lemma}

\begin{proof}
Theorem 7.6 in \cite{LaudaManion} identifies $\mathbb{A}_{n,k}$ as the quotient of the algebra $\B_l(n,k) \cong \Bt(\cal{V})$ by the ideal generated by the elementary symmetric functions $e_1,\ldots,e_k$ in the central elements $U_i := \oplus_{|\x| = k} (U_i)\cdot \mathbf{I}_{\x}$ of $\B_l(n,k)$.   The basis  elements of the proper standard module $\bar{\Vt}(\alpha)$ are indexed by $\beta \in\cal{P}$ with $\alpha \geq \beta$ and $\Delta_{\alpha}\cap \Delta_{\beta} \neq \emptyset$, and they are fixed under the quotient map.  The claim then follows from
  explicit description of Sartori's proper standard modules $\bar{\Delta}^S(\alpha)$ given in \cite[Equation 5.59]{Sar-diagrams} in terms of the combinatorics of ``fork diagrams.''  To translate the indexing sets for Sartori's bases into the indexing set of proper standard modules for $\Bt(\cal{V})$ we combining the partial order preserving bijections from Sartori's wedge sequences to left $I$-states $\x$ from \cite[Remark 3.1]{LaudaManion}, with the bijection $\kappa_l^{-1}$ from $I$-states to $\cal{P}$ from \cite[Definition 3.24 (i)]{LLM}; this gives a partial order preserving bijection from oriented fork diagrams to $\cal{P}$.

\end{proof}

\begin{theorem} \label{thm:bases}
Under the identification $K_0^{\C(q)}(\Bt(\cal{V}))\cong K_0^{\C(q)}(\B_l(n,k)) \cong V^{\otimes n}$ for left cyclic polarized arrangements $\cal{V}$, the bases  for $K_0^{\C(q)}(\Bt(\cal{V}))$  map to bases of the $U_q(\mf{gl}(1|1))$-module $V^{\otimes n}$ as follows:
\begin{align*}
 \{ \text{indecomposable projective modules $\Pt_{\alpha}$} \}
 &\leftrightarrow
 \{  \text{canonical basis elements $v_{\alpha}^{\diamondsuit}$}\}
\\
 \{ \text{simple modules $L_{\alpha}$} \}
&\leftrightarrow
  \{  \text{(Ozsv{\'a}th--Szab{\'o}) dual canonical basis elements $v_{\alpha}^{\heartsuit}$}\}
\\
 \{ \text{standard modules $\Vt_{\alpha}$} \}
&\leftrightarrow
  \{  \text{standard basis elements $v_{\alpha}$}\}
\\
 \{ \text{proper standard modules $\bar{\Vt}_{\alpha}$} \}
&\leftrightarrow
  \{  \text{(Ozsv{\'a}th--Szab{\'o}) dual standard basis elements $v_{\alpha}^{\ast}$}\}
\end{align*}
(see \cite[Section 4]{Sar-tensor} and \cite[Section 8]{LaudaManion} for a description of these bases).
In particular, the hypertoric canonical basis elements from Theorem~\ref{thm:canonical} map to the corresponding basis elements $v_{\a}^{\diamondsuit}$ of the $U_q(\mf{gl}(1|1))$-module $V^{\otimes n}$ from \cite[Section 4.3]{Sar-tensor}.
\end{theorem}

\begin{proof}
The form $(-,-)$ on $K_0^{\C(q)}(\Bt(\cal{V}))$ agrees with the form on $V^{\otimes n}$ from \cite[Section 8.7]{LaudaManion} where it was shown that under the identification $K_0(\B_l(n,k)) \cong V^{\otimes n}$, the bases $\{[\Pt_{\alpha}]\}$, $\{[L_{\alpha}]\}$, $\{[\bar{\Vt}_{\alpha}]\}$ agree with the canonical, dual canonical, and dual standard bases $\{v_{\alpha}^{\diamondsuit}\}$, $\{v_{\alpha}^{\heartsuit}\}$, $\{v_{\alpha}^{\ast}\}$ of $V^{\otimes n}$\footnote{In \cite{LaudaManion}, the bases $\{v_{\alpha}^{\heartsuit}\}$ and $\{v_{\alpha}^{\ast}\}$ were called $\{v_{\alpha}^{\heartsuit \heartsuit}\}$ and $\{v_{\alpha}^{\clubsuit \clubsuit}\}$ respectively; other bases $\{v_{\alpha}^{\heartsuit}\}$ and $\{v_{\alpha}^{\clubsuit}\}$ were also considered following \cite{Sar-tensor}, dual to $\{v_{\alpha}^{\diamondsuit}\}$ and $\{v_{\alpha}\}$ under a different bilinear form.} (the identification of $[\bar{\Vt}_{\alpha}]$ with $v_{\alpha}^{\ast}$ makes use of Lemma~\ref{lem:Sartori-inflate} identifying the inflation of Sartori's proper standard modules with the proper standard modules of $\Bt(\cal{V})$). Since $[\Vt_{\alpha}]$ is dual to $[\bar{\Vt}_{\alpha}]$ under $(-,-)$, it gets identified with the dual $v_{\alpha}$ of $v_{\alpha}^{\ast}$ in $V^{\otimes n}$.

 \end{proof}

Theorem~\ref{thm:bases} provides a new description of the canonical basis $v_{\alpha}^{\diamondsuit}$.  First we need the following lemma.

\begin{lemma}
Let $\alpha \in \cal{P}$. In $K_0(\tilde{B}(\cal{V})) = K_0(\tilde{A}(\cal{V}^{\vee}))$, we have
\[
[\Pt_{\alpha}] = \sum_{\mathbbm{x} \in \mathbb{B} \,\, : \,\, H_{\mathbbm{x}} \cap \Delta_{\alpha} \neq \emptyset} q^{d_{\alpha \mu(\mathbbm{x})}} [\Vt_{\mu(\mathbbm{x})}]
\]
(recall that $d_{\alpha \mu(\xx)}$ denotes the number of sign changes required to transform $\alpha$ into $\mu(\xx)$).
\end{lemma}

\begin{proof}
By property {\sf PHW1}, we have $[\Pt_{\alpha}] = [\Vt_{\alpha}] + [K_{> \alpha}]$. By  the proof of Lemma~\ref{lem:Kfiltration},
\[
[K_{> \alpha}] = \sum_{\mathbbm{x} \in \mathbb{B} \,\, : \,\, \alpha \in \cal{B}^{\vee}_{\mathbbm{x}^c}} q^{d_{\alpha \mu^{\vee}(\mathbbm{x}^c)}} [\Vt_{\mu^{\vee}(\mathbbm{x}^c)}];
\]
recall that $\mu^{\vee}(\mathbbm{x}^c) = \mu(\mathbbm{x})$. The lemma now follows from Lemma~\ref{lem:cone-adjacent}.
\end{proof}

\begin{corollary}
  For $\alpha \in \cal{P}$, let $v^{\diamondsuit}_{\alpha}$ and $v_{\alpha}$ be the corresponding canonical and standard basis elements of $V^{\otimes n}$. We have
\[
v^{\diamondsuit}_{\alpha} = \sum_{\mathbbm{x} \in \mathbb{B} \,\, : \,\, H_{\mathbbm{x}} \cap \Delta_{\alpha} \neq \emptyset} q^{d_{\alpha \mu(\mathbbm{x})}} v_{\mu(\mathbbm{x})}.
\]
\end{corollary}

\subsection{Standard modules for Ozsv{\'a}th--Szab{\'o} algebras } \label{subsec:OSz-standard}

The description of standard modules over $\Bt(\cal{V})$ given in Section~\ref{sec:BSide} above can be translated into modules for $\B_l(n,k)$ via the isomorphism from Theorem~\ref{thm:CyclicIsom}.     Recall from Section~\ref{sec:SignSeqAndOSzIdems} the   lexicographic partial order on $V_l(n,k)$ generated by the relations $\x < \y$ when $\y$ is obtained from $\x$ by moving a dot one step to the right. Let
\[
\cal{R}_{\x} := \left\{
\y \in V_l(n,k) \mid \x \leq \y
\right\}
\]
denote the set of dot sequences obtained from $\x$ by sliding dots to the right. Then define an idempotent in $\B_l(n,k)$ by
\[
\mathbf{\cal{I}}_{\x} := \sum_{\y < \x} \Ib_{\y}.
\]
The standard modules $\Vt_{\x}$ (see Remark~\ref{rem:VtVersusDelta}), expressed in the language of Ozsv{\'a}th--Szab{\'o}, take the form
\[
\Vt_{\x} :=  (\B_l(n,k) \Ib_{\x}) /  (\B_l(n,k) \mathbf{\cal{I}}_{\x}\B_l(n,k) \Ib_{\x}) =  \left( \bigoplus_{\y \in \cal{R}_{\x}} \Ib_{\y} \B_l(n,k) \Ib_{\x} \right) / (U_{i+1} \mid i \notin \x  )
\]
Intuitively, the standard module $\Vt_{\x}$ is the Ozsv{\'a}th--Szab{\'o} projective module $P_{\x} := \B_l(n,k) \Ib_{\x}$ quotiented by any element that factors through an idempotent where dots in $\x$ slide to the left.


\begin{remark}
Since $\Vt_{\x}$ is not projective, it cannot be described as the bordered Heegaard Floer invariant $\widehat{CFD}$ of any (bordered sutured) 3-manifold. The invariant $\widehat{CFD}$ can see the projective resolution of $\Vt_{\x}$ but not $\Vt_{\x}$ itself, so that algebraic arguments working directly with standard modules $\Vt_{\x}$ can sometimes give more information than what is immediately visible from the Heegaard diagrams. We do not know any bordered sutured 3-manifold whose invariant $\widehat{CFA}$ gives the standard module $\Vt_{\x}$ directly, although we do not have an argument that no such 3-manifold exists.
\end{remark}

\subsection{Ext groups between standard modules as homology of a strands algebra}\label{sec:ExtAndStrandsHomology}

We will show in this section that for left cyclic $\cal{V}$, the direct sum of the Ext groups between standard modules over $\tilde{B}(\cal{V})$ is isomorphic (disregarding the multiplicative structure for now and working over $\F_2$) to the homology of a higher tensor product algebra
\[
\A(\Zc^{\sta}_n) \cong \A(\Zc^{\sta}_1) \ootimes \cdots \ootimes \A(\Zc^{\sta}_1)
\]
with a natural Heegaard Floer interpretation as a strands algebra, where $\ootimes$ is taken in the sense of \cite{ManionRouquier}. The algebra $\A := \A(\Zc^{\sta}_n)$ can be given a concrete description as follows; we will discuss its Heegaard Floer origins in more detail in Section~\ref{sec:HeegaardDiags}.

\begin{definition}
Let $\A(n,k)$ be Lipshitz--Ozsv{\'a}th--Thurston's unmatched strands algebra over $\F_2$ from \cite[Section 3.1]{LOTBorderedOrig}; recall that $\A(n,k)$ has an idempotent $I_S$ for each $k$-element subset $S$ of $\{1,\ldots,n\}$. We have
\[
\A = \A(\Zc^{\sta}_n) := \frac{\A(n,k) \otimes \F_2[U_1,\ldots,U_n]}{J}
\]
where $J$ is the ideal of $\A(n,k) \otimes \F_2[U_1,\ldots,U_n]$ generated by $U_i$ times any idempotent $I_S$ with $i \notin S$ as well as $U_i$ times any strands basis element with a moving strand starting or ending at position $i$.
\end{definition}

Let $S,T \subset \{1,\ldots,n\}$ and consider $I_S \cdot \A \cdot I_T$. As in \cite[Section 4]{LOTBimodules}, let $h$ be the unique homology class in $H_1(Z,a)$ such that
\[
\partial h = T - S,
\]
where $Z$ is an interval and $a$ is a subset of $Z$ consisting of $n$ distinct points labeled $1,\ldots,n$ in order (we identify $S,T$ with subsets of $a$).  Write $S \cap T \cap \mathrm{int}(\supp(h))= \{i_1, \ldots, i_a\}$. The complex $I_S \cdot \A \cdot I_T$ can be written as
\[
\bigoplus_{C \subset \{i_1,\ldots,i_a\}} \left(I_{S \setminus C} \cdot \A(n-|C|,k-|C|) \cdot I_{T \setminus C} \right) \otimes \F[U_j | j \in C] \otimes \F[U_i: i \in S \cap T \setminus \{i_1,\ldots,i_a\}];
\]
indeed, each basis element of $I_S \cdot \A \cdot I_T$ is or is not divisible by $U_{i_1},\ldots,U_{i_a}$ (diagrammatically: there is or is not at least one strand wrapping around the $i^{th}$ cylinder).

For the summand $C = \{i_1,\ldots,i_a\}$, we get
\[
\left(I_{S \setminus \{i_1,\ldots,i_a\}} \cdot \A(n-a,k-a) \cdot I_{T \setminus \{i_1,\ldots,i_a\}} \right) \otimes \F[U_{i_1}, \ldots, U_{i_a}].
\]
By \cite[proof of Proposition 4.2]{LOTBimodules}, the homology of $I_{S \setminus \{i_1,\ldots,i_a\}} \cdot \A(n-a,k-a) \cdot I_{T \setminus \{i_1,\ldots,i_a\}}$ is zero if $h$ has multiplicity $\geq 2$ anywhere or has negative multiplicity anywhere, and the homology is 1-dimensional otherwise. In the second case, we get a summand of
\[
\F[U_i: i \in S \cap T \setminus \{i_1,\ldots,i_a\}] \otimes \F[U_{i_1}, \ldots, U_{i_a}] = \F[U_i : i \in S \cap T]
\]
in the homology of $I_S \cdot \A \cdot I_T$.

For any summand $C \neq \{i_1,\ldots,i_a\}$, the homology of $I_{S \setminus C} \cdot \A(n-|C|,k-|C|) \cdot I_{T \setminus C}$ is zero by \cite[proof of Proposition 4.2]{LOTBimodules}.

\begin{lemma}\label{lem:WhenIsAlgZero}
Viewing $S$ and $T$ as elements of the set $\mathbb{B}$, the homology class $h$ above has multiplicity $2$ or negative multiplicity somewhere if and only if $\mathcal{B}_S \cap \mathcal{F}_T = \emptyset$.

\end{lemma}

\begin{proof}
Let $\x, \y$ be the elements of $V_l(n,k)$ corresponding to $S,T$, viewed as sets of $k$ dots in regions. The class $h$ has all its multiplicities in $\{0,1\}$ if and only if $\y$ is obtained from $\x$ by moving each dot some number of steps to the left, such that each dot moves at most as far as the starting point of its leftmost neighbor. This holds if and only if there exists $\z \in V_l(n,k)$ such that:
\begin{itemize}
\item $\z$ is obtained from $\x$ by moving each dot some number of steps to the left, such that no dot moves far enough to reach the starting point of its leftmost neighbor, and
\item $\y$ is obtained from $\z$ by moving each dot at most one step to the left.
\end{itemize}
By Lemma~\ref{lem:ConesForCyclic}, the set of such $\z$ is in bijection with $\mathcal{B}_S \cap \mathcal{F}_T$.
\end{proof}

\begin{proposition}\label{prop:BFIntersectionForCyclic}
If $\cal{V}$ is left cyclic and $\alpha,\beta \in \cal{P}$ with $\cal{B}_{\xx_{\a}} \cap \cal{F}_{\xx_{\b}} \neq \varnothing$, then we have $\alpha(i) \neq \beta(i)$ for all $i \in \xx_{\a}^c \cap \xx_{\b}$.
\end{proposition}

\begin{proof}
By Lemma~\ref{lem:ConesForCyclic}, for any $\gamma \in \cal{B}_{\xx_{\a}} \cap \cal{F}_{\xx_{\b}}$, the concatenation of a taut path from $\b$ to $\gamma$ with a taut path from $\gamma$ to $\a$ can be viewed as a motion of dots in which dots are always moving to the right, never the left, and is thus taut itself. By Lemma~\ref{lem:BFIntersection}, if $i \in \xx_{\a}^c \cap \xx_{\b}$ then there exist taut paths from $\b$ to $\a$ through regions $\gamma$ with both $\gamma(i) = +$ and $\gamma(i) = -$. It follows that $\a(i) \neq \b(i)$.
\end{proof}

\begin{theorem}\label{thm:ExtStrands}
The homology $H(\A(\Zc^{\sta}_n))$ is isomorphic (disregarding multiplication) to the direct sum of Ext groups, with coefficients in $\F_2$, between standard modules over $\tilde{B}(\cal{V})$ for left cyclic $\cal{V}$.
\end{theorem}

\begin{proof}
For $(S,T)$, the corresponding summand of the homology algebra is zero if and only if the corresponding Ext group is zero, by Lemma~\ref{lem:WhenIsAlgZero}. When the summand is nonzero, it is $\F_2[U_i | i \in S \cap T]$, agreeing with the Ext group over $\F_2$ by Proposition~\ref{prop:BFIntersectionForCyclic} and Corollary~\ref{cor:BTildeExtComputation}.
\end{proof}

\section{Heegaard diagrams}\label{sec:HeegaardDiags}

In this section we show how Heegaard Floer ideas offer additional insight on multiplicative structure in Theorem~\ref{thm:ExtStrands}; along a way we see how the projective resolutions of standard modules from Section~\ref{sec:BSide} (in the cyclic case) appear naturally in Heegaard Floer homology.

\subsection{Arc diagrams and sutured surfaces}

\subsubsection{Arc diagrams}

The algebra $\A(\Zc^{\sta}_n)$ above is part of a general family of dg algebras $\A(\Zc)$, called strands algebras, associated to arc diagrams $\Zc$ in bordered sutured Floer homology.

\begin{definition}[cf. \cite{Zarev}, \cite{ManionRouquier}]
An arc diagram (or chord diagram) $\Zc$ is a compact oriented 1-manifold $Z$ with boundary (disjoint union of intervals and circles) equipped with a 2-to-1 matching of finitely many points in its interior (while Zarev does not allow closed circles in arc diagrams, the strands algebra construction still makes sense in this generality as discussed in \cite{ManionRouquier}).
\end{definition}

\begin{figure}
\includegraphics[scale=0.6]{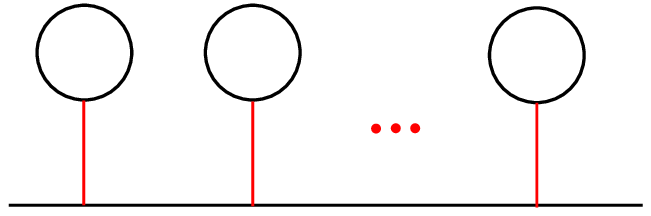}
\caption{The diagram $\Zc^{\sta}_n$, with $n$ circles.}
\label{fig:StandardDiag}
\end{figure}

\begin{figure}
\includegraphics[scale=0.6]{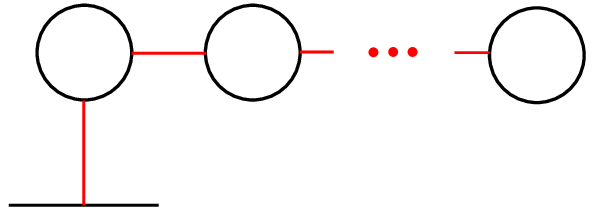}
\caption{The diagram $\Zc^{\can}_n$, with $n$ circles.}
\label{fig:CanonicalDiag}
\end{figure}

\begin{remark}
One can show that $\A(\Zc^{\sta}_n)$, as defined above, is isomorphic to the strands algebra of the arc diagram shown in Figure~\ref{fig:StandardDiag} (pairs of matched points are connected by red arcs). Note that, topologically, $\Zc^{\sta}_n$ is formed from $n$ copies of $\Zc^{\sta}_1$ by simple end-to-end gluings; correspondingly, the results of \cite{ManionRouquier} imply that
\[
\A(\Zc^{\sta}_n) \cong \A(\Zc^{\sta}_1) \ootimes \cdots \ootimes \A(\Zc^{\sta}_1)
\]
where $\ootimes$ is defined using the 2-representation structure on strands algebras $\A(\Zc)$ from \cite{ManionRouquier}.

In fact, the Ozsv{\'a}th--Szab{\'o} algebras we have been studying are also related to strands algebras $\A(\Zc)$, this time for the arc diagram $\Zc^{\can}_n$ shown in Figure~\ref{fig:CanonicalDiag}. As shown in \cite{LP} and independently in \cite{MMW2}, Ozsv{\'a}th--Szab{\'o}'s algebra $\B_l(n,k)$ is the homology of $\A(\Zc^{\can}_n)$, which is a formal dg algebra (similar results hold for the other variants of the Ozsv{\'a}th--Szab{\'o} algebra).
\end{remark}

\subsubsection{Sutured surfaces}

In bordered sutured Floer homology, arc diagrams are used as combinatorial representatives of sutured surfaces, which are surfaces with extra boundary structure.

\begin{definition}
A sutured surface $(F,S^+,S^-,\Lambda)$ is a compact oriented surface $F$ with boundary equipped with a finite subset $\Lambda$ of $\partial F$ such that each component of $\partial F$ has an even number of points of $\Lambda$, together with a partition of the (closures of) the connected components of $\partial F - \Lambda$ into two sets $S^+$ and $S^-$, such that each point of $\Lambda$ is contained in both an $S^+$ component and an $S^-$ component.
\end{definition}

An example of a sutured surface is shown in Figure~\ref{fig:SuturedSurface}; following Zarev \cite{Zarev}, we draw $S^+$ in orange and $S^-$ in black.

\begin{figure}
\includegraphics[scale=0.6]{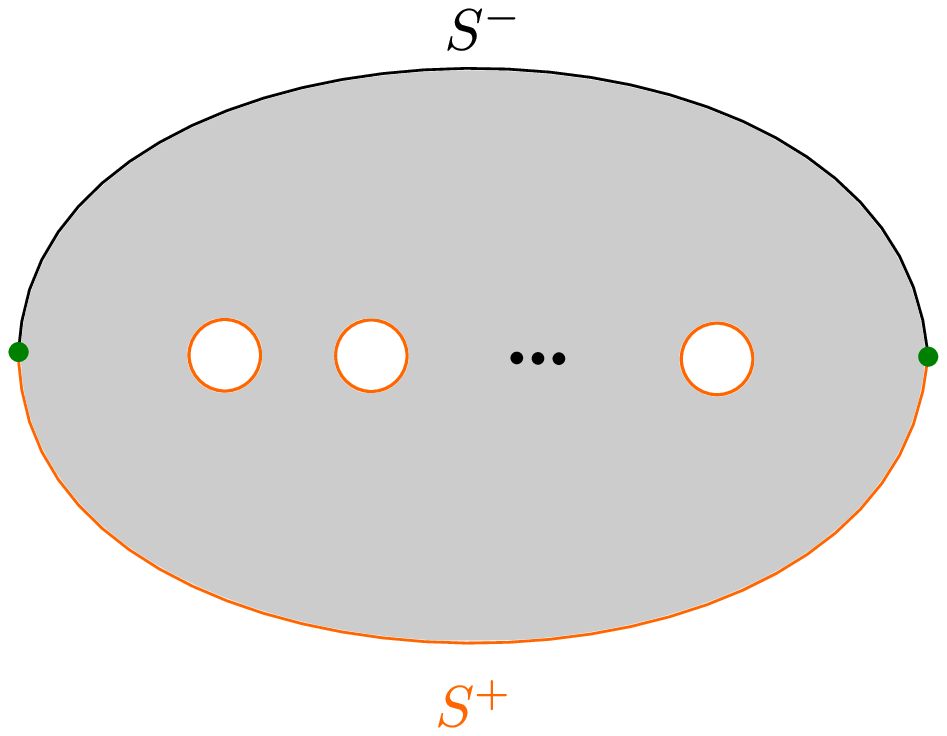}
\caption{The sutured surface represented by both $\Zc^{\sta}_n$ and $\Zc^{\can}_n$.}
\label{fig:SuturedSurface}
\end{figure}

To get a sutured surface from an arc diagram $\Zc$, one adds 2-dimensional 1-handles to $Z \times \{1\} \subset Z \times [0,1]$ according to the matching to obtain $F$; one can then take $\Lambda = \partial Z \times \{0\}$, $S^+ = Z \times \{0\}$, and $S^-$ to be the rest of the boundary of $F$. One says that the arc diagram is a representative for the sutured surface; a sutured surface $(F,S^+,S^-,\Lambda)$ can be represented by an arc diagram if and only if each component of $F$ (not $\partial F$) intersects both $S^+$ and $S^-$ nontrivially (in particular, no component of $F$ can be closed). There is a standard way to modify an arc diagram, known as an arc slide; two arc diagrams related by an arc slide represent the same sutured surface, and any two arc diagrams representing the same sutured surface can be related by a sequence of arc slides.

\begin{remark}
The diagrams $\Zc^{\sta}_n$ and $\Zc^{\can}_n$ both represent the sutured surface shown in Figure~\ref{fig:SuturedSurface}. One can view this surface as a disk minus a neighborhood of $n$ interior points (e.g. tangle endpoints for a tangle in $D^2 \times [0,1]$); there are two points of $\Lambda$ on the outer boundary, and all the internal circles are in $S^+$. (For this reason, the algebras of these diagrams are useful when trying to compute knot Floer homology using local tangle decompositions.) Since they represent the same sutured surface, $\Zc^{\sta}_n$ and $\Zc^{\can}_n$ should be related by a sequence of arc slides, and indeed there is a natural sequence of arc slides relating these two arc diagrams.
\end{remark}

\subsection{Derived equivalences and Heegaard diagrams}\label{sec:HDDerivedEq}

\begin{theorem}[Auroux \cite{Auroux}] If $\Zc$ represents a sutured surface $(F,S^+,S^-,\Lambda)$, the derived category of $\A(\Zc)$ is equivalent to the derived partially wrapped Fukaya category of $\sqcup_k \Sym^k(F)$, with stops for the partial wrapping determined by the sutured structure.
\end{theorem}

It follows that if $\Zc$ and $\Zc'$ are related by a sequence of arc slides the analogue of handle slides when one is sliding open arcs instead of closed attaching circles), then $\A(\Zc)$ and $\A(\Zc')$ are derived equivalent; this follows abstractly from Auroux's work and is no longer conjectural despite the careful language used by Auroux at the time. Furthermore, given an arc-slide sequence, bordered Floer homology offers holomorphic-curve methods for understanding these derived equivalences explicitly, so that one may use them for computations these methods are a bit different from Auroux's, which are based on relating both sides with an abstract partially wrapped Fukaya category).

In particular, since $\Zc^{\sta}_n$ and $\Zc^{\can}_n$ can be related by a sequence of arc slides, we know abstractly from Auroux's work that $\A(\Zc^{\sta}_n)$ and $\A(\Zc^{\can}_n)$ are derived equivalent. Explicitly, there are certain Heegaard diagrams (one is shown in Figure~\ref{fig:COBdiagram}) in which holomorphic curve counts should yield $A_{\infty}$ bimodules realizing this derived equivalence. For $n=2$, these bimodules were defined and shown to give equivalences in \cite[Section 11]{ManionTrivalent}. At the decategorified level (for $n=2$ and presumably for general $n$) the bimodules do recover the change of basis between the canonical and standard bases; this can be determined from the gradings of bimodule generators, without needing to understand the holomorphic geometry. Topologically, the diagram $\cal{H}$ of Figure~\ref{fig:COBdiagram} represents an identity ``bordered sutured cobordism'' with input sutured surface parametrized by $\Zc^{\sta}_n$ and output sutured surface parametrized by $\Zc^{\can}_n$; see Figure~\ref{fig:COBcobordism} for an illustration.

\begin{figure}
\includegraphics[scale=0.6]{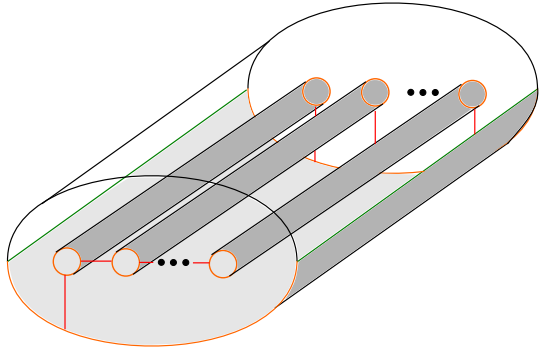}
\caption{Bordered sutured cobordism represented by the Heegaard diagram of Figure~\ref{fig:COBdiagram}.}
\label{fig:COBcobordism}
\end{figure}

Due to the presence of circles in our arc diagrams, the well-definedness of these Heegaard Floer bimodules for general $n$ is still conjectural, but there are good indicators in the literature as to how they should work, so that one can still use them as a useful source of motivation.

\begin{remark}
The reason well-definedness of the bimodules is conjectural is that nobody has attempted to define bimodules holomorphically for any general family of Heegaard diagrams that include the ones we're looking at. To do so, one would have to bring in ideas that go beyond the basic papers of Lipshitz--Ozsv{\'a}th--Thurston and Zarev; this is evident from recent work \cite{OSzHolo} of Ozsv{\'a}th--Szab{\'o} that can be viewed as establishing the new holomorphic-geometry ideas needed to treat one specific Heegaard diagram $\cal{H}_{\textrm{OSz}}$, where $\cal{H}_{\textrm{OSz}}$ also goes beyond Lipshitz--Ozsv{\'a}th--Thurston and Zarev's papers and shares many features with the diagrams we're studying here. We are not sure whether the ideas of \cite{OSzHolo} will end up being sufficient to holomorphically define bimodules from our diagrams; even if so, computations from \cite{ManionTrivalent} indicate that already for $n=2$ our diagrams are far from ``nice'' in the technical sense that would make their bimodules easy to compute, and instead one sees higher-weight domains for holomorphic curves resulting in complicated $A_{\infty}$ actions on the bimodules.
\end{remark}

\begin{figure}
\includegraphics[scale=0.6]{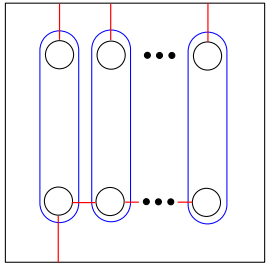}
\caption{Heegaard diagram for the change of basis from the standard basis to the canonical basis}
\label{fig:COBdiagram}
\end{figure}

\subsection{Standard-basis strands algebra and standard modules}

Since $\A(\Zc^{\sta}_n)$ and $\A(\Zc^{\can}_n)$ can both be used to categorify the $U_q(\gl(1|1))$ representation $V^{\otimes n}$ (with half of the quantum group action; see \cite{ManionDecat,ManionTrivalent}), with indecomposable projectives over $\A(\Zc^{\sta}_n)$ and $\A(\Zc^{\can}_n)$ categorifying the standard tensor-product basis and the canonical basis for $V^{\otimes n}$ respectively, we can view the desired explicit derived-equivalence bimodules in Section~\ref{sec:HDDerivedEq} as giving a ``categorified change of basis'' between the standard and canonical bases.

Since standard modules over the Ozsv{\'a}th--Szab{\'o} algebras also categorify the standard basis for $V^{\otimes n}$, it is natural to look for a relationship between standard modules and $\A(\Zc^{\sta}_n)$; indeed, we have already shown in Theorem~\ref{thm:ExtStrands} that the homology of $\A(\Zc^{\sta}_n)$ describes Ext groups between standard modules (disregarding multiplication).

It turns out that the change-of-basis Heegaard diagrams, with their conjectural $A_{\infty}$ bimodules from counting holomorphic curves, suggest a nontrivial way to relate multiplication on Ext groups between standards and multiplication in $\A(\Zc^{\sta}_n)$. Specifically, we show in Proposition ~\ref{prop:FSameAsDelta}
that the conjectured right $A_{\infty}$ action on the DA bimodule of the diagram $\cal{H}$ in Figure~\ref{fig:COBdiagram} 
is equivalent to the data of an $A_{\infty}$ homomorphism from $\A(\Zc^{\sta}_n)$ to the (opposite) dg endomorphism algebra of the projective resolutions of standard modules computed in Section~\ref{sec:BSide}, inducing our above isomorphism on homology. This observation reduces the algebraic problem of relating multiplicative structures on Ext groups and the strands algebra to the Heegaard Floer problem of defining an $A_{\infty}$ action based on (actual or heuristic) holomorphic disk counts.

The equivalence between the homomorphism and the right $A_{\infty}$ action follows from the appearance of our projective resolutions of standard modules in the ``easier-to-manage'' (and non-conjectural) aspects of the Heegaard Floer bimodule. After discussing these projective resolutions in Heegaard diagram terms in Sections \ref{sec:HD}--\ref{sec:HDProjResDiff} below, we briefly outline what we know about the right $A_{\infty}$ action and how it should arise in general in Sections \ref{sec:Strict}--\ref{sec:Ainfty}.

\subsection{Heegaard diagrams and DA bimodules}\label{sec:HD}
Let $n$ be the number of circles on the top and bottom of the Heegaard diagram $\cal{H}$ in Figure~\ref{fig:COBdiagram}. There are incoming (top) and outgoing (bottom) arc diagrams associated to $\cal{H}$, based on the matching pattern of red arcs in $\cal{H}$. The incoming arc diagram of $\cal{H}$ is $\Zc^{\sta}_n$, and the outgoing arc diagram is $\Zc^{\can}_n$.

Conjecturally, there is a type of $A_{\infty}$ bimodule known as a DA bimodule (see \cite{LOTBimodules}) associated to $\cal{H}$. For our purposes we can use the following definition.

\begin{definition}\label{def:DABimod}
A DA bimodule over dg algebras $(A,B)$ is an $A_{\infty}$ bimodule $X$ over $(A,B)$ such that the left action of $A$ has no higher $A_{\infty}$ terms and such that, as a non-differential left $A$-module, $X$ is finitely generated and projective. The right action of $B$ is allowed to have higher $A_{\infty}$ terms.
\end{definition}

Ozsv{\'a}th--Szab{\'o}'s disk-counting heuristics for Heegaard diagrams with input and output arc diagram $\Zc^{\can}_n$ actually produce DA bimodules over the homology algebra $H(\A(\Zc^{\can}_n)) \cong \oplus_k \B_l(n,k)$; correspondingly, we want to use $\cal{H}$ to produce a DA bimodule over $(H(\A(\Zc^{\can}_n)), \A(\Zc^{\sta}_n))$.

The structure of the DA bimodule associated to $\cal{H}$ as a non-differential left module, along with the right action on this module of the basic idempotents of $\A(\Zc^{\sta}_n)$, can be defined straightforwardly without needing any conjectural holomorphic curve counts.

\begin{definition}\label{def:XLeftNonDiff}
Consider the set $\cal{G}$ of sets of red-blue intersection points in $\cal{H}$ such that each red or blue circle has exactly one point chosen (although there are no red circles in $\cal{H}$), and such that each red arc has either zero or one points chosen.

For each element $\gamma$ of $\cal{G}$, we associate a ``left idempotent'' in $\A(\Zc^{\can}_n)$ (or its homology) and a ``right idempotent'' in $\A(\Zc^{\sta}_n)$. For the right idempotent of $\gamma$, one labels the arc as occupied if and only if it is occupied in $\gamma$ (then $S \subset \{1,\ldots,n\}$ is defined to be the set of occupied arcs, which are in natural bijection with $\{1,\ldots,n\}$, and the idempotent is $I_S$). For the left idempotent the convention is reversed, so that one labels the arc as occupied if and only if it is unoccupied in $\gamma$ (then $\x \subset \{0,\ldots,n-1\}$ is defined to be the set of occupied arcs, which are in natural bijection with $\{0,\ldots,n-1\}$, and the idempotent is $\Ib_{\x}$).

As a non-differential left module over $H(\A(\Zc^{\can}_n))$, the DA bimodule of $\cal{H}$ is defined as a direct sum of indecomposable projectives, one for each element of $\cal{G}$; the projective associated to an element $\gamma$ of $\cal{G}$ is specified by the left idempotent of $\gamma$. The right action of a basic idempotent $I_S$ of $\A(\Zc^{\sta}_n)$ on the projective summand coming from $\gamma$ is the identity if $I_S$ is the right idempotent of $\gamma$ and is zero otherwise.
\end{definition}

\subsection{Projective resolutions and Heegaard diagrams: no differentials}\label{sec:HDProjResNonDiff}

If we let $X$ denote the conjectural DA bimodule of $\cal{H}$ whose left module structure was specified above, then the functor on derived categories induced by $X$ is $X \boxtimes -$ where $\boxtimes$ is a pairing operation on DA bimodules related to the derived tensor product (see \cite{LOTBimodules}). If $P_S$ is the indecomposable projective corresponding to some basic idempotent $I_S$ of $\A(\Zc^{\st}_n)$, then $X \boxtimes P_S$ can be computed as a non-differential module over $H(\A(\Zc^{\can}_n))$ using only what we already know about $X$. The result is the left $H(\A(\Zc^{\can}_n))$-submodule of $X$ consisting of the projective summands for those $\gamma \in \cal{G}$ with right idempotent $I_S$.

\begin{lemma}\label{lem:ProjResNonDiff}
Let $S \subset \{1,\ldots,n\}$; let $X_S$ be the submodule of the left module over $H(\A(\Zc^{\can}_n))$ from Definition~\ref{def:XLeftNonDiff} consisting of the projective summands for those $\gamma \in \cal{G}$ with right idempotent $I_S$. Identify $H(\A(\Zc^{\can}_n)$ with $\oplus_k \B_l(n,k)$ and thus with $\oplus_k \Bt(\cal{V}_k)$ where $\cal{V}_k$ is left cyclic for $0 \leq k \leq n$. There is a natural identification of $X_S$ with the projective resolution of the standard module $\Vt_{\alpha}$ from Section~\ref{sec:BSide} (with coefficients taken in $\F_2$), where $\alpha \in \{+,-\}^n$ corresponds to $S \subset \{1,\ldots,n\}$ under the usual identification between $\cal{P}$ and $\mathbb{B}$ from Section~\ref{subsec:partial_order}.
\end{lemma}

\begin{proof}
Write $S = \{i_1,\ldots,i_k\}$. The set $\{1,\ldots,n\} \setminus \{i_1,\ldots,i_k\}$ decomposes into finitely many sequences $j,j+1,\ldots,j+\ell-1$ of consecutive elements. The indices $j,\ldots,j+\ell-1$ correspond to blue circles in the Heegaard diagram $\cal{H}$ of Figure~\ref{fig:COBdiagram}. To specify an indecomposable projective summand of $X_S$, one must specify an intersection point on each of these blue circles such that no red arc gets more than one intersection point. There are $\ell$ of the blue circles and (for $j + \ell - 1 < n$) they are incident to $\ell + 1$ red arcs on the bottom half of $\cal{H}$; one chooses one of these $\ell+1$ arcs to be unoccupied and this determines the intersection points of these $\ell$ blue circles (then one does this for each sequence). For $j + \ell - 1 = n$ there are only $\ell$ red arcs incident to these $\ell$ blue circles on the bottom side of $\cal{H}$ and there is only one valid choice of intersection points for the blue circles in this case.

Say that $\alpha$ corresponds to $\x \in V_l(n,k)$ as in Definition~\ref{def:StatesFromAlpha} (by Corollary~\ref{cor:BijectionsAgree}, $\x$ is obtained from $S$ by subtracting $1$ from each of its elements). We get a decomposition of $X_S$ as a direct sum of indecomposable projective modules (disregarding the differential):
\[
X_S \cong \bigoplus_{\y} P_{\y}
\]
where the sum runs over those $\y \in V_l(n,k)$ obtained from $\x$ by moving each dot some number $\geq 0$ of steps to the left without ever running into the starting point of another dot, and $P_{\y}$ is the indecomposable projective $\B_l(n,k)$-module corresponding to $\y$. Specifically, the summand of $X_S$ corresponding to $\gamma \in \cal{G}$ gets identified with $P_{\y}$ where $\y \in V_l(n,k)$ can be viewed as the set of red arcs on the bottom half of the Heegaard diagram $\cal{H}$ that are unoccupied in $\gamma$. By Lemma~\ref{lem:ConesForCyclic}, we have
\[
X_S \cong \bigoplus_{\beta} \tilde{P}_{\beta}
\]
where the sum runs over $\beta$ in the bounded cone $\B_{\mathbbm{x}_{\alpha}} = \B_S$ and $\tilde{P}_{\beta}$ is the indecomposable projective $\Bt(\cal{V})$-module corresponding to $\beta$. By definition, these are the terms of the projective resolution of $\Vt_{\alpha}$.
\end{proof}

\begin{remark}
In Lemma~\ref{lem:ProjResNonDiff} and the other results related to Heegaard diagrams below, we have ignored gradings. One could define a bigrading on the module $X$ of Definition~\ref{def:XLeftNonDiff} such that Lemma~\ref{lem:ProjResNonDiff} gives an isomorphism of bigraded modules (the projective resolution has one grading from the algebra and another from degree in the resolution). The extra structure on $X$ coming from holomorphic disk counts should respect this bigrading; for example, the bimodules for $n=2$ in \cite{ManionTrivalent} are bigraded. Since gradings on Heegaard diagram bimodules are a complicated subject in general, we will not say more about them here.
\end{remark}

\subsection{Projective resolutions and Heegaard diagrams: differentials}\label{sec:HDProjResDiff}

Unlike the structure on $X_S$ we have considered so far, the differential on $X_S$ is based on holomorphic disk counts.

For $\gamma \in \cal{G}$, write $x_{\gamma}$ for the generator of the corresponding projective summand of $X_S$ as a left module. For any pair $(\gamma,\gamma')$ where $\gamma, \gamma' \in \cal{G}$ have right idempotent $I_S$, one can consider ``domains from $\gamma$ to $\gamma'$,'' which can be viewed as certain $\Z$-linear combinations of the connected components of the complement of the red and blue curves in the Heegaard diagram. Various restrictions on the domains must hold; one is that the domain must have zero multiplicity on the two regions adjacent to the left and right vertical boundaries of the Heegaard diagram. To contribute to the differential on $X_S$, the domain must also have zero multiplicity on any region adjacent to the top boundary or any of the top circle boundary components of the diagram.

As a result, one can check that for the differential, one only need consider domains from $\gamma$ to $\gamma'$ when $\gamma'$ is obtained from $\gamma$ by moving an intersection point one step to the left along a blue circle, like how the hollow-dot generator $\gamma'$ is obtained from the solid-dot generator $\gamma$ in Figure~\ref{fig:COBdifferential}. Furthermore, the one domains one need consider look like the one shown in Figure~\ref{fig:COBdifferential}. These domains are bigons, and their contribution to the differential will be $d(a x_{\gamma}) = a R_i x_{\gamma'} + \cdots$ in any reasonable holomorphic definition of the DA bimodule $X$, where $a \in \B_l(n,k)$ and $R_i$ is a generator of $\B_l(n,k)$ from Definition~\ref{def:SmallStep} (the left and right idempotents of this $R_i$ generator are $\Ib_{\y}$, $\Ib_{\y'}$ respectively where $\gamma$, $\gamma'$ correspond to $\y, \y' \in V_l(n,k)$ as in the proof of Lemma~\ref{lem:ProjResNonDiff}).

\begin{figure}
\includegraphics[scale=0.85]{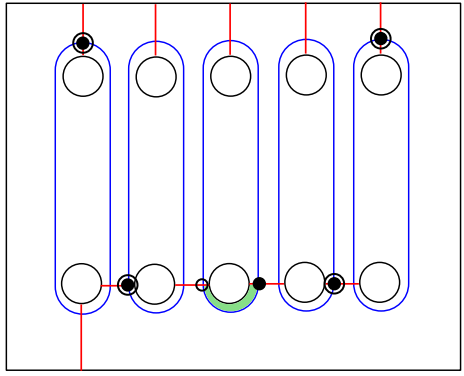}
\caption{A domain from the solid-dot generator $\gamma$ to the hollow-dot generator $\gamma'$.}
\label{fig:COBdifferential}
\end{figure}

\begin{theorem}\label{thm:ProjResDiff}
If we give $X_S$ the differential arising from the above-described bigon maps, then the identification in Lemma~\ref{lem:ProjResNonDiff} of $X_S$ with the projective resolution of $\Vt_{\alpha}$ preserves differentials.
\end{theorem}

\begin{proof}
Let $\gamma, \gamma'$ be two elements of $\cal{G}$ (with right idempotent $I_S)$ admitting a bigon domain between them as described above. By the proof of Lemma~\ref{lem:ProjResNonDiff}, we can identify $\gamma, \gamma'$ with elements $\y, \y' \in V_l(n,k)$, each of which is obtained from $\x$ by moving each dot some number $\geq 0$ of steps to the left without ever running into the starting point of another dot (i.e. with elements of the bounded cone $\B_{\mathbbm{x}_{\alpha}} = \B_S$). Explicitly, $\y, \y'$ are the sets of unoccupied arcs appearing along the bottom of Figure~\ref{fig:COBdifferential} in $\gamma,\gamma'$. Figure~\ref{fig:COBdifferential} shows an unoccupied arc moving one step to the right when passing from $\gamma$ to $\gamma'$; equivalently, $\y$ is obtained from $\y'$ by moving one dot one step to the left.

To compute the differential in the projective resolution of $\Vt_{\alpha}$, we let $\beta, \beta' \in \cal{P}$ correspond to $\y, \y' \in V_l(n,k)$. We then have $\beta = \alpha^T$ and $\beta' = \alpha^{T'}$ for $T, T' \subset \xx_{\alpha}^c = S^c$, where $\alpha^T, \alpha^{T'}$ are defined as in the proof of Theorem~\ref{thm:proj-resolution} ($\alpha^T$ is the sign vector that differs from $\alpha$ in exactly the indices in $T$). Since $\y$ is obtained from $\y'$ by moving one dot one step to the left, we have $T = T' \sqcup \{i\}$ for some $i$. The differential on $\Vt_{\alpha}$ then has a term mapping $\tilde{P}_{\alpha^T}$ to $\tilde{P}_{\alpha^{T'}}$ by $\varphi_{T,i}$, where $\varphi_{T,i}$ is right multiplication by the element $p(\alpha^T, \alpha^{T'}) \in \Bt(\cal{V})$. Under the isomorphism $\Bt(\cal{V}) \cong \B_l(n,k)$ of Theorem~\ref{thm:CyclicIsom} (see also \cite[Theorem 4.9]{LLM}), this element $p(\alpha^T, \alpha^{T'})$ corresponds to the $R_i$ generator appearing in the bigon map described above. Conversely, any term in the differential on $\Vt_{\alpha}$ arises in this way.

\end{proof}

\subsection{A conjectured \texorpdfstring{$A_{\infty}$}{A-infinity} morphism}\label{sec:Ainfty}

\subsubsection{Right $A_{\infty}$ actions and $A_{\infty}$ homomorphisms}\label{sec:AInftyHoms}

Conjecturally, holomorphic curve counts may be used to give the differential module $X$ above the structure of a right $A_{\infty}$ module over $\A(\Zc^{\sta}_n)$. Explicitly, this structure can be viewed as a set of left $H(\A(\Zc^{\can}_n))$-module homomorphisms
\[
\delta_i\colon X \otimes \left( \A(\Zc^{\sta}_n) \right)^{\otimes (i-1)} \to X
\]
for $i \geq 2$, where the tensor products are taken over the idempotent ring of $\A(\Zc^{\sta}_n)$, such that the $A_{\infty}$ relations
\begin{align*}
&d \circ \delta_i \\
&+ \delta_i \circ (d \otimes \id^{\otimes(i-1)}) \\
&+ \sum_{i_1 + i_2 = i-2; \,\, i_i, i_2 \geq 0} \delta_i \circ \left(\id_X \otimes \id^{\otimes(i_1)} \otimes d \otimes \id^{\otimes(i_2)}\right) \\
&+ \sum_{i_1 + i_2 = i-3; \,\,i_1, i_2 \geq 0} \delta_{i-1} \circ \left(\id_X \otimes \id^{\otimes(i_1)} \otimes \mathrm{mult} \otimes \id^{\otimes(i_2)} \right) \\
&+ \sum_{i_1 + i_2 = i+1; \,\, i_1, i_2 \geq 2} \delta_{i_2} \circ \left(\delta_{i_1} \otimes \id^{\otimes(i_2 - 1)} \right) \\
&= 0
\end{align*}
hold for all $i \geq 2$. These relations are depicted in Figure~\ref{fig:DeltaRels} (read top to bottom); the dots represent differentials.

\begin{figure}
\includegraphics[scale=0.6]{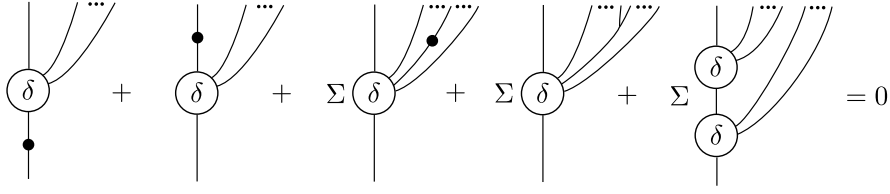}
\caption{$\A_{\infty}$ relations for the maps $\delta_i$.}
\label{fig:DeltaRels}
\end{figure}

Relatedly, recall that for differential algebras $A,B$ over $\F_2$ (ignoring gradings), an $A_{\infty}$ homomorphism $f\colon A \to B$ consists of a set of linear maps
\[
f_i\colon A^{\otimes i} \to B
\]
for $i \geq 1$, such that the $A_{\infty}$ relations
\begin{align*}
&d \circ f_i \\
&+ \sum_{i_1 + i_2 = i-1; \,\, i_1, i_2 \geq 0} f_i \circ \left(\id^{\otimes(i_1)} \otimes d \otimes \id^{\otimes(i_2)}\right) \\
&+ \sum_{i_1 + i_2 = i-2; \,\, i_1, i_2 \geq 0} f_{i-1} \circ \left(\id^{\otimes(i_1)} \otimes \mathrm{mult} \otimes \id^{\otimes(i_2)}\right) \\
&+ \sum_{i_1 + i_2 = i; \,\, i_1, i_2 \geq 1} \mathrm{mult} \circ \left(f_{i_1} \otimes f_{i_2}\right) \\
&= 0
\end{align*}
hold for all $i \geq 1$. These relations are depicted in Figure~\ref{fig:FRels}.

\begin{figure}
\includegraphics[scale=0.6]{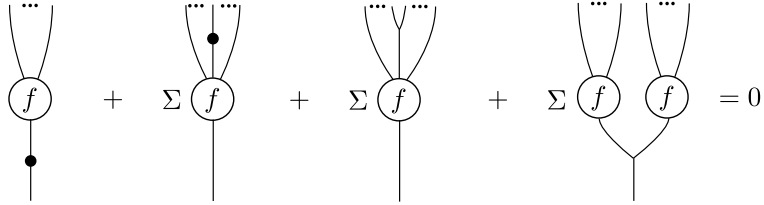}
\caption{Relations for an $\A_{\infty}$ homomorphism between differential algebras.}
\label{fig:FRels}
\end{figure}

If $A$ and $B$ are defined over the same ring of idempotents, we will take the tensor products over this idempotent ring in the above definition. For example, if we let $E$ denote the endomorphism differential algebra of $X$ as a differential left $H(\A(\Zc^{\can}_n))$-module, then for each $S \subset \{1,\ldots,n\}$ (giving a basic idempotent $I_S$ of $\A(\Zc^{\sta}_n)$), there is an idempotent $\phi_S \in E$ acting as the identity on $X_S \subset X$ and acting as zero on all other summands. We use these idempotents to view $E$ as an algebra over the idempotent ring of $\A(\Zc^{\sta}_n)$ (a finite product of copies of $\F_2$).

\begin{proposition}\label{prop:FSameAsDelta}
The above maps $\delta_i$ for $X$, satisfying the $A_{\infty}$ relations, are the same data as an $A_{\infty}$ homomorphism from the differential algebra $\A(\Zc^{\sta}_n)$ to $E^{\op}$, where $E$ is the differential endomorphism algebra of $X$ as a left differential $H(\A(\Zc^{\can}_n))$-module.
\end{proposition}

\begin{proof}

To define an $A_{\infty}$ homomorphism from $\A(\Zc^{\sta}_n)$ to $E^{\op}$ given $\delta_i$ for $i \geq 2$, suppose that $a_1,\ldots,a_i$ are elements of $\A(\Zc^{\st}_n)$ for $i \geq 1$. Let $f_i(a_1 \otimes \ldots \otimes a_i)$ be the element of $E^{\op}$, i.e. endomorphism of $X$, sending
\[
ax_{\gamma} \mapsto \delta_{i+1}(ax_{\gamma} \otimes a_1 \otimes \cdots \otimes a_i)
\]
where $x_{\gamma}$ is a generator of $X$. Conversely, given $f_i$ for $i \geq 1$, define $\delta_i$ for $i \geq 2$ by
\[
\delta_i(ax_\gamma \otimes a_1 \otimes \cdots \otimes a_{i-1}) = f_{i-1}(a_1 \otimes \cdots \otimes a_{i-1})(ax_{\gamma}).
\]
Visually, we define $f$ from $\delta$ and vice-versa so that the equality in Figure~\ref{fig:FDeltaEquality} holds. One can then check that the $A_{\infty}$ relations for $\delta$ are equivalent to the $A_{\infty}$ relations for $f$, using that the right action of $E^{\op}$ on $X$ is an action of a differential algebra on an ordinary differential module (summarized in Figure~\ref{fig:XERels}).

\end{proof}

\begin{figure}
\includegraphics[scale=0.6]{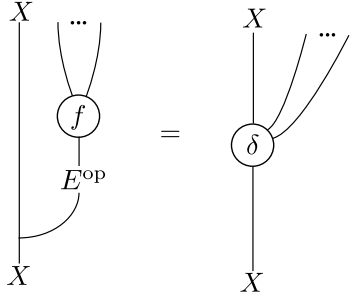}
\caption{Relationship between $f$ and $\delta$ in Proposition~\ref{prop:FSameAsDelta}.}
\label{fig:FDeltaEquality}
\end{figure}

\begin{figure}
\includegraphics[scale=0.6]{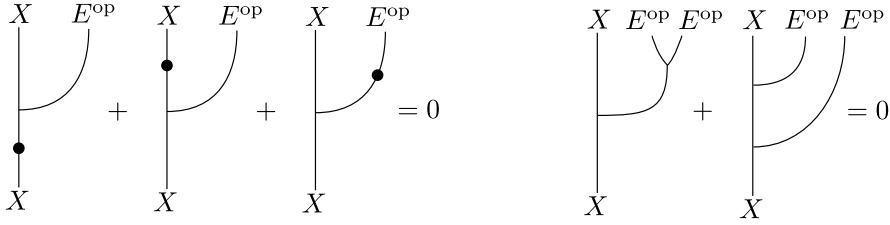}
\caption{Relations for the right action of $E^{\op}$ on $X$.}
\label{fig:XERels}
\end{figure}

\subsubsection{Statement of the conjecture}\label{sec:AinftyConjecture}

By Theorem~\ref{thm:ProjResDiff}, the differential algebra $E$ of Proposition~\ref{prop:FSameAsDelta} is the endomorphism algebra of the direct sum of projective resolutions of all standard modules $\Vt_{\alpha}$ over $\Bt(\cal{V}_k)$ for left cyclic $\cal{V}_k$, $0 \leq k \leq n$. Thus, the homology of $E$ is the direct sum of Ext groups between these standard modules, and the differential algebra structure of $E$ gives rise to the $A_{\infty}$ multiplication on this sum of Ext groups.

We know by Theorem~\ref{thm:ExtStrands} that the homology of $E$ is isomorphic to the homology of $\A(\Zc^{\sta}_n)$, and Proposition~\ref{prop:FSameAsDelta} says that we can get an $A_{\infty}$ homomorphism from $\A(\Zc^{\sta}_n)$ to $E^{\op}$ given the existence of a certain bimodule structure in Heegaard Floer homology.

\begin{conjecture}\label{conj:ExtMultStrands}
Let $X$ be the left differential $H(\A(\Zc^{\can}_n))$-module $X$ from Definition~\ref{def:XLeftNonDiff} with differential from bigon maps as in Section~\ref{sec:HDProjResDiff}, and let $E$ be the endomorphism algebra of $X$. Then $X$ admits the structure of a DA bimodule over $(H(\A(\Zc^{\can}_n)),\A(\Zc^{\sta}_n))$ in such a way that the corresponding $A_{\infty}$ homomorphism $\A(\Zc^{\sta}_n) \to E^{\op}$ from Proposition~\ref{prop:FSameAsDelta} induces the isomorphism $H(\A(\Zc^{\sta}_n)) \to H(E^{\op}) = H(E)$ from Theorem~\ref{thm:ExtStrands} (identifying $X$ with the sum of projective resolutions of standard modules as in Theorem~\ref{thm:ProjResDiff}).
\end{conjecture}

\begin{remark}
A satisfying conceptual approach to Conjecture~\ref{conj:ExtMultStrands} would be to define DA bimodules holomorphically for a general family of Heegaard diagrams including $\cal{H}$ from Figure~\ref{fig:COBdiagram}, then compute enough of the structure of the resulting DA bimodule $X$ to show that the isomorphism $H(\A(\Zc^{\sta}_n)) \to H(E^{\op}) = H(E)$ from Theorem~\ref{thm:ExtStrands} is recovered. Even better for e.g. relating different tangle-based approaches to knot Floer homology would be to compute all the higher actions on $X$; alternately, given that defining bimodules holomorphically can be difficult, one could try to define the higher actions on $X$ by hand (using holomorphic geometry as motivation and looking for patterns).
\end{remark}

Given Conjecture~\ref{conj:ExtMultStrands}, to understand the $A_{\infty}$ multiplication on the direct sum of Ext groups between standard modules, it suffices to understand the $A_{\infty}$ multiplication on $H(\A(\Zc^{\sta}_n))$. This can be done as follows, although we will not write up the details here: by making various choices, one can define an endomorphism $T$ of $\A(\Zc^{\sta}_n)$ such that $dT + Td = \id$, and using $T$ one can obtain a model for the higher $A_{\infty}$ actions on $H(\A(\Zc^{\sta}_n))$. Higher $A_{\infty}$ actions on the homology of bordered strands algebras have also been studied elsewhere in the literature, e.g. in \cite[Section 4.2]{LOTBimodules}.

\subsubsection{The strict part of an $A_{\infty}$ morphism}\label{sec:Strict}

If $f = (f_i)_{i=1}^{\infty}$ is an $A_{\infty}$ homomorphism of differential algebras from $A$ to $B$, the map induced by $f$ on homology depends only on the ``strict part'' $f_1$ of $f$ which is a chain map from $A$ to $B$. Suppose we have DA bimodule operations $(\delta_i)_{i=2}^{\infty}$ on $X$ as in Conjecture~\ref{conj:ExtMultStrands}; by Proposition~\ref{prop:FSameAsDelta}, to check that the corresponding $A_{\infty}$ homomorphism induces the isomorphism on homology from Theorem~\ref{thm:ExtStrands}, it suffices to understand the operation $\delta_2$ on $X$.

As with the differential on $X$ (often referred to as $\delta_1$), it is easier to understand $\delta_2$ in terms of holomorphic curves than it is to understand $\delta_i$ for general $i$. Using the holomorphic curves as motivation, below we define a chain map from $\A(\Zc^{\sta}_n)$ to the endomorphism algebra $E$ of $X$ (not, however, compatible with multiplication) and check that it induces the isomorphism on homology from Theorem~\ref{thm:ExtStrands}. Afterwards we briefly remark on the domains of the relevant holomorphic curves.

\begin{figure}
\includegraphics[scale=0.6]{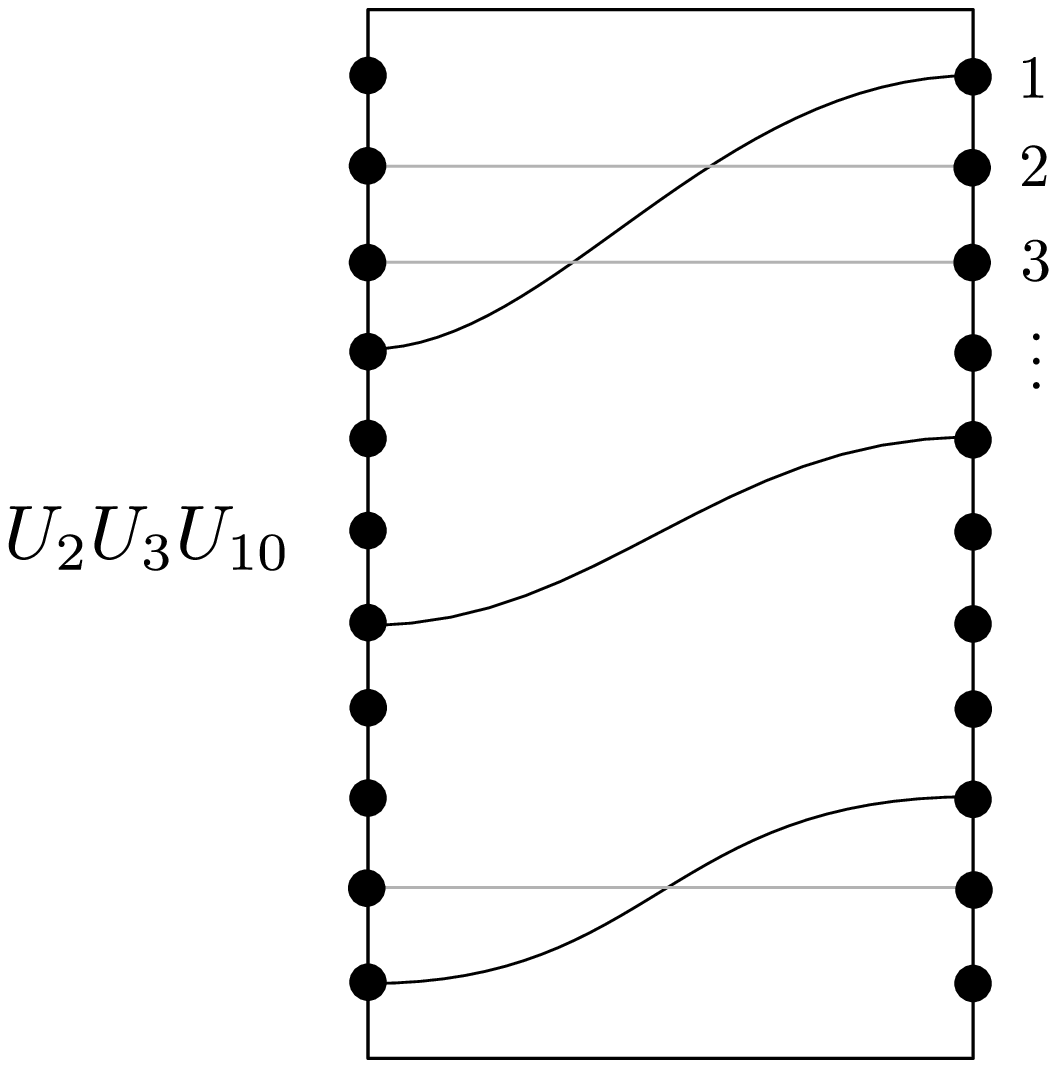}
\caption{Basis element of $\A(\Zc^{\sta}_n)$ on which $f_1$ is nonzero; we have $n = 11$, $k = 6$, $S = \{2,3,4,7,10,11\}$, $T = \{1,2,3,5,9,10\}$, $\mathrm{int}(\supp(h)) = \{2,3,6,10\}$, and $\{i_1,i_2,i_3\} = \{2,3,10\}$.}
\label{fig:StrandsHomologyGen}
\end{figure}

\begin{definition}\label{def:F1}
Define a linear map of $\F_2$-vector spaces
\[
f_1 \colon \A(\Zc^{\sta}_n) \to E
\]
as follows; let $S,T \subset \{1,\ldots,n\}$ with $|S| = |T| = k$.
\begin{itemize}
\item If $\cal{B}_S \cap \cal{F}_T = \emptyset$, define $f_1$ to be zero on $I_S \cdot \A(\Zc^{\sta}_n) \cdot I_T$.

\item If $\cal{B}_S \cap \cal{F}_T \neq \emptyset$, write $I_S \cdot \A \cdot I_T$ as
\[
\bigoplus_{C \subset \{i_1,\ldots,i_a\}} \left(I_{S \setminus C} \cdot \A(n-|C|,k-|C|) \cdot I_{T \setminus C} \right) \otimes \F[U_j | j \in C] \otimes \F[U_i: i \in S \cap T \setminus \{i_1,\ldots,i_a\}]
\]
as in Section~\ref{sec:ExtAndStrandsHomology}, where $S \cap T \cap \mathrm{int}(\supp(h)) = \{i_1,\ldots,i_a\}$. Define $f_1$ to be zero on all summands except the summand for $C = \{i_1, \ldots, i_a\}$.

\item On the summand for $C = \{i_1,\ldots,i_a\}$, the space $I_{S \setminus \{i_1,\ldots,i_a\}} \cdot \A(n-a,k-a) \cdot I_{T \setminus \{i_1,\ldots,i_a\}}$ (not just its homology) is one-dimensional by \cite[proof of Proposition 4.2]{LOTBimodules}. Let $f_1$ send $U_{i_1} \cdots U_{i_a}$ times the unique nonzero element $r$ of this space, shown in an example in Figure~\ref{fig:StrandsHomologyGen}, to the element $\varphi_{\a,\b}$ of $E$ from Definition~\ref{def:ChainMapForExt}, where $\alpha,\beta \in \cal{P}$ correspond to $S,T \in \mathbb{B}$ (we are using Lemma~\ref{lem:ProjResNonDiff}).

\item More generally, if $\mu$ is any monomial in the variables $U_i$ for $i \in S \cap T$ such that $\mu$ is divisible by $U_{i_1} \cdots U_{i_a}$, let
\[
f_1(\mu r) = \frac{\mu}{U_{i_1}\cdots U_{i_a}} f_1(U_{i_1} \cdots U_{i_a} r) = \frac{\mu}{U_{i_1}\cdots U_{i_a}} \varphi_{\a,\b}.
\]
Define $f_1(\mu r) = 0$ when $\mu$ is not divisible by $U_{i_1} \cdots U_{i_a}$.

\end{itemize}

\end{definition}

\begin{theorem}
The map $f_1 \colon \A(\Zc^{\sta}_n) \to E$ from Definition~\ref{def:F1} commutes with the differentials on $\A(\Zc^{\sta}_n)$ and $E$. Its induced map from $H(\A(\Zc^{\sta}_n))$ to $H(E) = \oplus_{\a,\b} \Ext_{\Bt}(\Vt_{\a},\Vt_{\b})$ is the isomorphism from Theorem~\ref{thm:ExtStrands}.

\end{theorem}

\begin{proof}
To see that the map intertwines the differentials, note that none of the basis elements $\mu r$ of $\A(\Zc^{\sta}_n)$ on which $f_1$ is nonzero can arise as a term in the differential of any other basis element of $\A(\Zc^{\sta}_n)$ (for example, the strands element shown in Figure~\ref{fig:StrandsHomologyGen} cannot arise as a term in the differential of another element), so $f_1 \circ d = 0$. On the other hand, we have $d \circ f_1 = 0$ by Proposition~\ref{prop:VarphiChainMap}. The fact that $f_1$ induces the given isomorphism on homology then follows from Corollary~\ref{cor:VarphiInducesRightExtElts} and the fact that the elements $\mu r$ on which $f_1$ is nonzero are the basis elements for $H(\A(\Zc^{\sta}_n))$ from Section~\ref{sec:ExtAndStrandsHomology}.

\end{proof}

Proposition~\ref{prop:BFIntersectionForCyclic} implies that, for an element $\mu r$ of $\A(\Zc^{\sta}_n)$ with $f_1(\mu r) \neq 0$, the map $\varphi_{\a,\b}$ appearing in the definition of $f_1(\mu r)$ admits a slightly simpler description. Letting $p_{\alpha^S}^{\beta^{S'}}$ denote the image of $e_{\alpha^S} \in \Pt_{\alpha^S}$ under $\varphi_{\a,\b}$ as in Definition~\ref{def:ChainMapForExt}, we observe the following:
\begin{itemize}
\item For $i \in \xx_{\a} \cap \xx_{\b}$ with $\a(i) = \b(i)$, we have $i \notin S$ and $i \notin S'$, so $\alpha^S(i) = \beta^{S'}(i)$.
\item For $i \in \xx_{\a} \cap \xx_{\b}$ with $\a(i) \neq \b(i)$, the same reasoning gives $\alpha^S(i) \neq \beta^{S'}(i)$.
\item For $i \in \xx_{\a} \cap \xx_{\b}^c$, we also have $i \notin S$ and $i \notin S'$; it follows from Lemma~\ref{lem:BFIntersection} that $\alpha^S(i) = \beta^{S'}(i)$.
\item For $i \in \xx_{\a}^c \cap \xx_{\b}^c$ with $\a(i) = \b(i)$, by definition of $S$ and $S'$ we have $i \in S'$ if and only if $i \in S$, so $\alpha^S(i) = \beta^{S'}(i)$.
\item For $i \in \xx_{\a}^c \cap \xx_{\b}^c$ with $\a(i) \neq \b(i)$, we have $i \in S_{\min}$ which is contained in $S$ but is disjoint from $S'$, so $\alpha^S(i) = \beta^{S'}(i)$.
\item For $i \in \xx_{\a}^c \cap \xx_{\b}$, we have $i \in S$ and $i \notin S'$; by Proposition~\ref{prop:BFIntersectionForCyclic}, $\alpha^S(i) = \beta^{S'}(i)$.
\end{itemize}
Thus, $\alpha^S$ and $\beta^{S'}$ differ precisely in those indices $i \in \xx_{\a} \cap \xx_{\b}$ with $a(i) \neq b(i)$. Furthermore, defining $h$ to be the unique relative homology class with $\partial(h) = \xx_{\b} - \xx_{\a}$ as usual, for $i \in \xx_{\a} \cap \xx_{\b}$ we have $i \in \mathrm{int}(\supp(h))$ if and only if $a(i) \neq b(i)$. These indices are $i_1,\ldots,i_a$ as in Definition~\ref{def:F1}. It follows that in the language of $\B_l(n,k)$, we can write the element $p_{\alpha^S}^{\beta^{S'}}$ as $L_{i_1} \cdots L_{i_a}$.

\begin{figure}
\includegraphics[scale=0.6]{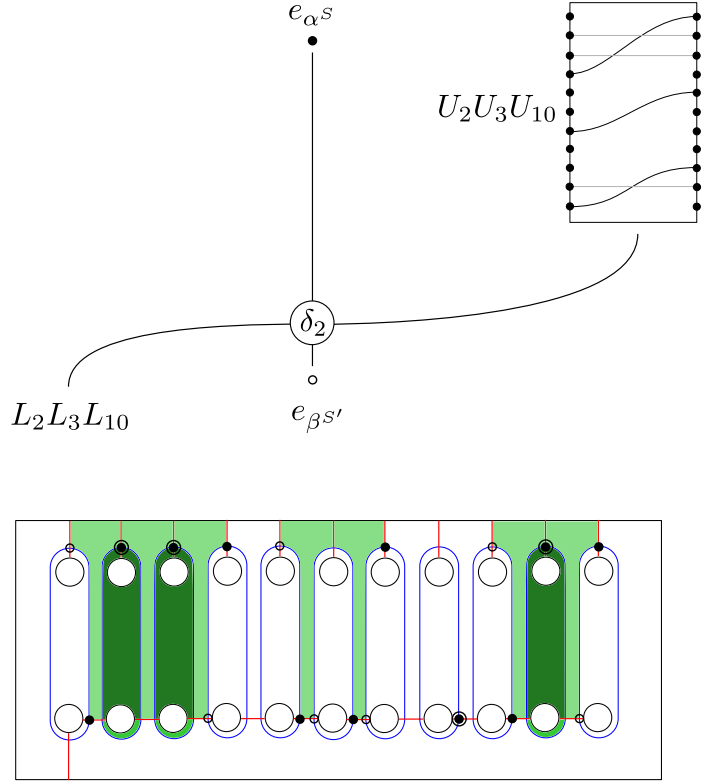}
\caption{Heegaard diagram domain for a term of $\delta_2$ corresponding to a term of $f_1$ as in Remark~\ref{rem:DomainForAction}.}
\label{fig:F1Domain}
\end{figure}

\begin{remark}\label{rem:DomainForAction}
If we reformulate the map $f_1$ in terms of $\delta_2\colon X \otimes \A(\Zc^{\sta}_n) \to X$ using Figure~\ref{fig:FDeltaEquality}, the result should admit an interpretation in terms of holomorphic curve counts. It is possible to identify the Heegaard diagram domains that would be relevant for such a description; Figure~\ref{fig:F1Domain} shows one example. In the notation of Definition~\ref{def:ChainMapForExt}, we have:
\begin{itemize}
\item $\a = (+-+---+++-+)$, $\b = (-+--++++-++)$
\item $\xx_{\a} = \{2,3,4,7,10,11\}$, $\xx_{\b} = \{1,2,3,5,9,10\}$
\item $S_{\min} := \{i \in \xx_{\a}^c \cap \xx_{\b}^c : \alpha(i) \neq \beta(i)\} = \{ 6 \}$
\item $\{i \in \xx_{\a}^c \cap \xx_{\b}^c : \alpha(i) = \beta(i) \} = \{8\}$
\item $S = S_{\min} \cup (\xx_{\a}^c \cap \xx_{\b}) \cup S' = \{6\} \cup \{1,5,9\} \cup \{8\} = \{1,5,6,8,9\}$, with $S' = \{8\}$
\item $\a^{S} = (--+-+++---+)$, $\b^{S'} = (-+--+++--++)$
\item $\xx_{\a^S} = \{1,3,4,5,8,11\}$, $\xx_{\b^{S'}} = \{1,2,3,5,8,10\}$
\item $p_{\a^S}^{\b^{S'}}$ is the algebra element labeled $L_2 L_3 L_{10}$.
\end{itemize}
Indeed, the patterns of occupied red arcs on the top of the diagram is $\xx_{\a}$ for the solid-dot generator, and it is $\xx_{\b}$ for the hollow-dot generator. The pattern of unoccupied red arcs on the bottom of the diagram (numbering the leftmost arc as $1$) is $\xx_{\a^S}$ for the solid-dot generator and $\xx_{\b^{S'}}$ for the hollow-dot generator. We would get a similar picture, differing only in the position of one solid dot and one hollow dot, if we took $S' = \emptyset$ rather than $S' = \{8\}$.

The shading of the domain in Figure~\ref{fig:FDeltaEquality} indicates that if we multiply the algebra input $U_2 U_3 U_{10} \cdot (\textrm{strands gen.})$ by $U_2^p U_3^q U_{10}^r$ before applying $f_1$, then the output is correspondingly multiplied by $U_2^p U_3^q U_{10}^r$. However, if we apply $f_1$ to $\mu \cdot (\textrm{strands gen.})$ where $\mu$ is not divisible by $U_2 U_3 U_{10}$, we get zero.

\end{remark}

\begin{figure}
\includegraphics[scale=0.6]{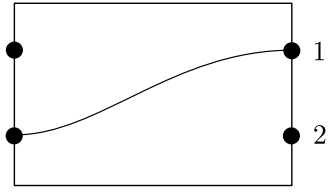}
\caption{Strands element $\lambda$ in the $n=2$ case.}
\label{fig:SimpleStrandsGen}
\end{figure}

\begin{remark}
When $n=1$, both $\A(\Zc^{\sta}_n)$ and $E^{\op}$ are $\F_2 \oplus \F_2[U_1]$, and one can check that $f_1 \colon \A(\Zc^{\sta}_n) \to E^{\op}$ is an isomorphism of algebras.

However, even for $n=2$, the map $f_1$ does not respect multiplication. For example, letting $\lambda$ be the element of $\A(\Zc^{\sta}_2)$ shown in Figure~\ref{fig:SimpleStrandsGen}, we have $\lambda U_1 = 0$ in $\A(\Zc^{\sta}_2)$. However, $f_1(\lambda)$ maps the generator $e_{-+}$ of the summand $\Pt_{-+}$ of the projective resolution of $\Vt_{+-}$ to the generator $e_{-+}$ of the summand $\Pt_{-+}$ of the projective resolution of $\Vt_{-+}$. In turn, $f_1(U_1)$ sends this generator $e_{-+}$ to the nonzero element $U_1 e_{-+}$ (recall that in $E^{\op}$, $f(\lambda) f(U_1)$ means ``do $f(\lambda)$, then $f(U_1)$'').

For $f$ to respect multiplication, we need to include a higher $A_{\infty}$ term; we can let $f_2(\lambda, U_1)$ be the element of $E$ sending the generator $e_{+-}$ of the summand $\Pt_{+-}$ of the projective resolution of $\Vt_{+-}$ to $L_1 e_{-+}$ where $e_{-+}$ is the generator of the summand $\Pt_{-+}$ of the projective resolution of $\Vt_{-+}$. More generally, the ``$\kappa$'' change-of-basis DA bimodule from \cite[Definition 11.2]{ManionTrivalent}, with differential and action terms $\delta_2$ and $\delta_3$, encodes (via Proposition~\ref{prop:FSameAsDelta}) an $A_{\infty}$ homomorphism from $\A(\Zc^{\sta}_n)$ to $E$, with strict part $f_1$, in the case $n=2$. Thus, this change-of-basis bimodule resolves Conjecture~\ref{conj:ExtMultStrands} affirmatively when $n=2$.

A ``backwards'' DA bimodule over $(\A(\Zc^{\sta}_n), H(\A(\Zc^{\can}_n)))$ for $n=2$ is also defined in \cite[Definition 11.3]{ManionTrivalent} (the ``$\lambda$'' change-of-basis bimodule), and the bimodules are checked to be inverse to each other up to homotopy equivalence. The $\lambda$ change-of-basis bimodule should be easier to define in general than the $\kappa$ bimodule we are trying to define here for general $n$; for $n=2$ the $\lambda$ bimodule has no higher $A_{\infty}$ actions, and it may be possible to arrange this in general.

More complicated computations, in the spirit of \cite{ManionTrivalent}, resolve Conjecture~\ref{conj:ExtMultStrands} affirmatively when $n=3$ and (more tenuously) when $n=4$. Even when $n=3$ the computations are quite involved, so we will omit them here. The higher $A_{\infty}$ terms appear to proliferate as $n$ increases; holomorphically, the main issue seems to be that domains such as the one shown in Figure~\ref{fig:F1Domain} can support a large number of relevant $A_{\infty}$ actions, with the $f_1$ (or $\delta_2$) contribution being just the tip of a large iceberg.

\end{remark}

\newcommand{\etalchar}[1]{$^{#1}$}

\bigskip

\end{document}